 \documentclass[11pt]{amsart}
\usepackage{amssymb}
\usepackage{amscd}
\usepackage{amsmath}
\usepackage[all]{xy}
\usepackage{braket}
\usepackage{mathrsfs}
\usepackage{setspace}
 \usepackage{color}

\theoremstyle{plain}
\newtheorem{theorem}{Theorem}[section]
\newtheorem{proposition}[theorem]{Proposition}
\newtheorem{lemma}[theorem]{Lemma}
\newtheorem{corollary}[theorem]{Corollary}

\theoremstyle{remark}
\newtheorem{remark}[theorem]{Remark}

\theoremstyle{definition}
\newtheorem{definition}[theorem]{Definition}

\newtheorem{hypothesis}[theorem]{Hypothesis}

  1

\DeclareMathOperator{\Ext}{Ext}
\DeclareMathOperator{\Gal}{Gal}

\DeclareMathOperator{\Hom}{Hom}

\DeclareMathOperator{\N}{N}

\DeclareMathOperator{\im}{im}
\DeclareMathOperator{\coker}{coker}

\DeclareMathOperator{\Fitt}{Fitt}

\newcommand{\bQ}{\mathbb{Q}}

\newcommand{\bZ}{\mathbb{Z}}

\newcommand{\cA}{\mathcal{A}}

\newcommand{\cD}{\mathcal{D}}

\newcommand{\cF}{\mathcal{F}}
\newcommand{\cG}{\mathcal{G}}
\newcommand{\cH}{\mathcal{H}}
\newcommand{\cI}{\mathcal{I}}

\newcommand{\cK}{\mathcal{K}}

\newcommand{\cN}{\mathcal{N}}
\newcommand{\cO}{\mathcal{O}}
\newcommand{\cP}{\mathcal{P}}
\newcommand{\cQ}{\mathcal{Q}}
\newcommand{\cR}{\mathcal{R}}

\newcommand{\cT}{\mathcal{T}}
\newcommand{\cX}{\mathcal{X}}

\newcommand{\fa}{\mathfrak{a}}
\newcommand{\fb}{\mathfrak{b}}

\newcommand{\fd}{\mathfrak{d}}

\newcommand{\fq}{\mathfrak{q}}
\newcommand{\fp}{\mathfrak{p}}
\newcommand{\fr}{\mathfrak{r}}
\newcommand{\fs}{\mathfrak{s}}
\newcommand{\fn}{\mathfrak{n}}
\newcommand{\fm}{\mathfrak{m}}

\newcommand{\CC}{\mathbb{C}}

\newcommand{\QQ}{\mathbb{Q}}

\newcommand{\RR}{\mathbb{R}}

\newcommand{\ZZ}{\mathbb{Z}}

\newcommand{\Ann}{\mathrm{Ann}}
\newcommand{\bz}{\mathbb{Z}}

\begin{document}

\title[]{On the theory of higher rank\\
Euler, Kolyvagin and Stark systems, II}

\author{David Burns, Ryotaro Sakamoto and Takamichi Sano}

\begin{abstract} We prove the existence of a canonical `higher Kolyvagin derivative' homomorphism between the modules of higher rank Euler systems and higher rank Kolyvagin systems, as has been conjectured to exist by Mazur and Rubin. This homomorphism exists in the setting of $p$-adic representations that are free with respect to the action of a Gorenstein order $\mathcal{R}$ and, in particular, implies that higher rank Euler systems control the $\mathcal{R}$-module structures of Selmer modules attached to the representation. We give a first application of this theory by considering the (conjectural) Euler system of Rubin-Stark elements.
\end{abstract}

\address{King's College London,
Department of Mathematics,
London WC2R 2LS,
U.K.}
\email{david.burns@kcl.ac.uk}

\address{Graduate School of Mathematical Sciences, The University of Tokyo,
3-8-1 Komaba, Meguro-Ku, Tokyo, 153-8914, Japan}
\email{sakamoto@ms.u-tokyo.ac.jp}

\address{Osaka City University,
Department of Mathematics,
3-3-138 Sugimoto\\Sumiyoshi-ku\\Osaka\\558-8585,
Japan}
\email{sano@sci.osaka-cu.ac.jp}

\maketitle

\tableofcontents

\section{Introduction}

\subsection{Background to the problem}
Ever since its introduction by Kolyvagin \cite{kolyvagin}, the theory of Euler systems has played a vital role in the proof of many celebrated results concerning the structure of Selmer groups of $p$-adic representations over number fields.

In an attempt to axiomatise, and extend, the use of Euler systems, Mazur and Rubin \cite{MRkoly} developed an associated theory of `Kolyvagin systems' and showed both that Kolyvagin systems controlled the structure of Selmer groups and that Kolyvagin's `derivative operator' gave rise to a canonical homomorphism between the modules of Euler and Kolyvagin systems that are associated to a given representation.

In this way, it became clear that Kolyvagin systems play the key role in obtaining structural results about Selmer groups and that the link to Euler systems is pivotal for the supply of Kolyvagin systems that are related to the special values of $L$-series.

For many representations, however, families of cohomology classes (such as Euler or Kolyvagin systems) are not themselves sufficient to control Selmer groups and in such `higher rank' cases authors have considered various collections of elements in higher exterior powers of cohomology groups.

The theory of `higher rank Euler systems' has in principle been well-understood for some time by now, with the first general approach being described by Perrin-Riou in \cite{PR} after significant contributions were made by Rubin in an important special case (related to Stark's Conjecture) in \cite{rubinstark}.

In addition, a general method was introduced in \cite[\S6]{rubinstark} whereby higher rank Euler systems could be used to construct, in a non-canonical way, classical (rank one) Euler or Kolyvagin systems to which standard techniques could then be applied.

However, whilst this `rank reduction' method has since been used both often and to great effect, notably by Perrin-Riou in \cite{PR} and by B\"uy\"ukboduk in \cite{Buyuk} and \cite{Buyuk2}, it is intrinsically non-canonical and also requires several auxiliary hypotheses (such as, for example, the validity of Leopoldt's Conjecture in the settings considered in \cite{rubinstark} and \cite{Buyuk}) that can themselves be very difficult to verify.

In an attempt to address these deficiencies, Mazur and Rubin \cite{MRselmer} have developed a theory of `higher rank Kolyvagin systems', and an associated notion of `higher rank Stark systems' (these are collections of cohomology classes generalizing the units predicted by Stark-type conjectures), and  showed that, under suitable hypotheses, such systems can be used to control Selmer groups.

However, the technical difficulties encountered when computing with higher exterior powers meant that the theory developed in \cite{MRselmer} was insufficient in the following respects.
\begin{itemize}
\item[$\bullet$] Coefficient rings were restricted to be either principal artinian local rings or discrete valuation rings, whilst dealing with more general coefficient rings is essential if one is to deal effectively with questions arising, for example, in either deformation theory or Galois module theory.
\item[$\bullet$] More importantly, whilst Mazur and Rubin conjectured the existence of a canonical link between the theories of higher rank Euler and Kolyvagin (or Stark) systems, they were unable to shed any light on the precise nature of this relationship (which they described as `mysterious').
\end{itemize}

\subsection{Overview of the solution}

The first of the above problems was resolved independently by the first and third authors in \cite{sbA} and by the second author in \cite{sakamoto}, a key part of the solution being the introduction of `exterior power biduals' as a functorially stronger version of exterior powers.

Building on these earlier articles, {\it we shall now resolve the second problem, and hence prove the conjecture of Mazur and Rubin, by constructing a canonical `higher Kolyvagin derivative' map between the modules of Euler and Kolyvagin systems of any given rank}.

We shall construct this map in the setting of $p$-adic representations that are free with respect to the action of an arbitrary Gorenstein order $\cR$, as one would expect to suffice for applications to deformation theory. In addition, by combining the construction with results from \cite{sbA}, we are able to deduce that, under natural hypotheses, higher rank Euler systems determine all of the higher Fitting ideals over $\cR$ of the relevant Selmer modules.



In this regard we also note that obtaining concrete structural information about natural arithmetic modules such as ideal class groups, Tate-Shafarevic groups and Selmer groups with respect to coefficient rings that are not regular is a notoriously difficult problem and, despite an extensive literature discussing special cases (for recent examples see for instance Kurihara \cite{kuri}, Greither and Popescu \cite{GreitherPopescu} and Greither and Ku\v cera \cite{GK} and the references contained therein), there has not hitherto been any general approach to this problem.

There are two further differences between our approach (using exterior power biduals) and that of earlier articles that seem worthy of comment.

Firstly, we are able to show that, under standard hypotheses, {\it the module of higher rank Kolyvagin systems is canonically isomorphic to the corresponding module of higher rank Stark systems and is therefore free of rank one over the coefficient ring $\cR$}. This key fact allows us avoid the problem highlighted by Mazur and Rubin in \cite[Rem. 11.9]{MRselmer} that not all higher rank Kolyvagin systems defined in terms of higher exterior powers are, in their terminology, `stub' systems.

Secondly, as a key step in our construction of the higher Kolyvagin derivative operator, we shall develop a variant of the (non-canonical) rank-reduction methods employed in \cite{Buyuk}, \cite{PR} and \cite{rubinstark} that is both canonical in nature and also avoids difficult auxiliary hypotheses that are used in these earlier articles.

In a further article it will be shown that all of the key aspects of the theory developed here extend naturally to the analogous Iwasawa-theoretic setting.

In view of the range of existing applications of the classical theory of Euler and Kolyvagin systems, it thus seems plausible that the very general theory developed here will have significant applications in the future.

For the moment, however, to give an early indication of the usefulness of this approach we shall just restrict to the setting that was originally considered by Rubin in \cite{rubinstark}.


In this setting, we find that our methods allow us, in a straightforward fashion, to extend, refine and remove any hypotheses concerning the validity of Leopoldt's Conjecture from the results of Rubin in \cite{rubinstark} and of B\"uy\"ukboduk in \cite{Buyuk}. 

\subsection{A summary of results} For the reader's convenience, we now give a brief summary of the main results of this article. For simplicity, we shall only discuss a special case of the general setting considered in later sections. In addition, we shall  omit stating explicitly hypotheses that are standard in the theory of Euler, Kolyvagin and Stark systems (since, in each case, they are made precise by the indicated results in later sections).

We thus fix an odd prime $p$, a number field $K$ and a Galois representation $T$ over a Gorenstein $\ZZ_p$-order $\cR$ that is endowed with a continuous action of the absolute Galois group of $K$. We also fix a power $M$ of $p$ and set $A:=T/MT$ and $R:=\cR/(M)$. For each natural number $r$ we write ${\rm ES}_r(T)$, respectively ${\rm KS}_r(A)$ and ${\rm SS}_r(A)$, for the modules of Euler systems of rank $r$ for $T$, respectively of Kolyvagin and Stark systems of rank $r$ for $A$, that are defined in \S \ref{euler sys sec 1}, respectively in \S \ref{defkoly} and \S \ref{defstark}.

Our main result is then the following.

\begin{theorem}[{Theorem \ref{derivable1} and Corollaries \ref{higher der} and \ref{remark surjective}}]\label{thm1} Under standard hypotheses, there exists a canonical `higher Kolyvagin derivative' homomorphism 
$$\cD_r: {\rm ES}_r(T) \to {\rm KS}_r(A).$$
Under certain mild additional hypotheses, this homomorphism is surjective.
\end{theorem}


By this result, one can associate a canonical Kolyvagin system
\[ \kappa(c):=\cD_r(c)\]
in ${\rm KS}_r(A)$ to every Euler system $c$ in ${\rm ES}_r(T)$.

A key aspect of the proof of Theorem \ref{thm1} is the development of a suitable `rank-reduction' method by which consideration is restricted to the case $r=1$ where the result can be established by the existing methods of Mazur and Rubin.

We shall also further develop the general theory of higher rank Kolyvagin systems in order to prove the next result.

In the sequel, for each commutative noetherian ring $\Lambda$, each non-negative integer $i$ and each finitely generated $\Lambda$-module $M$ we write ${\rm Fitt}_\Lambda^i(M)$ for the $i$-th Fitting ideal of $M$.

\begin{theorem}[{Theorem \ref{main}}]\label{thm2} Under standard hypotheses, the following claims are valid.
\begin{itemize}
\item[(i)] The module ${\rm KS}_r(A)$ of Kolyvagin systems of rank $r$ is free of rank one over $R$.
\item[(ii)] For each system $\kappa$ in ${\rm KS}_r(A)$ and each non-negative integer $i$ one has
$$I_i(\kappa)\subseteq {\rm Fitt}_R^i({\rm Sel}(A)),$$
where $I_i(\kappa)$ is a canonical ideal associated with $\kappa$ (see Definition \ref{koly ideal}) and ${\rm Sel}(A)$ is a natural (dual) Selmer module for $A$ (which is denoted by $H^1_{\cF^\ast}(K,A^\ast(1))^\ast$ in Theorem \ref{main}).
\item[(iii)] If $R$ is a principal ideal ring and $\kappa$ is a basis of ${\rm KS}_r(A)$, then the inclusion in claim (ii) is an equality.
\end{itemize}
\end{theorem}

Upon combining Theorems \ref{thm1} and \ref{thm2}, we immediately obtain the following result.

\begin{theorem}[{Corollaries \ref{derivable cor} and \ref{remark surjective}}] \label{thm3} The following claims are valid.
\begin{itemize}
\item[(i)] Under standard hypotheses, for each Euler system $c$ in ${\rm ES}_r(T)$ and each non-negative integer $i$ one has $I_i(\kappa(c))\subseteq {\rm Fitt}_R^i({\rm Sel}(A)).$
\item[(ii)] Under certain mild additional hypotheses, and if $R$ is a principal ideal ring, then for each non-negative integer $i$ one has
 $\langle I_i(\kappa(c)) \mid c \in {\rm ES}_r(T ) \rangle_R = {\rm Fitt}_R^i({\rm Sel}(A)).$
\end{itemize}
\end{theorem}

As a key step in the proof of Theorem \ref{thm2}, we must develop the theory of Stark systems in order to prove the next result.

\begin{theorem}[{Theorem \ref{thm stark}}] \label{thm5} Under standard hypotheses, the following claims are valid.
\begin{itemize}
\item[(i)] The module ${\rm SS}_r(A)$ of Stark systems of rank $r$ is free of rank one over $R$.
\item[(ii)] For each system $\epsilon$ in ${\rm SS}_r(A)$ and each non-negative integer $i$ one has
$$I_i(\epsilon)\subseteq {\rm Fitt}_R^i({\rm Sel}(A)),$$
where $I_i(\epsilon)$ is a canonical ideal associated with $\epsilon$ (see Definition \ref{stark ideal}).
\item[(iii)] If $\epsilon$ is a basis of ${\rm SS}_r(A)$, then the inclusion in claim (ii) is an equality.
\end{itemize}
\end{theorem}

To relate Stark and Kolyvagin systems we prove (in \S\ref{sec regulator}) that there exists a canonical `regulator' homomorphism of $R$-modules
$${\rm Reg}_r: {\rm SS}_r(A) \to {\rm KS}_r(A).$$
{\it We are able to prove that this map is bijective, and as a consequence, that the module ${\rm KS}_r(A)$ is free of rank one} (see Theorem \ref{main}(i)). By combining this fact 
with Theorem~\ref{thm5}, we are then able to prove Theorem~\ref{thm2}.

We remark here that, for each system $\epsilon$ in ${\rm SS}_r(A)$, it is natural to expect that for each non-negative integer $i$ one has   $I_i(\epsilon)=I_i({\rm Reg}_r(\epsilon))$. However, this seems to be difficult to prove and at present we have only verified it in the case $R$ is a principal ideal ring (which leads to Theorem \ref{thm2}(iii)). For more details concerning this issue see Remark \ref{difficulty remark2}.

By simultaneously considering the representations $T/p^m T$ for all natural numbers $m$, one can define Kolyvagin and Stark systems `over $\cR$' by setting
$${\rm KS}_r(T):=\varprojlim_m {\rm KS}_r(T/p^mT) \,\,\text{ and }\,\, {\rm SS}_r(T):=\varprojlim_m {\rm SS}_r(T/p^mT).$$
In this way we obtain analogues of Theorems \ref{thm2}, \ref{thm3} and \ref{thm5} for the Selmer modules of $T$ (see Theorem \ref{thm koly'}, Corollary \ref{main cor} and Theorem \ref{thm stark'} respectively).

To give a straightforward application of the general theory, we consider, for each one-dimensional $p$-adic character $\chi$ of the absolute Galois group of $K$, a certain twisted form $T_\chi$ of the representation $\ZZ_p(1)$. Since in this case the (dual) Selmer module coincides with the $\chi$-isotyic component of a suitable ideal class group, we obtain the following result.

\begin{theorem}[{Theorem \ref{RS theorem}}]\label{thm 6} Let $\chi$ be a one-dimensional $p$-adic character of the absolute Galois group of $K$ that  is of finite prime-to-$p$ order. Let $L$ be the field fixed by $\ker(\chi)$. Assume that all archimedean places of $K$ split completely in $L$, that no $p$-adic place of $K$ splits completely in $L$, that $\chi$ is neither trivial nor equal to the Teichm\"uller character, and that either $p>3$ or $\chi^2$ is not equal to the Teichm\"uller character.

Let $r$ be the number of archimedean places of $K$ and set $\mathcal{O} := \ZZ_p[\im(\chi)]$. Then for each non-negative integer $i$ one has
\[ \langle I_i(\kappa(c)) \mid c \in {\rm ES}_r(T_\chi) \rangle_{\mathcal{O}} =
{\rm Fitt}_{\mathcal{O}}^i((\ZZ_p\otimes_\ZZ {\rm Cl}(\mathcal{O}_L))^\chi),\]
where ${\rm Cl}(\mathcal{O}_L)$ is the ideal class group of $L$.
\end{theorem}

We recall that, in this setting, the Rubin-Stark conjecture predicts the existence of a canonical Euler system in ${\rm ES}_r(T_\chi)$ comprising `Rubin-Stark elements' that are explicitly related to the values at zero of $r$-th derivatives of Dirichlet $L$-functions.

Finally, we note that, under mild additional hypotheses, a slightly more careful application of the methods used to prove Theorem \ref{thm 6} shows that for abelian extensions $L$ of $K$ the higher Fitting ideals of $\ZZ_p\otimes_\ZZ {\rm Cl}(\mathcal{O}_L)$ as a $\ZZ_p[\Gal(L/K)]$-module are determined by Euler systems of rank $r$ for induced forms of the representation $\ZZ_p(1)$. (For brevity, however, this result will be discussed in a separate article.)

%

\subsection{Organization} In a little more detail, the basic contents of this article is as follows. In \S\ref{ext bidual sec} we review the definitions and basic properties of exterior power biduals and then establish several new functorial properties that will play a crucial role in later sections. In \S\ref{pre} we establish various preliminary results concerning Selmer structures and Galois cohomology groups over zero-dimensional Gorenstein rings. In \S\ref{stark sys sec} we review the main results on higher rank Stark systems and Selmer groups that were established in our earlier articles \cite{sbA} and \cite{sakamoto} and also extend these results to the setting of representations over Gorenstein orders. In \S\ref{koly sys sec} we develop a theory of higher rank Kolyvagin systems in the same degree of generality and, in particular, show that under standard hypotheses, the modules of Stark systems and Kolyvagin systems (of any given rank) are both canonically isomorphic and free of rank one over the relevant ring of coefficients. In \S\ref{euler sys sec} we construct a canonical higher rank `Kolyvagin-derivative' homomorphism between the module of Euler systems of any given rank and the module of Kolyvagin systems (defined with respect to the canonical Selmer structure) of the same rank. This is the key result of this article and, as an essential part of its proof, we establish precise links to the corresponding situation in rank one.
Finally, in \S\ref{app sec} we show that our approach leads directly to new results concerning the Galois structure of ideal class groups.


\subsection{Some general notation}
In this article, $K$ always denotes a (base) number field (that is, a finite degree extension of $\QQ$). We fix an algebraic closure $\overline \QQ$ of $\QQ$, and every algebraic extension of $\QQ$ is regarded as a subfield of $\overline \QQ$. For a positive integer $m$, let $\mu_m$ denote the group of $m$-th roots of unity in $\overline \QQ$. For any field $E$, we denote the absolute Galois group of $E$ by $G_E$. For each place $v$ of $K$, we fix a place $w$ of $\overline \QQ$ lying above $v$, and identify the decomposition group of $w$ in $G_K$ with $G_{K_v}$. For a finite extension $F/K$, the ring of integers of $F$ is denoted by $\cO_F$. For a finite set $\Sigma$ of places of $K$, we denote by $\Sigma_F$ the set of places of $F$ which lie above a place in $\Sigma$. The ring of $\Sigma_F$-integers of $F$ is denoted by $\cO_{F,\Sigma}$. The set of archimedean  places (resp. $p$-adic places) of $F$ is denoted by $S_\infty(F)$ (resp. $S_p(F)$). We denote the set of places of $K$ which ramify in $F$ by $S_{\rm ram}(F/K)$.

Non-archimedean places (or `primes') of $K$ are usually denoted by $\fq$. The Frobenius element of $\fq$ is denoted by ${\rm Fr}_\fq$.

For a continuous $G_K$-module $A$, we denote the set of places of $K$ at which $A$ is ramified by $S_{\rm ram}(A)$. Suppose that there is a finite set $S$ of places of $K$ such that
$$S_\infty(K)\cup S_p(K) \cup S_{\rm ram}(A) \subseteq S.$$
Let $K_S$ denote the maximal Galois extension of $K$ unramified outside $S$. Then we can consider Galois cohomology groups
$$H^i(\cO_{K,S},A):=H^i(K_S/K,A).$$
If $A$ is a $p$-adic representation, then
one can define, for each place $v $ of $K$, a canonical (so-called) `finite' local condition
$$
H^1_f(K_v, A) \subseteq H^1(K_v,A)
$$
(see \cite[\S 1.3]{R}, for example). When $v \notin S$, this is the unramified cohomology $H^1_{\rm ur}(K_v,A)$, which is defined by the kernel of the restriction to the inertia group. We set
$$
H^1_{/f}(K_v,A):=H^1(K_v,A)/H^1_f(K_v,A).
$$
More generally, for each index $\ast$ we set $H^1_{/\ast}(K_v,A):=H^1(K_v,A)/H^1_\ast(K_v,A).$

\section{Exterior power biduals}\label{ext bidual sec} In this section we quickly review the basic theory of exterior power biduals and then prove several new results that will be very useful in the sequel.

At the outset we fix a commutative ring $R$. {\it All rings are assumed to be noetherian}. For each $R$-module $X$ we set
\[ X^\ast:=\Hom_R(X,R).\]
For each subset $\mathcal{X}$ of $X$ we write $\langle \mathcal{X}\rangle_R$ for the $R$-submodule of $X$ that is generated by $\mathcal{X}$.

If $X$ is finitely presented, then for each non-negative integer $i$ we write ${\rm Fitt}_R^i(X)$ for the $i$-th Fitting ideal of $X$ (as discussed by Northcott in \cite{north}).

\subsection{Definition and basic properties} We first recall the basic definitions.

For each $R$-module $X$, each positive integer $r$ and each map $\varphi$ in $X^\ast$, there exists a unique homomorphism of $R$-modules
$$ {{\bigwedge}}_R^r X \to {{\bigwedge}}_R^{r-1} X$$
with the property that
$$x_1\wedge\cdots\wedge x_r \mapsto \sum_{i=1}^{r} (-1)^{i+1} \varphi(x_i) x_1\wedge\cdots\wedge x_{i-1}\wedge x_{i+1} \wedge \cdots \wedge x_r$$

\noindent{}for each subset $\{x_i\}_{1\le i\le r}$ of $X$. By abuse of notation, we shall also denote this map by $\varphi$.
For non-negative integers $r$ and $s$ with $r \leq s$, this construction induces a natural homomorphism
$$
{{\bigwedge}}_R^r X^\ast \to \Hom_R\left({{\bigwedge}}_R^s X, {{\bigwedge}}_R^{s-r} X\right); \ \varphi_1\wedge \cdots \wedge \varphi_r \mapsto \varphi_r \circ \cdots \circ \varphi_1.
$$
Here $\bigwedge_R^r X^\ast$ means $\bigwedge_R^r (X^\ast)$. (We often use such an abbreviation.)
%
%
We use this map to regard any element of ${{\bigwedge}}_R^r X^\ast$
as an element of $\Hom_R({{\bigwedge}}_R^s X, {{\bigwedge}}_R^{s-r} X)$.


\begin{definition}\label{def exterior bidual}
For any non-negative integer $r$, we define the `$r$-th exterior bidual' of $X$ to be the $R$-module obtained by setting
$$
{{\bigcap}}_R^r X:=\left({{\bigwedge}}_R^r X^\ast \right)^\ast.
$$
\end{definition}
Note that ${\bigcap}_R^1 X= X^{\ast \ast}$. So, if $X$ is reflexive, i.e. the canonical map
\begin{eqnarray} \label{2 dual}
X \to X^{\ast \ast}; \ x\to (\varphi \mapsto \varphi(x))
\end{eqnarray}
is an isomorphism, then one can regard ${\bigcap}_R^1 X=X$.
In practice, we usually consider exterior biduals of reflexive modules.
(The fact that any modules over self-injective rings (in other words, zero-dimensional Gorenstein rings) are reflexive is often used in this paper.) 

Note also that there is a canonical homomorphism
$$
\xi_X^r: {{\bigwedge}}_R^r X \to {{\bigcap}}_R^r X; \ x \mapsto (\Phi \to \Phi(x)),
$$
which is neither injective nor surjective in general.
However, if $X$ is a finitely generated projective $R$-module, then one can show that $\xi_X^r$ is an isomorphism.

For non-negative integers $r$, $s$ with $r\leq s$ and $\Phi \in {{\bigwedge}}_R^r X^\ast$, define a homomorphism
\begin{eqnarray} \label{map bidual}
{{\bigcap}}_R^s X \to {{\bigcap}}_R^{s-r} X
\end{eqnarray}
as the $R$-dual of
$${{\bigwedge}}_R^{s-r} X^\ast \to {{\bigwedge}}_R^{s} X^\ast ; \ \Psi \mapsto \Phi\wedge \Psi.$$
We denote the map (\ref{map bidual}) also by $\Phi$, by abuse of notation. One can check that the following diagram is commutative:
$$
\xymatrix{
{{\bigwedge}}_R^s X \ar[r]^-{\Phi} \ar[d]_{\xi_X^s} &
{{\bigwedge}}_R^{s-r}X \ar[d]^{\xi_X^{s-r}}
\\
{{\bigcap}}_R^s X \ar[r]^-{\Phi} & {{\bigcap}}_R^{s-r} X.
}
$$
We now recall two results from \cite{sbA} and \cite{sakamoto} that we will frequently use in the sequel.

\begin{proposition}[{\cite[Prop. A.2]{sbA}}, {\cite[Lem. 4.8]{sakamoto}}]\label{prop injective}\
\begin{itemize}
\item[(i)] Let $\iota: X \to Y$ be an injective homomorphism of $R$-modules for which the group
${\rm Ext}_R^1({\rm coker}( \iota), R)$ vanishes. Then for each $r\ge 0$ the homomorphism
$$
{{\bigcap}}_R^r X \hookrightarrow {{\bigcap}}_R^r Y
$$
that is naturally induced by $\iota$ is injective.
\item[(ii)] Suppose that we have an exact sequence of $R$-modules
$$
Y \xrightarrow{\bigoplus_{i=1}^s \varphi_i} R^{\oplus s} \to Z \to 0.
$$
If $Y$ is free of rank $r+s$, then ${\rm Fitt}_R^0(Z)$ is generated over $R$ by the set
$$
\left\{ {\rm im}( F )\ \middle|  \ F \in {\rm im}\left( {{\bigcap}}_R^{r+s} Y \xrightarrow{{{\bigwedge}}_{1\leq i \leq s} \varphi_i} {{\bigcap}}_R^{r} Y \right)   \right\}.
$$
\end{itemize}
\end{proposition}

\begin{proposition}[{\cite[Prop. A.3]{sbA}}, {\cite[Lem. 2.1]{sakamoto}}] \label{prop self injective}
Suppose that $R$ is self-injective, i.e. $R$ is injective as an $R$-module, and that we have an exact sequence of $R$-modules
$$0 \to X \to Y \stackrel{\bigoplus_{i=1}^s \varphi_i}{\to} R^{\oplus s},$$
where $s$ is a positive integer. Then, for every non-negative integer $r$, we have
$${\rm im}\left({{\bigwedge}}_{1\leq i \leq s} \varphi_i: {{\bigcap}}_R^{r+s} Y \to  {{\bigcap}}_R^{r} Y \right)\subseteq  {{\bigcap}}_R^{r} X.$$
Here we regard ${{\bigcap}}_R^r X \subseteq  {{\bigcap}}_R^r Y$ by Proposition \ref{prop injective}(i).
In particular, ${{\bigwedge}}_{1\leq i\leq s} \varphi_i$ induces a homomorphism
$${{\bigwedge}}_{1\leq i \leq s} \varphi_i: {{\bigcap}}_R^{r+s}Y \to {{\bigcap}}_R^r X.$$
\end{proposition}


\subsection{Further functorialities} In this section we prove several new properties of exterior power biduals that will play a key role in subsequent sections.

We first establish an appropriate formalism in our setting of the `rank reduction methods' that were initiated by Rubin in \cite{rubinstark} and subsequently used by Perrin-Riou and by B\"uy\"ukboduk. This result will later play an essential role in the proof of Theorem~\ref{thm1}.

\begin{proposition} \label{reduction}
Suppose that $R$ is self-injective.
Let $X$ be an $R$-module and $Y$ an $R$-submodule of $X$.
Let $r$ be a non-negative integer and identify ${\bigcap}_R^r Y$ with
an $R$-submodule of ${\bigcap}_R^{r} X$ by using Proposition~\ref{prop injective}(i).
Then one has
\[
{\bigcap}_R^r Y = \left\{ x \in {\bigcap}^{r}_{R}X \ \middle| \ \Phi(x) \in  Y\!\left(={\bigcap}_R^1 Y\right) \text{ for all } \Phi \in {\bigwedge}_R^{r-1} X^\ast \right\}.
\]


%

\end{proposition}

\begin{proof} Since $R$ is self-injective, upon taking the $R$-dual of the tautological exact sequence $0 \to Y \to X \to X/Y\to 0$ we obtain another exact sequence
$$
0 \to (X/Y)^\ast \to X^\ast \to Y^\ast \to 0.
$$
Then, by applying the result of Lemma \ref{wedge-kernel} below to this sequence, we deduce that
\begin{align*}
\ker \left( {\bigwedge}^r_R X^\ast \to {\bigwedge}^r_R Y^\ast \right)=\left\langle \varphi \wedge \Phi \ \middle| \ \varphi \in (X/Y)^\ast, \ \Phi \in {\bigwedge}^{r-1}_R X^\ast \right\rangle_R,
\end{align*}
and hence that the sequence
$$
\left\langle \varphi \wedge \Phi \ \middle| \ \varphi \in (X/Y)^\ast, \ \Phi \in {\bigwedge}^{r-1}_R X^\ast \right\rangle_R \to {\bigwedge}^r_R X^\ast \to {\bigwedge}^r_R Y^\ast \to 0
$$
is exact.
Taking $R$-duals of the latter exact sequence, we deduce that
\begin{align}\label{characterize}
{\bigcap}_R^r Y &= \left\{ x \in {\bigcap}_R^r X \  \middle| \ x (\varphi\wedge\Phi)=0 \text{ for every $\varphi \in (X/Y)^\ast$ and $\Phi \in {\bigwedge}^{r-1}_R X^\ast$} \right\}
\end{align}

Now we suppose that $x \in {\bigcap}_R^r X$ satisfies $\Phi(x) \in  Y $ for every $\Phi \in \bigwedge_R^{r-1} X^\ast $. Then we have
$$
x(\varphi \wedge \Phi)=\varphi(\Phi(x))=0
$$
for every $\varphi \in (X/Y)^\ast\subseteq  X^\ast$ and $\Phi \in \bigwedge_R^{r-1} X^\ast $.
Hence, by (\ref{characterize}), we have $x \in {\bigcap}_R^r Y$.
This proves the proposition since the opposite inclusion is clear.
\end{proof}


\begin{lemma}\label{wedge-kernel} We suppose given a short exact sequence of $R$-modules
\[
0\to X \xrightarrow{\iota} Y \xrightarrow{\pi} Z\to 0.
\]
Then, for every non-negative integer $r$, the kernel of the natural homomorphism
\[ {\bigwedge}^r_R Y \stackrel{\pi}{\to} {\bigwedge}^r_R Z\]
is generated over $R$ by the elements $\iota(\varphi) \wedge \Phi$ as $\varphi$ ranges over $X$ and $\Phi$ over ${\bigwedge}^{r-1}_R Y$.
\end{lemma}
\begin{proof}
We shall denote $\otimes_R$ simply by $\otimes$.
For each integer $i$ with $1 \leq i \leq r$, we consider the $R$-module
$$
\Omega_{r,i} := Y \otimes \cdots \otimes Y \otimes X \otimes Y \otimes \cdots \otimes Y,
$$
in which there are $r-1$ copies of $Y$ and the module $X$ occurs as the $i$-th module from the left hand end (so that $\Omega_{r,i}$ is isomorphic to $X \otimes Y^{\otimes(r-1)}$).

We set
$$\Omega_r:=\bigoplus_{i=1}^{r} \Omega_{r,i}$$
and claim that the sequence of $R$-modules
\begin{align*}
\Omega_r \to  Y^{\otimes r} \xrightarrow{\pi^{\otimes r}}  Z^{\otimes r} \to 0,
\end{align*}
in which the first map is the direct sum of the maps $\Omega_{r,i}\to Y^{\otimes r}$ that are induced by the given map $\iota: X \to Y$, is exact.

To prove this claim we use induction on $r$. When $r=1$, the claim is clear. When $r>1$, by the inductive hypothesis, we have
the following commutative diagram of $R$-modules, whose rows and columns are exact:
\begin{align*}
\xymatrix{
\Omega_{r-1} \otimes X \ar[r] \ar[d] &  Y^{\otimes (r-1)} \otimes X \ar[r] \ar[d] &
 Z^{\otimes(r-1)} \otimes X \ar[r] \ar[d] & 0
\\
\Omega_{r-1} \otimes Y \ar[r] \ar[d] &  Y^{\otimes r} \ar[r] \ar[d] &
Z^{\otimes(r-1)} \otimes Y \ar[r] \ar[d] & 0
\\
\Omega_{r-1} \otimes Z \ar[r] \ar[d] & Y^{\otimes (r-1)} \otimes Z \ar[r] \ar[d] &
 Z^{\otimes r} \ar[r] \ar[d] & 0
\\
0 & 0 & 0.
}
\end{align*}
Since $\Omega_r = (\Omega_{r-1} \otimes Y) \oplus  (Y^{\otimes (r-1)} \otimes X)$,
we see that
$$\im \left( \Omega_r \to Y^{\otimes r} \right) =
\ker \left( Y^{\otimes r} \to  Z^{\otimes r} \right)$$
 by diagram chasing. Hence we have proved the claim.

Now we put $\Upsilon_r := \ker\left( {\bigwedge}^r_R Y \to {\bigwedge}^r_R Z \right)$ and
consider the commutative diagram
\begin{align*}
\xymatrix{
& \Omega_r \ar[r] \ar[d] & Y^{\otimes r} \ar[r] \ar[d]^{f_1} & Z^{\otimes r} \ar[r] \ar[d]^{f_2} & 0
\\
0 \ar[r] & \Upsilon_r \ar[r]  &  {\bigwedge}_R^rY \ar[r] & {\bigwedge}_R^r Z \ar[r] & 0,
}
\end{align*}
whose rows are exact.
Clearly, $f_1$ is surjective, and the map $\ker (f_1) \to \ker (f_2)$ is also surjective, so we see that the map $\Omega_r \to \Upsilon_r$ is surjective (by Snake lemma).
This shows that the inclusion
$$\Upsilon_r \subseteq  \left\langle \varphi \wedge \Psi \mid \varphi \in X, \ \Psi \in {\bigwedge}^{r-1}_R Y \right\rangle_R$$
holds.
Since the opposite inclusion clearly holds, we have proved the lemma.
\end{proof}

\begin{corollary}\label{bidual-ker}
Suppose that $R$ is self-injective. Let $X$ be an $R$-module and $f \colon X \to R$ an $R$-homomorphism.
Then for any non-negative integer $r$, we have
\begin{align*}
{\bigcap}^{r}_{R}\ker\left( f \right) = \ker\left(f \colon {\bigcap}^{r}_{R}X \to {\bigcap}^{r-1}_{R}X \right).
\end{align*}
\end{corollary}
\begin{proof}
By definition, we see that the composition
${\bigcap}^{r}_{R}\ker\left( f \right) \hookrightarrow {\bigcap}^{r}_{R}X \xrightarrow{f} {\bigcap}^{r-1}_{R}X$ is the zero map.
Hence we have ${\bigcap}^{r}_{R}\ker\left( f \right) \subseteq  \ker\left(f \colon {\bigcap}^{r}_{R}X \to {\bigcap}^{r-1}_{R}X \right)$.
To prove the opposite inclusion, take $x \in \ker\left(f \colon {\bigcap}^{r}_{R}X \to {\bigcap}^{r-1}_{R}X \right)$.
Then, by definition, we have
$$
0 = f(x)(\Phi) = x(f \wedge \Phi)=f(\Phi(x))
$$
for any $\Phi \in \bigwedge^{r-1}_{R}X^{*}$, and so $\Phi(x) \in \ker(f)$.
By Proposition~\ref{reduction}, we conclude that $x \in {\bigcap}^{r}_{R}\ker\left( f \right)$.
\end{proof}

The following corollary will be used in \S\S \ref{stark sys sec} and \ref{koly sys sec} to define Kolyvagin and Stark systems over a Gorenstein order.

\begin{corollary}\label{morph}
 We suppose to be given a surjective homomorphism $R \to S$ of self-injective rings. Let $X$ be an $R$-module, $F$ a free $R$-module of finite rank, and $Y$ an $S$-module.
We then suppose to be given a commutative diagram of $R$-modules
\begin{align*}
\xymatrix{
X \ar[r] \ar[d] & F \ar[d]^{\pi}
\\
Y \ar@{^{(}->}[r] & F \otimes_{R}S
}
\end{align*}
in which the lower horizontal arrow is injective and the map $\pi$ is induced by the given homomorphism $R \to S$.

Then for any positive integer $r$, there exists a natural homomorphism of $R$-modules
\[ {{\bigcap}}^{r}_{R}X \to {{\bigcap}}^{r}_{S}Y\]
that is independent of the given maps $X \to F$ and $Y \hookrightarrow F \otimes_{R} S$ and is such that all squares of the following diagram
\begin{align*}
\xymatrix{
{{\bigwedge}}^{r}_{R}X \ar[r]^-{\xi^{r}_{X}} \ar[d] & {{\bigcap}}^{r}_{R}X \ar[d] \ar[r] & {{\bigcap}}^{r}_{R}F \ar[d]
\\
{{\bigwedge}}^{r}_{S}Y \ar[r]^-{\xi^{r}_{Y}}  & {{\bigcap}}^{r}_{S}Y \ar@{^{(}->}[r] & {{\bigcap}}^{r}_{S}(F \otimes_{R}S)
}
\end{align*}
commute.
Here the
second upper and lower horizontal arrows denote 
the maps induced by the given maps $X \to F$ and $Y \hookrightarrow F \otimes_{R} S$.
\end{corollary}
\begin{proof}
Since the map $Y \hookrightarrow F \otimes_{R}S$ is injective, by replacing $X$ with $\im\left(X \to F\right)$, we may assume that the map $X \to F$ is injective.
Since $F$ is free of finite rank, the maps $\xi_{F}^{r}$ and $\xi_{F \otimes_{R}S}^{r}$ are isomorphisms.
Hence we get the following commutative diagram:
\begin{align*}
\xymatrix{
{{\bigwedge}}^{r}_{R}X \ar[r]^-{\xi^{r}_{X}} \ar[d] & {{\bigcap}}^{r}_{R}X \ar@{^{(}->}[r] & {{\bigcap}}^{r}_{R}F  \ar[r]^-{(\xi_{F}^{r})^{-1}} & {{\bigwedge}}^{r}_{R}F \ar[d]^-{{\wedge}^{r}\pi}
\\
{{\bigwedge}}^{r}_{S}Y \ar[r]^-{\xi_{Y}^{r}} & {{\bigcap}}^{r}_{S}Y \ar@{^{(}->}[r] & {{\bigcap}}^{r}_{S}(F \otimes_{R}S)  \ar[r]^-{(\xi_{F \otimes_{R}S}^{r})^{-1}} & {{\bigwedge}}^{r}_{S}(F \otimes_{R}S).
}
\end{align*}
Since the diagram commutes and the map ${{\bigcap}}^{r}_{S}Y \hookrightarrow {{\bigcap}}^{r}_{S}(F \otimes_{R}S)$ is injective by Proposition~\ref{prop injective}(i), we only need to show that
$$
\im\left({{\bigcap}}^{r}_{R}X \to {{\bigwedge}}^{r}_{S}(F \otimes_{R}S) \right) \subseteq  \im\left({{\bigcap}}^{r}_{S}Y \to {{\bigwedge}}^{r}_{S}(F \otimes_{R}S)\right).
$$
To see this, take elements $x \in {{\bigwedge}}^{r}_{R}F$ with $\xi_{F}^{r}(x) \in {{\bigcap}}_{R}^{r}X$ and $\Phi \in {{\bigwedge}}^{r-1}_{R}F^{*}$. We also denote by $\Phi$ the image of $\Phi$ under the map
$$
{{\bigwedge}}^{r}_{R}F^{*} \to {{\bigwedge}}^{r}_{S}(F^{*} \otimes_{R}S) \xrightarrow{\sim} {{\bigwedge}}^{r-1}_{S}(F \otimes_{R}S)^{*}.
$$
Then by Proposition~\ref{reduction}, we have $\Phi(x) \in X$.
Hence we conclude that $\Phi((\wedge^{r}\pi)(x)) = \pi(\Phi(x)) \in Y$. Since the natural map
\[ \Hom_{R}(F, R) \otimes_{R}S \to \Hom_{S}(F \otimes_{R}S, S)\]
is bijective we can again apply Proposition~\ref{reduction} to deduce that $\xi_{F \otimes_{R}S}^{r}((\wedge^{r}\pi)(x))$ belongs to $ {{\bigcap}}^{r}_{S}Y$, as required.
\end{proof}

\section{Preliminaries concerning Galois cohomology} \label{pre}

\subsection{Notation and hypotheses}

Let $K$ be a number field. Let $(R, \fp)$ be a self-injective (commutative) local ring (in other words, a zero-dimensional Gorenstein local ring)
with a finite residue field $\Bbbk=R/\fp$ of characteristic $p>0$.
Note that $R$ is an artinian ring, which is finite of order a power of $p$.
Let $A$ be a free $R$-module of finite rank with an $R$-linear continuous action of $G_K$. ($A$ is endowed with discrete topology.) Since $A$ is finite, the action of $G_K$ factors through a finite quotient, so $A$ is unramified outside a finite set of places of $K$.

\subsubsection{Selmer structures} \label{section sel}

We fix a Selmer structure $\mathcal{F}$ on $A$.
Recall that a Selmer structure $\mathcal{F}$ on $A$ is a collection of the following data (see \cite[Def. 2.1.1]{MRkoly}):
\begin{itemize}
\item a finite set $S(\mathcal{F})$ of places of $K$ such that
$S_\infty(K)\cup S_p(K) \cup S_{\rm ram}(A) \subseteq  S(\mathcal{F})$;
\item for every $v \in S(\mathcal{F})$, a choice of an $R$-submodule $H_\mathcal{F}^1(K_v, A) \subseteq  H^1(K_v, A)$.
\end{itemize}
The Selmer module attached to $\mathcal{F}$ is defined by
$$
H_\mathcal{F}^1(K, A) := \ker \left( H^1(\cO_{K,S(\mathcal{F})}, A) \to \bigoplus_{v\in S(\mathcal{F})} H^1_{/\cF}(K_v, A)\right).
$$
(Recall that for any index $\ast$ we set $H^1_{/\ast}(K_v,A):=H^1(K_v,A)/H^1_\ast(K_v,A).) $
Note that
$$
H^1(\cO_{K,S(\cF)}, A)=\ker \left( H^1(K, A) \to \bigoplus_{v \notin S(\cF)} H^1_{/f}(K_v, A)
\right),
$$
so we have
$$
H^1_{\cF}(K,A) = \ker \left( H^1(K, A) \to \bigoplus_{v} H^1_{/\cF}(K_v, A)
\right),
$$
where we set $H^1_\cF(K_v,A):=H^1_f(K_v,A)$ ($=H^1_{\rm ur}(K_v,A)$) for $v \notin S(\cF)$.
In the following, we set $S:=S(\cF)$ for simplicity. (Note that Selmer structures and modules can be defined for continuous representations of $G_K$ of various other types: in \S \ref{one-dim case}, for example, we consider representations over one-dimensional Gorenstein rings and Selmer modules associated to them.)

We next review the definition of the dual Selmer structure $\cF^\ast$ on $A^\ast(1)$. By local Tate duality, we have a canonical isomorphism
$$H^1(K_v, A) \simeq H^1(K_v, A^\ast(1))^\ast.$$
Using this, we define $\cF^\ast$ to be the following data
\begin{itemize}
\item $S(\cF^\ast):=S(\cF)(=S)$;
\item for $v\in S$,
$$H^1_{\cF^\ast}(K_v, A^\ast(1)):=\ker(H^1(K_v, A^\ast(1)) \simeq H^1(K_v, A)^\ast \to H^1_{\cF}(K_v,A)^\ast). $$
\end{itemize}

If $B$ is an $R[G_{K}]$-submodule (resp. quotient) of $A$, then the Selmer structure $\cF$ on $A$ induces a Selmer structure on $B$ (which we also denote by $\cF$) as follows: we define the local condition by
$$
H^1_\cF(K_v, B):=\ker(H^1(K_v,B) \to  H^1_{/\cF}(K_v,A))
$$
$$
(\text{resp. }H^1_{\cF}(K_v,B):=\im (H^1_{\cF}(K_v,A) \to H^1(K_v,B))).
$$

The following result is a consequence of the (Poitou-Tate) global duality and is proved by Mazur and Rubin in \cite[Th. 2.3.4]{MRkoly}.

\begin{theorem}[Global duality]\label{gd} Let $\cF_1$ and $\cF_2$ be Selmer structures on $A$ such that at every place $v$ one has $H^1_{\cF_1}(K_v,A)\subseteq   H^1_{\cF_2}(K_v,A)$. Then there exists a canonical exact sequence
\begin{multline*}
0\to H^1_{\cF_1}(K,A) \to H^1_{\cF_2}(K,A) \to \bigoplus_{v } H^1_{\cF_2}(K_v,A)/H^1_{\cF_1}(K_v,A) \\
\to H^1_{\cF_1^\ast}(K,A^\ast(1))^\ast \to H^1_{\cF_2^\ast}(K,A^\ast(1))^\ast \to 0.
\end{multline*}
\end{theorem}

\subsubsection{Hypotheses} \label{section hyp}

Let $K(A)$ be the minimal Galois extension of $K$ such that $G_{K(A)}$ acts trivially on $A$.
Let $M := \min \{ p^n \mid  p^n R = 0 \}$.
We denote by $K(1)$ the maximal $p$-extension of $K$ inside the Hilbert class field of $K$. 
We set
\[
K_M := K (\mu_{M}, (\cO^\times_K)^{1/M}) K(1)\,\,\text{ and }\,\,K(A)_M := K(A)K_M.
\]
Here $(\cO_K^\times)^{1/M}$ denotes the set $\{u \in \overline \QQ \mid u^M \in \cO_K^\times\}$.

In the following, we assume the following hypotheses.

\begin{hypothesis}\label{hyp1}\
\begin{itemize}
\item[(i)] $A \otimes_R \Bbbk$ is an irreducible $\Bbbk[G_K]$-module;
\item[(ii)] there exists $\tau \in G_{K_M}$ such that $A/(\tau - 1)A \simeq R$ as $R$-modules;
\item[(iii)] $H^1(K(A)_M/K, A) = H^1(K(A)_M/K, A^*(1)) = 0$.
\end{itemize}
\end{hypothesis}

\begin{hypothesis}\label{hyp2}
$(A \otimes_R \Bbbk)^{G_K} = ((A \otimes_R \Bbbk)^*(1))^{G_K} = 0$.
\end{hypothesis}

\begin{remark} It is clear that if $A$ satisfies Hypothesis~\ref{hyp1}, then so also does $A^*(1)$. Note also that the vanishing of $ ((A \otimes_R \Bbbk)^*(1))^{G_K}$ is equivalent to the vanishing of $ (A^\ast(1)\otimes_R \Bbbk)^{G_K}$ and so Hypothesis~\ref{hyp2} for $A$ is equivalent to the same hypothesis for $A^\ast(1)$.
\end{remark}

Let $\mathcal{P}$ be the set of primes $\fq \not\in S $ of $K$ such that ${\rm Fr}_\fq$ is conjugate to $\tau$ in ${\rm Gal}(K(A)_M/K)$. In particular, by Hypothesis~\ref{hyp1}(ii), we see that $\fq$ splits completely in $K_M$ and that $A/({\rm Fr}_\fq-1)A \simeq R$ for every $\fq \in \cP$.
 For a set $\mathcal{Q}$ of primes of $K$, we denote by $\mathcal{N}(\mathcal{Q})$ the set of square-free products of primes in $\mathcal{Q}$. (We let $1 \in \mathcal{N}(\mathcal{Q})$ for convention.) For $\fn \in \cN$, the number of primes which divide $\fn$ is denoted by $\nu(\fn)$. (When $\fn=1$, we set $\nu(1):=0$.) We often abbreviate $\mathcal{N}(\mathcal{P})$ to $\mathcal{N}$.

For $\fq \in \cP$, we denote by $K(\fq)$ the maximal $p$-extension of $K$
inside the ray class field modulo $\fq$. Put $G_\fq:=\Gal(K(\fq)/K(1))=\Gal(K(\fq)_\fq/K_\fq)$, where $K(\fq)_\fq$ is the completion of $K(\fq)$ at the (fixed) place lying above $\fq$.
For $\fn \in \cN$, we denote by $K(\fn)$ the composite of
$K(\fq)$'s with $\fq \mid \fn$.
We set
\begin{align*}
G_\fn := \bigotimes_{\fq \mid \fn} G_\fq.
\end{align*}
Note that, since $\fq \in \cP$ splits completely in $K_M$,
the order of $G_\fq$ is divisible by $M$.

\subsubsection{Modified Selmer structures} \label{section mod sel}

For $\fa,\fb,\fn \in \cN$ which are pairwise relatively prime, we define a Selmer structure $\mathcal{F}_\fa^\fb(\fn)$ by
\begin{itemize}
\item $S(\mathcal{F}_\fa^\fb(\fn)):=S \cup \{ \fq \mid \fa\fb\fn \}$;
\item for $v \in S(\mathcal{F}_\fa^\fb(\fn))$,
\begin{align*}
H^1_\mathcal{F_\fa^\fb(\fn)}(K_v,A) :=
\begin{cases}
H^1_\cF(K_v, A) &\text{ if $v \in S$}, \\
0 &\text{ if $v \mid \fa$,}\\
H^1(K_v , A) &\text{ if $v\mid \fb$},\\
H^1_{\rm tr}(K_v,A) &\text{ if $v \mid \fn$},
\end{cases}
\end{align*}
\end{itemize}
where $H^1_{\rm tr}(K_\fq,A)$ is the `transverse' submodule of $H^1(K_\fq, A)$ which fits in a canonical decomposition
$$H^1(K_\fq, A)=H^1_f(K_\fq, A)\oplus H^1_{\rm tr}(K_\fq, A). $$
Explicitly, one defines
$$H^1_{\rm tr}(K_\fq, A):=H^1(K(\fq)_\fq/K_\fq, A^{G_{K(\fq)_\fq}})=\Hom(G_\fq, A^{{\rm Fr}_\fq=1}),$$
which is regarded as a submodule of $H^1(K_\fq, A)$ by the inflation map. Since $H^1_{\rm tr}(K_\fq, A)$ is canonically isomorphic to $H^1_{/f}(K_\fq,A)$, we sometimes identify them.
Note that $(\cF_\fa^\fb(\fn))^\ast=\cF_\fb^\fa(\fn)$. This follows from the fact that
$$H^1_f(K_\fq,A)\simeq H^1_{/f}(K_\fq,A^\ast(1))^\ast$$
for every $\fq\in \cP$.
If $\fa=1$, we abbreviate $\cF_\fa^\fb(\fn)$ to $\cF^\fb(\fn)$. Similarly, if $\fb=1$ or $\fn=1$, they are omitted.
The Selmer structures $\cF^\fn$ and $\cF(\fn)$ will often appear. We remark that the associated Selmer modules are
$$
H^1_{\cF^\fn}(K,A) := \ker \left( H^1(\cO_{K,S_\fn}, A) \to \bigoplus_{v\in S}H^1_{/\cF}(K_v,A)
\right)
$$
and
$$
H^1_{\cF(\fn)}(K,A) := \ker \left( H^1(\cO_{K,S_\fn}, A) \to \bigoplus_{v\in S}H^1_{/\cF}(K_v,A)
\oplus \bigoplus_{\fq \mid \fn}  H^1_{/ {\rm tr}}(K_v,A)\right)$$
respectively, where $S_\fn:=S\cup \{\fq \mid \fn\}$.

\subsubsection{Finite-singular comparison maps} \label{section fs}
For $\fq \in \cP$, we recall the definition of the `finite-singular comparison map'
\begin{align*}
\varphi^{\rm fs}_\fq \colon H^1_f(K_\fq, A) \simeq H^1_{\rm tr}(K_\fq, A) \otimes G_\fq,
\end{align*}
which is important in the theory of Kolyvagin systems.
We use the following canonical isomorphisms:
\begin{align*}
H^1_{f}(K_\fq, A) \simeq A/({\rm Fr}_\fq - 1)A; \ a \mapsto a({\rm Fr}_\fq) \ \text{(evaluation of ${\rm Fr}_\fq$ to the 1-cocycle $a$)},
\end{align*}
\begin{align*}
H^1_{\rm tr}(K_\fq, A) \otimes G_\fq= \Hom(G_\fq, A^{{\rm Fr}_\fq = 1})  \otimes G_\fq \simeq A^{{\rm Fr}_\fq = 1}; \
f \otimes \sigma \mapsto f(\sigma).
\end{align*}
Since $A/({\rm Fr}_\fq-1)A\simeq R$, one has $\det(1 - {\rm Fr}_\fq  \mid A) = 0$, so
there exists a unique polynomial $Q_\fq(x) \in R[x]$ such that
$(x - 1)Q_\fq(x) = \det(1 - {\rm Fr}_\fq x \mid A)$ in $R[x]$.
By the Cayley-Hamilton theorem, we know that $({\rm Fr}_\fq^{-1} - 1)Q_\fq({\rm Fr}_\fq^{-1})$ annihilates $A$,
so we have a well-defined map
$$
Q_\fq({\rm Fr}_\fq^{-1}): A/({\rm Fr}_\fq-1)A \to  A^{{\rm Fr}_\fq = 1}; \ a \mapsto Q_\fq({\rm Fr}_\fq^{-1})a.
$$
(This is actually an isomorphism, see \cite[Cor. A.2.7]{R}.)

Now we define $\varphi^{\rm fs}_\fq$ to be the composite homomorphism

\begin{align*}
\varphi^{\rm fs}_\fq \colon H^1_{f}(K_\fq, A) \simeq A/({\rm Fr}_\fq - 1)A \xrightarrow{Q_\fq({\rm Fr}_\fq^{-1})}
A^{{\rm Fr}_\fq = 1} \simeq H^1_{\rm tr}(K_\fq, A) \otimes G_\fq.
\end{align*}

\subsection{Application of the Chebotarev density theorem}

In this subsection, we prove several basic results that will be used later.

For an $R$-module $X$ and an ideal $I$ of $R$ we consider the $R$-submodule
\[ X[I] := \{x \in X \mid ax=0\, \text{ for all } \, a \in I\}\]
of $X$ comprising elements that are annihilated by elements of $I$.

\begin{proposition}\label{prop h1}
Assume $(A \otimes_R  \Bbbk )^{G_K}$ vanishes. Then, for any ideal $I$ of $R$, the homomorphism $H^1(K, A[I]) \to H^1(K, A)[I]$
induced by the inclusion $A[I] \hookrightarrow A$ is bijective.
\end{proposition}

Although a proof of this result was given by Mazur and Rubin in \cite[Lem. 3.5.3]{MRkoly}, it relies on the validity of a lemma, which contains an error (see \cite[Lem. 2.1.4]{MRkoly} and \cite[`Erratum' on p.182]{MRselmer}), and so their induction argument apparently fails. For the reader's convenience, we shall give a full proof of Proposition~\ref{prop h1}.

To prove this proposition, we need the following algebraic lemmas.

\begin{lemma}\label{lemma S}
Let $(S, \fp_{S})$ be a local ring, $G$ a group and $B$ an $S[G]$-module.
Suppose that $B$ is a flat $S$-module and that $(B/\fp_{S}B)^{G} = 0$.
Then we have
$$
B^G = {\bigcap}_{i = 0}^{\infty}(\fp^{i}_{S}B)^{G}.
$$
In particular, if $\mathfrak{p}_S^n B=0$ for a sufficiently large $n$, then we have $B^G=0$.
\end{lemma}
\begin{proof}
Since $B$ is a flat $S$-module, we have $S[G]$-isomorphisms
$$
\fp^{i}_{S}B/\fp^{i+1}_{S}B \simeq \fp^i_{S}/\fp^{i+1}_{S} \otimes_{S} B \simeq (B/\fp_{S}B)^{\dim_{S/\fp_{S}}(\fp^{i}_{S}/\fp^{i+1}_{S})}.
$$
By the assumption $(B/\fp_{S}B)^{G} = 0$, we see that $(\fp^i_{S} B)^{G} = (\fp^{i+1}_{S}B)^{G}$.
This equality holds for arbitrary $i$, so we have $B^{G} = {\bigcap}_{i=0}^{\infty}(\fp^{i}_{S}B)^{G}$.
\end{proof}

\begin{lemma}\label{fixed-cor} If $(A \otimes_R  \Bbbk )^{G_K}$ vanishes, then the following claims are valid.  
\begin{itemize}
\item[(i)] For any ideals $I_{1}, \ldots, I_{d}$ of $R$, we have
$$
\left({\rm coker}\left( A \to A/I_{1}A \times \cdots \times A/I_{d}A \right) \right)^{G_{K}} = 0,
$$
where $A \to A/I_{1}A \times \cdots \times A/I_{d}A$ is the diagonal map.
\item[(ii)] For any elements $r_{1}, \ldots, r_{d} \in R$, we have
$$
\left({\rm coker}\left( A \xrightarrow{ (r_{1}, \ldots, r_{d}) \times } A^{d} \right) \right)^{G_{K}} = 0.
$$
\end{itemize}
\end{lemma}
\begin{proof}
We prove claim (i) by induction on $d$. When $d=1$, the claim is clear.
When $d>1$, we set
$$
M_{d} := {\rm coker}\left( A \to A/I_{1}A \times \cdots \times A/I_{d}A \right)
$$
and
$$
M_{d-1} := {\rm coker}\left( A \to A/I_{1}A \times \cdots \times A/I_{d-1}A \right).
$$
Applying the `kernel-cokernel lemma' (see \cite[Exer. 2 in Chap. I, \S 3]{NSW}) to the sequence
$$A \stackrel{f}{\to } A/I_{1}A \times \cdots \times A/I_{d}A \stackrel{g}{\to} A/I_{1}A \times \cdots \times A/I_{d-1}A,$$
where $g$ is the natural projection, we obtain the exact sequence
$$\ker(g\circ f)\to \ker (g) \to \coker (f)\to  \coker(g\circ f) \to \coker(g).$$
Noting that
\begin{align*}
\ker (g\circ f) &= I_1\cap \cdots \cap I_{d-1}=:J,
\\
\ker(g) &= A/I_dA,
\\
\coker(f) &= M_d,
\\
\coker(g\circ f) &= M_{d-1},
\\
\coker(g) &= 0,
\end{align*}
we have an exact sequence
$$
0 \to A/(J+I_d)A \to M_d \to M_{d-1} \to 0.
$$
By the inductive hypothesis, we have $M_{d-1}^{G_K}=0$. Also, by applying Lemma \ref{lemma S} to $S:=R/(J+I_d)$, $B:=A/(J+I_d)A$ and $G:=G_K$, we see that $(A/(J+I_d)A)^{G_K}=0$. Hence we have $M_d^{G_K}=0$. This proves (i).

Next, we prove claim (ii). We set $I_{i} := R[r_{i}]$ for each $1 \leq i \leq d$ and consider the sequence
$$ A \stackrel{f}{\to} A/I_1 A \times \cdots \times A/I_d A \stackrel{g}{\to} A^d,$$
where $g$ is induced by the `multiplication by $(r_1,\ldots,r_d)$':
$$g: A/I_1 A \times \cdots \times A/I_d A \to A^d; \  (a_1,\ldots,a_d) \mapsto (r_1a_1,\ldots,r_da_d).$$
(Note that this is injective.) Using the kernel-cokernel lemma, we obtain the exact sequence
$$0 \to \coker(f)\to \coker(g\circ f) \to A/r_1A \times \cdots \times A/r_dA \to 0. $$
Since we know that $\coker(f)^{G_K}=0$ by (i) and that $(A/r_iA)^{G_K}=0$ by Lemma \ref{lemma S}, we have $\coker(g\circ f)^{G_K}=0$, i.e.
$$
\left({\rm coker}\left( A \xrightarrow{ (r_{1}, \ldots, r_{d}) \times } A^{d} \right) \right)^{G_{K}} = 0.
$$
\end{proof}
%
\begin{proof}[Proof of Proposition \ref{prop h1}]
Let $\{r_1, \ldots, r_d\}$ be a set of generators of an ideal $I$ of $R$.
We have a short exact sequence
\begin{align} \label{ex1}
0 \to A[I] \to A \to A/A[I] \to 0
\end{align}
and an injection
\begin{align} \label{inj1}
A/A[I] \hookrightarrow A^{d};\ a \mapsto (r_1 a ,\ldots, r_d a).
\end{align}
Since $(A \otimes_R \Bbbk)^{G_K} = 0$, we see by Lemma~\ref{lemma S} that the sequence (\ref{ex1}) induces an exact sequence
\begin{align*}
0 \to H^1(K, A[I]) \to H^1(K, A) \to H^1(K, A/A[I]).
\end{align*}
(We apply Lemma~\ref{lemma S} with $S:=R/R[I]$ and $B:=A/A[I]$.) One also sees by Lemma~\ref{fixed-cor}(ii) that (\ref{inj1}) induces an injection
\begin{align*}
H^1(K, A/A[I]) \hookrightarrow H^1(K, A^{d}) = H^1(K, A)^d.
\end{align*}
Since the composition map
\begin{align*}
H^1(K, A) \to H^1(K, A/A[I]) \to H^1(K, A)^d
\end{align*}
is multiplication by $(r_1, \ldots, r_d)$, we have
\begin{align*}
H^1(K, A[I]) &= \ker\left( H^1(K, A) \to H^1(K, A/A[I]) \right)
\\
&= \ker\left( H^1(K, A) \xrightarrow{ (r_1, \ldots, r_d)\times } H^1(K, A)^d \right)
\\
&= H^1(K, A)[I].
\end{align*}
\end{proof}


\begin{corollary}[{\cite[Lem. 3.5.3]{MRkoly}}]\label{dualselisom}
Suppose that $(A \otimes_R \Bbbk)^*(1)^{G_K}$ vanishes.
Then, for any ideal $I$ of $R$, the inclusion map $A^\ast(1)[I] \hookrightarrow A^*(1)$ induces an isomorphism
\begin{align*}
H^1_{\cF^*}(K, A^\ast(1)[I]) \simeq H^1_{\cF^*}(K, A^*(1))[I].
\end{align*}
(Here $\cF$ denotes the fixed Selmer structure on $A$, and $\cF^\ast$ the dual Selmer structure on $A^\ast(1)$ of $\cF$. Note that the Selmer structure on $A^\ast(1)[I]$ induced by $\cF^\ast$ coincides with the dual of the Selmer structure on $A/IA$ induced by $\cF$, so the notation $H^1_{\cF^*}(K, A^\ast(1)[I])$ makes no confusion.)
\end{corollary}

\begin{proof}
We have the following commutative diagram:
\begin{align*}
\xymatrix{
0 \ar[r] & H^{1}_{\cF^{*}}(K, A^{*}(1)[I])  \ar[r] \ar[d] & H^{1}(\cO_{K, S}, A^{*}(1)[I])  \ar[r] \ar[d] &
\ar[d] 
\bigoplus_{v \in S}H^1_{/\cF^{*}}(K_v,A^\ast(1)[I])
\\
0 \ar[r] & H^{1}_{\cF^{*}}(K, A^{*}(1))[I]  \ar[r] & H^{1}(\cO_{K, S}, A^{*}(1))[I]  \ar[r] &
\bigoplus_{v \in S}H^1_{/\cF^{*}}(K_v,A^\ast(1)).
}
\end{align*}
Here each row is exact.
The middle vertical arrow is an isomorphism by Proposition~\ref{prop h1}.
The right vertical arrow is injective since the dual $H^{1}_{\cF}(K_{v}, A) \to H^{1}_{\cF}(K_{v}, A/IA)$ is surjective (by the definition of the induced Selmer structure).
Hence we conclude that the left vertical arrow is an isomorphism.
\end{proof}

Let $\tau \in G_{K_{M}}$ be the element in Hypothesis~\ref{hyp1}(ii).
Recall that $\cP$ is the set of primes $\fq \not \in S $
of $K$ such that ${\rm Fr}_{\fq}$ is conjugate to $\tau$ in $\Gal(K(A)_{M}/K)$.


\begin{lemma}[{\cite[Prop. 3.6.1]{MRkoly}}]\label{chebotarev}
Assume Hypothesis \ref{hyp1}. Let $c_{1}, \ldots, c_{s} \in H^1(K, A)$ and $c_{1}^{*}, \ldots, c_{t}^{*} \in H^1(K, A^*(1))$ be non-zero elements.
If $s+t < p$, then there is a subset $\cQ \subseteq  \cP$ of positive density such that
${\rm loc}_{\fq}(c_{i}) $ and ${\rm loc}_{\fq}(c^{*}_{i}) $ are all non-zero for every $\fq \in \cQ$,
where ${\rm loc}_{\fq}$ denotes the localization map $H^{1}(K, -) \to H^{1}(K_{\fq}, -)$.
\end{lemma}
\begin{proof}
Let $B \in \{A, A^{*}(1) \}$.
By Hypothesis \ref{hyp1}(iii), we know that the restriction map
$$
{\rm Res}: H^{1}(K, B) \to H^{1}(K(A)_{M}, B)^{G_{K}} = \Hom(G_{K(A)_{M}}, B)^{G_{K}}
$$
is injective.
Since $B$ is an irreducible $R[G_{K}]$-module by Hypothesis \ref{hyp1}(i) and since $(\tau - 1)B$ is a free $R$-module of rank ${\rm rank}_{R}(A) - 1$ by Hypothesis \ref{hyp1}(ii), there is no non-trivial $G_{K}$-stable $R$-submodule of $(\tau - 1)B$.
Hence the natural map
$$
f: H^{1}(K(A)_{M}, B)^{G_{K}} = \Hom(G_{K(A)_{M}}, B)^{G_{K}} \to \Hom(G_{K(A)_{M}}, B/(\tau-1)B)
$$
is also injective.

Now suppose $B=A$. For $1 \leq i \leq s$, let $F_{i}$ be the field corresponding to the subgroup $\ker({\rm Res}(c_{i}))$ of $G_{K(A)_M}$. 
Note that $F_{i}$ is a finite Galois extension over $K$ (this follows from the fact that ${\rm Rec}(c_i)$ is a $G_K$-homomorphism).
Let $\widetilde{c}_{i} \colon G_{K} \to A$ be a $1$-cocycle, which represents $c_{i}$, and set
$a_{i} := - \widetilde{c}_{i}(\tau) \in A/(\tau-1)A$.
By the definition of $1$-coboundary, note that $a_{i} \in A/(\tau-1)A$ is independent of the choice of the representative $\widetilde{c}_{i}$ of $c_{i}$.
Define $H_{i} \subseteq  G_{K(A)_{M}}$ by
$$
H_{i} = f({\rm Res}(c_{i}))^{-1}(a_{i}).
$$
Since $c_i$ is non-zero and both ${\rm Res}$ and $f$ are injective, we have
$$
[G_{K(A)_M} \colon \ker\left( f({\rm Res}(c_{i}))\right)] \geq p.
$$

Next, we suppose $B=A^\ast(1)$. For $1 \leq i \leq t$, define $F_{i}^{*}$, $a_{i}^{*}$, and $H_{i}^{*} \subseteq  G_{K(A)_M}$ similarly by using $c_{i}^{*}$ instead of $c_{i}$.

Since $s+t <p$ and $[G_{K(A)_M} \colon \ker\left( f({\rm Res}(c))\right)] \geq p$ for each $c \in \{ c_{1}, \ldots, c_{s}, c_{1}^{*}, \ldots, c_{t}^{*} \}$, we have
$$
G_{K(A)_M} \neq H_{1} \cup \cdots \cup H_{s} \cup H_{1}^{*} \cup \cdots \cup H^{*}_{t}.
$$
Now set $F := F_{1} \cdots F_{s}F_{1}^{*} \cdots F_{t}^{*}$ (this is a finite Galois extension of $K$) and define $\cQ$ to be the set of primes $\fq \not \in S $ of $K$ such that
$\fq$ is unramified in $F/K$ and
${\rm Fr}_{\fq}$ is conjugate to $\tau\gamma$ in $\Gal(F/K)$
for some $\gamma \in G_{K(A)_M} \setminus (H_{1} \cup \cdots \cup H_{s} \cup H_{1}^{*} \cup \cdots \cup H^{*}_{t})$.
Then we see that $\cQ \subseteq  \cP$ by construction, and that $\mathcal{Q}$ is of positive density by the Chebotarev density theorem.
If $\fq \in \cQ$, then for any $1 \leq i \leq s$ we have
$$
{\rm loc}_{\fq}(c_{i}) = \widetilde{c}_{i}(\tau\gamma) = \tau \widetilde{c}_{i}(\gamma) + \widetilde{c}_{i}(\tau)
= f({\rm Res}(c_{i}))(\gamma) - a_{i} \neq 0 \text{ in } A/(\tau - 1)A
$$
where we identify $H^{1}_{f}(K_{\fq}, A)=A/({\rm Fr}_{\fq} - 1)A = A/(\tau - 1)A$.
Similarly, we have ${\rm loc}_{\fq}(c_{i}^{*}) \neq 0$ for any $1 \leq i \leq t$ and $\fq \in \cQ$.
This completes the proof.
\end{proof}

\begin{lemma}\label{injective}
Assume Hypothesis \ref{hyp1}. Let $S$ be a self-injective ring and $R \to S$ a surjective ring homomorphism. Then for any free, finitely generated, $R$-submodule $X$ of $H^{1}(K, A)$ the natural homomorphism $X \otimes_{R} S \to H^{1}(K, A \otimes_{R}S)$ is injective.
\end{lemma}
\begin{proof}
We may assume $X \neq 0$. Let $\cQ$ be the set of all primes $\fq$ in $\cP$ such that ${\rm loc}_{\fq}(X) \not\subseteq   H^{1}_{f}(K_{\fq}, A)$. Note that $\cQ$ is a finite set.
Let $e_{1} \in X$ with $\Ann_{R}(e_{1}) = 0$ and $x$ a generator of $R[\fp]$. (The assumption that $R$ is a zero-dimensional Gorenstein local ring ensures that $R[\fp]$ is a principal ideal.) Note that $a \in R$ is a unit if and only if $xa \neq 0$.
Since $xe_{1} \neq 0$, there is a prime $\fq_{1} \in \cP \setminus \cQ$ with ${\rm loc}_{\fq_{1}}(xe_{1}) \neq 0$ by Lemma \ref{chebotarev}.
Since $H^{1}_{f}(K_{\fq_{1}}, A)$ is a free $R$-module of rank $1$ and ${\rm loc}_{\fq_{1}}(xe_{1}) \neq 0$,
the composition $Re_{1} \to X \to H^{1}_{f}(K_{\fq_{1}}, A)$ is an isomorphism and we have
$$
X \simeq H^1_f(K_{\fq_1},A) \oplus \ker \left( {\rm loc}_{\fq_{1}}|_{X} \right) \simeq Re_{1} \oplus \ker\left( {\rm loc}_{\fq_{1}}|_{X} \right).
$$
In particular, the $R$-module $\ker\left( {\rm loc}_{\fq_{1}}|_{X} \right)$ is free.
Hence we can take an element $e_{2} \in \ker\left( {\rm loc}_{\fq_{1}}|_{X} \right)$ with $\Ann_{R}(e_{2}) = 0$.
Similarly, by Lemma \ref{chebotarev}, we get a prime $\fq_{2} \in \cP \setminus (\cQ \cup \{\fq_{1}\} )$ such that
the composition $Re_{2} \to X \to H^{1}_{f}(K_{\fq_{2}}, A)$ is an isomorphism.
Since $e_{2} \in \ker\left( {\rm loc}_{\fq_{1}}|_{X} \right)$, the composition
$$
Re_{1} \oplus Re_{2} \to X \to H^{1}_{f}(K_{\fq_{1}}, A) \oplus H^{1}_{f}(K_{\fq_{2}}, A)
$$
is an isomorphism and
$$
X = Re_{1} \oplus  Re_{2} \oplus
\left(\ker\left( {\rm loc}_{\fq_{1}}|_{X} \right) \cap \ker\left( {\rm loc}_{\fq_{2}}|_{X} \right) \right).
$$
By repeating this argument, we can find an ideal $\fm \in \cN(\cP \setminus \cQ)$ such that
the sum of localization maps
$$
X \to \bigoplus_{\fq \mid \fm}H^{1}_{f}(K_{\fq}, A)
$$
is an isomorphism. Since $H^{1}_{f}(K_{\fq}, A) \otimes_{R} S \simeq H^{1}_{f}(K_{\fq}, A \otimes_{R} S)$ for any prime $\fq \in \cP$, the map $X \otimes_{R} S \to  \bigoplus_{\fq \mid \fm}H^{1}_{f}(K_{\fq}, A\otimes_R S ) $ is an isomorphism. This implies that $X\otimes_R S \to H^{1}(K, A \otimes_{R}S)$ is injective.
\end{proof}

\section{Stark systems}\label{stark sys sec}
In this section, we review some basic results on Stark systems. We continue to use the notation introduced in the previous section.

\subsection{Definition} \label{defstark}
Let $r$ be a non-negative integer. We recall the definition of Stark systems of rank $r$. The module of Stark systems is defined by the inverse limit
$${\rm SS}_r(A,\cF):=\varprojlim_{\fn \in \cN} {{\bigcap}}_{R}^{r + \nu(\fn)} H_{\cF^\fn}^1(K, A), $$
where the transition maps
\begin{align*}
v_{\fm,\fn}: {{\bigcap}}_{R}^{r + \nu(\fm)} H^1_{\cF^\fm} (K, A) \to {{\bigcap}}_{R}^{r + \nu(\fn)} H_{\cF^\fn}^1 (K, A)
\end{align*}
($\fm, \fn\in \cN$, $\fn\mid \fm$) are defined as follows. For each $\fq \in \cP$, we fix an isomorphism $H^1_{/f}(K_\fq, A)\simeq R$ and let $v_\fq$ be the composition map
\begin{align*}
v_\fq \colon H_{\cF^\fm}^1(K,A) \stackrel{{\rm loc}_\fq}{\to} H^1(K_\fq, A) \to H_{/f}^1(K_\fq, A) \simeq R.
\end{align*}
Since we have the exact sequence
\begin{align*}
0 \to H_{\cF^\fn}^1(K, A) \to H_{\cF^\fm}^1(K,A)
\stackrel{\bigoplus_{\fq \mid \fm / \fn} v_\fq }{\to} R^{\nu(\fm / \fn)},
\end{align*}
we see that ${{\bigwedge}}_{\fq \mid \fm/\fn}v_\fq$ induces
\begin{align*}
{{\bigwedge}}_{\fq \mid \fm/\fn}v_\fq \colon {{\bigcap}}_R^{r+\nu(\fm)} H_{\cF^\fm}^1(K, A) \to
{{\bigcap}}_R^{r+\nu(\fn)} H_{\cF^\fn}^1(K, A),
\end{align*}
by Proposition~\ref{prop self injective}. We denote this map by $v_{\fm,\fn}$.
Note that $v_{\fm,\fn}=\bigwedge_{\fq \mid \fm/\fn}v_\fq$ is well-defined up to sign, but one can explicitly choose a sign for each pair $(\fm,\fn)$ so that one has $v_{\fm', \fn}=v_{\fm, \fn} \circ v_{\fm',\fm}$ when $\fn \mid \fm \mid \fm'$. (For an explicit choice of sign, see \cite[\S 3.1]{sbA}.) Thus the collection $\{v_{\fm,\fn}\}_{\fm,\fn}$ forms an inverse system.

By definition, a Stark system of rank $r$ (for ($A,\cF$)) is an element
\begin{align*}
(\epsilon_\fn)_\fn \in \prod_{\fn\in\cN} {{\bigcap}}_{R}^{r + \nu(\fn)} H_{\cF^\fn}^1(K, A)
\end{align*}
that satisfies $v_{\fm,\fn}(\epsilon_\fm)=\epsilon_{\fn}$ for all ideals $\fm$ and $\fn$ in $\cN$ with $\fn \mid \fm$.

\subsection{Stark systems and Selmer modules}

We can now define the key invariants that are associated to a Stark system.

In this definition we use the fact that each element of ${\bigcap}_R^{r+\nu(\fn)} H^1_{\cF^\fn}(K,A)$ is, by its very definition, a homomorphism of $R$-modules from ${\bigwedge}_R^{r+\nu(\fn)}H^1_{\cF^\fn}(K, A)^*$ to $R$.

\begin{definition}\label{stark ideal}
Let $\epsilon = ( \epsilon_\fn )_\fn \in {\rm SS}_r(A, \cF)$. Then for each non-negative integer $i$ we define an ideal $I_i(\epsilon)$ of $R$ by setting
\begin{align*}
I_i(\epsilon) := \sum_{\fn \in \cN,\ \nu(\fn) = i} {\rm im}(\epsilon_\fn).
\end{align*}
\end{definition}

We consider the following hypothesis.

\begin{hypothesis} \label{hyp large}
There exists an ideal $\fn$ in $\cN$ such that $H^1_{(\cF^\ast)_\fn}(K,A^\ast(1))$ vanishes and $H^1_{\cF^\fn}(K, A)$ is a free $R$-module of rank $r +\nu(\fn)$.
\end{hypothesis}

\begin{remark}\label{free rem}
The following observation will frequently be used: if one assumes Hypothesis~\ref{hyp large}, then, for any ideal $\fm$ in $\cN$ for which the group $H^1_{(\cF^\ast)_\fm}(K,A^\ast(1))$ vanishes, the $R$-module $H^1_{\cF^\fm}(K, A)$ is free of rank $r +\nu(\fm)$.
In fact, let $\fn \in \cN$ be an ideal as in Hypothesis~\ref{hyp large} and $\fm \in \cN$ with $H^1_{(\cF^\ast)_\fm}(K,A^\ast(1))=0$.
Take an ideal $\fd \in \cN$ such that $\fm \mid \fd$ and $\fn \mid \fd$. Then, by the global duality, we have exact sequences
$$
0 \to H^{1}_{\cF^{\fn}}(K, A) \to H^{1}_{\cF^{\fd}}(K, A) \to \bigoplus_{\fq \mid \frac{\fd}{\fn}}H^{1}_{/f}(K_{\fq}, A) \to 0
$$
and
$$
0 \to H^{1}_{\cF^{\fm}}(K, A) \to H^{1}_{\cF^{\fd}}(K, A) \to \bigoplus_{\fq \mid \frac{\fd}{\fm}}H^{1}_{/f}(K_{\fq}, A) \to 0.
$$
By Hypothesis~\ref{hyp large} and the definition of $\cP$, we have $H^{1}_{\cF^{\fn}}(K, A) \simeq R^{r+\nu(\fn)}$ and $H^{1}_{/f}(K_{\fq}, A) \simeq R$ for any $\fq \in \cP$, and so $H^{1}_{\cF^{\fd}}(K, A) \simeq R^{r + \nu(\fd)}$. Since $R$ is a local ring, a projective $R$-module is free. Hence
we conclude that $H^{1}_{\cF^{\fm}}(K, A) \simeq R^{r+\nu(\fm)}$ by using the second (split) exact sequence.
\end{remark}


\begin{lemma}\label{fitt-lemma}
Assume Hypothesis \ref{hyp large}.
Let $i$ be a positive integer and $\fn \in \cN$ with $\nu(\fn) \geq i$. Suppose that $H^{1}_{(\cF^{*})_{\fn}}(K, A^{*}(1))$ vanishes.
Then we have
\begin{align*}
\Fitt_{R}^{i}\left(H^{1}_{\cF^{*}}(K, A^{*}(1))^{*}\right)
&= \sum_{ \fq \in \cP, \ \fq \mid \fn} \Fitt_{R}^{i-1}\left(H^{1}_{(\cF^{*})_{\fq}}(K, A^{*}(1))^{*}\right)
\\
&= \sum_{\fq \in \cP} \Fitt_{R}^{i-1}\left(H^{1}_{(\cF^{*})_{\fq}}(K, A^{*}(1))^{*}\right).
\end{align*}
\end{lemma}
\begin{proof}
By Hypothesis \ref{hyp large}, we see that, for any prime $\fq \in \cP$, there is an ideal $\fn \in \cN$ with $\fq \mid \fn$ such that $\nu(\fn)\geq i$ and
$H^{1}_{(\cF^{*})_{\fn}}(K, A^{*}(1)) = 0$. Hence the first equality implies the second equality.
To show the first equality, take an ideal $\fn \in \cN$ with $\nu(\fn) \geq i$ such that $H^{1}_{(\cF^{*})_{\fn}}(K, A^{*}(1)) = 0$.
Note that $H^{1}_{\cF^{{\fn}}}(K, A) \simeq R^{r+\nu(\fn)}$ by Remark~\ref{free rem}.
Then, by the global duality, we have exact sequences
$$
H^{1}_{\cF^{\fn}}(K, A) \to \bigoplus_{\fq \mid \fn}H^{1}_{/f}(K_{\fq}, A) \to H^{1}_{\cF^{*}}(K, A^{*}(1))^{*} \to 0
$$
and
$$
H^{1}_{\cF^{\fn}}(K, A) \to \bigoplus_{\fq \mid \frac{\fn}{\fr}}H^{1}_{/f}(K_{\fq}, A) \to H^{1}_{(\cF^{*})_{\fr}}(K, A^{*}(1))^{*} \to 0
$$
for any prime $\fr \mid \fn$.
Since $H^{1}_{\cF^{\fn}}(K, A)  \simeq R^{r + \nu(\fn)}$ and $H^{1}_{/f}(K_{\fq}, A) \simeq R$ for any prime $\fq \in \cP$,
these sequences give finite presentations of $H^{1}_{\cF^{*}}(K, A^{*}(1))^{*}$ and $H^{1}_{(\cF^{*})_{\fr}}(K, A^{*}(1))^{*}$ respectively.
Hence we see that
$$
\Fitt_{R}^{i}\left(H^{1}_{\cF^{*}}(K, A^{*}(1))^{*}\right)
= \sum_{ \fq \mid \fn} \Fitt_{R}^{i-1}\left(H^{1}_{(\cF^{*})_{\fq}}(K, A^{*}(1))^{*}\right).
$$
\end{proof}

Using Lemma \ref{fitt-lemma} repeatedly, we deduce the following 

\begin{corollary}
Assume Hypothesis~\ref{hyp large}.
Then, for any non-negative integer $i$, we have
\begin{align*}
\Fitt_{R}^{i}\left(H^{1}_{\cF^{*}}(K, A^{*}(1))^{*}\right)
= \sum_{\fm \in \cN, \ \nu(\fm) = i} \Fitt_{R}^{0}\left(H^{1}_{(\cF^{*})_{\fm}}(K, A^{*}(1))^{*}\right).
\end{align*}
\end{corollary}

We can now state one of the main results in the theory of Stark systems.

\begin{theorem}\label{thm stark} Under Hypothesis \ref{hyp large} all of the following claims are valid.
\begin{itemize}
\item[(i)]
Let $\fn \in \cN$ with $H^{1}_{(\cF^{*})_{\fn}}(K, A^{*}(1)) = 0$. Then the natural projection homomorphism
$${\rm SS}_{r}(A, \cF) \to {\bigcap}^{r+\nu(\fn)}_{R}H^{1}_{\cF^{\fn}}(K, A); \ \epsilon \mapsto \epsilon_\fn$$ is bijective. In particular, the $R$-module ${\rm SS}_r(A,\cF)$ is free of rank one.
\item[(ii)] The following claims are valid for all $\epsilon$ in ${\rm SS}_{r}(A, \cF)$ and all $i \ge 0$.
\begin{itemize}
\item[(a)] $I_i(\epsilon)\subseteq   I_{i+1}(\epsilon)$, with equality for all sufficiently large $i$.
\item[(b)] $I_\infty(\epsilon):= \bigcup_{i \ge 0}I_i(\epsilon)$ is equal to $R$ if and only if $\epsilon$ is a basis of ${\rm SS}_{r}(A, \cF)$.
\item[(c)] $I_i(\epsilon) = I_\infty(\epsilon)\cdot {\rm Fitt}_{R}^i \left( H^1_{\cF^\ast}(K, A^*(1))^* \right).$
\end{itemize}
\end{itemize}
\end{theorem}

\begin{proof} Both claim (i) and, in the case that $\epsilon$ is a basis of ${\rm SS}_{r}(A, \cF)$, the equality $I_i(\epsilon) = {\rm Fitt}_{R}^i \left( H^1_{\cF^\ast}(K, A^*(1))^* \right)$ in claim (ii)(c) were independently proved by the first and third authors in \cite[Th. 3.17 and 3.19(ii)]{sbA} and by the second author in \cite[Th. 4.7 and 4.10]{sakamoto}.

To deduce the remainder of claim (ii) we fix a basis $\epsilon_0$ of ${\rm SS}_{r}(A, \cF)$ and then for each $\epsilon$ in ${\rm SS}_{r}(A, \cF)$ define $\lambda_\epsilon \in R$ by the equality $\epsilon = \lambda_\epsilon\cdot \epsilon_0$. Then for each $i \ge 0$ one has
\begin{equation}\label{gen case} I_i(\epsilon) = \lambda_\epsilon\cdot I_i(\epsilon_0) =  \lambda_\epsilon\cdot {\rm Fitt}_{R}^i \left( H^1_{\cF^\ast}(K, A^*(1))^* \right).\end{equation}

Claim (ii)(a) is thus true since the definition of higher Fitting ideal implies both that
\[ {\rm Fitt}_{R}^i \left( H^1_{\cF^\ast}(K, A^*(1))^* \right) \subseteq   {\rm Fitt}_{R}^{i+1} \left( H^1_{\cF^\ast}(K, A^*(1))^* \right)\]
and ${\rm Fitt}_{R}^i \left( H^1_{\cF^\ast}(K, A^*(1))^* \right) = R$ for all sufficiently large $i$.

The latter fact also combines with (\ref{gen case}) to imply $I_i(\epsilon) = (\lambda_\epsilon)$ for all sufficiently large $i$, and hence that $I_\infty(\epsilon) = (\lambda_\epsilon)$. This equality implies claim (ii)(b) directly and also shows that (\ref{gen case}) implies claim (ii)(c).
\end{proof}






\subsection{Stark Systems over Gorenstein orders}\label{one-dim case}
Let $Q$ be a finite extension of $\bQ_{p}$ and $\cO$ the ring of integers of $Q$. Let $\cQ$ be a finite-dimensional semisimple commutative $Q$-algebra. Let $(\cR, \fp)$ be a local Gorenstein $\cO$-order
in $\cQ$ (for basic properties of Gorenstein orders, see \cite[\S A.3]{sbA}).
Note that $\cR/(p^{m})$ is a zero-dimensional Gorenstein local ring since $p$ is a regular element of $\cR$.

Let $T$ be a free $\cR$-module of finite rank with an $\cR$-linear continuous action of $G_{K}$ which is unramified outside a finite set of places of $K$. Let $K(T)$ denote the minimal Galois extension of $K$ such that $G_{K(T)}$ acts trivially on $T$.
Recall that $K_{p^{m}} := K (\mu_{p^{m}}, (\cO^\times_K)^{1/p^{m}}) K(1)$ and $K(T)_{p^{m}} := K(T)K_{p^{m}}$ for any positive integer $m$.
Set $K(T)_{p^{\infty}} := \bigcup_{m >0}K(T)_{p^{m}}$ and $\overline{T} := T/\fp T$.

For an $\cR$-module $X$ we endow the Pontryagin dual $X^\vee := \Hom_\cO(X,Q/\cO)$ with the natural action of $\cR$.
For any positive integer $m$ and any $\cR/(p^m)$-module $X$, the module $X^\vee $ is naturally isomorphic to $X^\ast:=\Hom_{\cR/(p^m)}(X, \cR/(p^m))$ and in such cases we often identify the functors $(-)^\vee$ and $(-)^\ast$.

In this subsection, we assume the following hypothesis:
\begin{hypothesis}\label{hyp1'}\
\begin{itemize}
\item[(i)] $\overline{T}$ is an irreducible $(\cR/\fp)[G_{K}]$-module;
\item[(ii)] there exists $\tau \in G_{K(T)_{p^{\infty}}}$ such that $T/(\tau-1)T \simeq \cR$ as $\cR$-modules;
\item[(iii)] $H^{1}(K(T)_{p^{{\infty}}}/K, \overline{T}) = H^{1}(K(T)_{p^{{\infty}}}/K, \overline{T}^{\vee}(1)) = 0$;
\end{itemize}
\end{hypothesis}

\begin{remark}\label{hyp(iv)} Hypothesis~\ref{hyp1'} implies that $\overline{T}^{G_{K}}$ and $\overline{T}^{\vee}(1)^{G_{K}}$ both vanish. To see this note that if $\overline{T}^{G_{K}}$ does not vanish, then Hypothesis~\ref{hyp1'}(i) implies $\dim_{\Bbbk}(\overline{T}) = 1$ and hence that $\overline{T}$ is the trivial $G_{K}$-representation. But then, in this case, the module
$$
H^{1}(K(T)_{p^{{\infty}}}/K, \overline{T}) = \Hom\left(\Gal(K(T)_{p^{{\infty}}}/K), \overline{T}\right)
$$
does not vanish since there is a non-trivial $p$-subextension of $K(T)_{p^{{\infty}}}/K$ and this contradicts Hypothesis~\ref{hyp1'}(iii).
The vanishing of $\overline{T}^{\vee}(1)^{G_{K}}$ is proved by a similar argument.\end{remark}

\begin{remark}
If we assume Hypothesis \ref{hyp1'}, then $T/p^{m}T$ satisfies Hypotheses \ref{hyp1} and \ref{hyp2} for any positive integer $m$. In fact, it clearly satisfies Hypotheses \ref{hyp1}(i) and (ii). By Remark~\ref{hyp(iv)}, Hypothesis~\ref{hyp2} also holds true. We shall show Hypothesis \ref{hyp1}(iii), i.e. that, setting $A := T/p^{m}T$, one has
$$
H^{1}(K(A)_{p^{m}}/K, A) = H^{1}(K(A)_{p^{m}}/K, A^{*}(1)) = 0.
$$

Since the inflation map $H^{1}(K(A)_{p^{m}}/K, \overline{T}) \to H^{1}(K(T)_{p^{\infty}}/K, \overline{T})$ is injective, we have $H^{1}(K(A)_{p^{m}}/K, \overline{T}) = 0$ by Hypothesis \ref{hyp1'}(iii). Let $i$ be a non-negative integer. By the exact sequence $0 \to \fp^{i+1}A \to \fp^{i}A \to \fp^{i}A/\fp^{i+1}A \to 0$ and Remark~\ref{hyp(iv)}
we have an exact sequence
$$
0 \to H^{1}(K(A)_{p^{m}}/K, \fp^{i+1}A) \to H^{1}(K(A)_{p^{m}}/K, \fp^{i}A) \to  H^{1}(K(A)_{p^{m}}/K, \fp^{i}A/\fp^{i+1}A).
$$
Since the $(\cR/\fp)[G_{K}]$-module $\fp^{i}A/\fp^{i+1}A$ is isomorphic to a direct sum of $\overline{T}$, the group $H^{1}(K(A)_{p^{m}}/K, \fp^{i}A/\fp^{i+1}A)$ vanishes and so the natural map
\[ H^{1}(K(A)_{p^{m}}/K, \fp^{i+1}A) \to H^{1}(K(A)_{p^{m}}/K, \fp^{i}A)\]
is bijective. Since $i$ is arbitrary and $\fp^{m}A = 0$, we conclude that $H^{1}(K(A)_{p^{m}}/K, A)$ vanishes. The vanishing of $H^{1}(K(A)_{p^{m}}/K, A^{*}(1))$ is proved by a similar argument.
\end{remark}

We fix a Selmer structure $\cF$ on $T$.
For a positive integer $m$, let $\mathcal{P}_{m}$ denote the set of primes $\fq \not\in S(\cF)$ of $K$ such that
${\rm Fr}_\fq$ is conjugate to $\tau$ in $\Gal(K(T/p^{m}T)_{p^{m}}/K)$.
Put $\cN_{m}=\cN(\cP_{m})$. Note that $\cN_{m+1} \subseteq   \cN_{m}$ for any positive integer $m$.

We suppose that $(T/p^{m}T, \cF, \cP_{m})$ satisfies Hypothesis~\ref{hyp large} for any positive integer $m$.
Let $m$ be a positive integer and let $\fm$ and $\fn$ be ideals of $\cN_{m+1}$ such that $\fm \mid \fn$ and $H^{1}_{(\cF^{*})_{\fn}}(K, (T/p^{m+1}T)^{\vee}(1))$ vanishes.
Then $H^{1}_{(\cF^{*})_{\fn}}(K, (T/p^{m}T)^{\vee}(1))$ vanishes by Corollary~\ref{dualselisom} and so
$H^{1}_{\cF^{\fn}}(K, T/p^{m+1}T) \otimes_{\cR/(p^{m+1})} \cR/(p^{m})$ and $H^{1}_{\cF^{\fn}}(K, T/p^{m}T)$ are free $\cR/(p^{m})$-modules of the same rank by Remark~\ref{free rem}.
Hence by Lemma~\ref{injective}, the natural homomorphism
$$
H^{1}_{\cF^{\fn}}(K, T/p^{m+1}T) \otimes_{\cR/(p^{m+1})} \cR/(p^{m}) \to H^{1}_{\cF^{\fn}}(K, T/p^{m}T)
$$
is an isomorphism. Hence there is a canonical injection
$$
H^{1}_{\cF^{\fm}}(K, T/p^{m}T) \to
H^{1}_{\cF^{\fn}}(K, T/p^{m+1}T) \otimes_{\cR/(p^{m+1})} \cR/(p^{m}).
$$
Applying Corollary~\ref{morph} with $X = H^{1}_{\cF^{\fm}}(K, T/p^{m+1}T)$, $Y = H^{1}_{\cF^{\fm}}(K, T/p^{m}T)$, and $F = H^{1}_{\cF^{\fn}}(K, T/p^{m+1}T)$, we get a natural homomorphism
$$
{\bigcap}^{r+\nu(\fm)}_{\cR/(p^{m+1})}H^{1}_{\cF^{\fm}}(K, T/p^{m+1}T) \to {\bigcap}^{r+\nu(\fm)}_{\cR/(p^{m})}H^{1}_{\cF^{\fm}}(K, T/p^{m}T).
$$
Since $\fm$ is any element of $\cN_{m+1}$, we get a homomorphism
$$
{\rm SS}_{r}(T/p^{m+1}T, \cF) \to {\rm SS}_{r}(T/p^{m}T, \cF).
$$

\begin{lemma}\label{compatible}
Let $m$ be a positive integer.
\begin{itemize}
\item[(i)] The map ${\rm SS}_{r}(T/p^{m+1}T, \cF) \to {\rm SS}_{r}(T/p^{m}T, \cF)$ is surjective.
\item[(ii)] Fix $\epsilon^{(m+1)}$ in ${\rm SS}_{r}(T/p^{m+1}T, \cF)$ and write $\epsilon^{(m)}$ for its image in ${\rm SS}_{r}(T/p^{m}T, \cF)$. Then for each non-negative integer $i$ one has $I_{i}(\epsilon^{(m+1)})\cR/(p^m) = I_{i}(\epsilon^{(m)})$.
\end{itemize}
\end{lemma}
\begin{proof}
Take an ideal $\fn \in \cN_{m+1}$ such that $H^{1}_{(\cF^{*})_{\fn}}(K, (T/p^{m+1}T)^{\vee}(1))$ vanishes. Then we have the following commutative diagram:
$$
\xymatrix{
{\rm SS}_{r}(T/p^{m+1}T, \cF) \ar[d] \ar[r] & {{\bigcap}}^{r+\nu(\fn)}_{\cR/(p^{m+1})}H^{1}_{\cF^{\fn}}(K, T/p^{m+1}T) \ar[d]
\\
{\rm SS}_{r}(T/p^{m}T, \cF) \ar[r] & {{\bigcap}}^{r+\nu(\fn)}_{\cR/(p^{m})}H^{1}_{\cF^{\fn}}(K, T/p^{m}T).
}
$$
The horizontal maps are isomorphisms by Theorem~\ref{thm stark}(i) and the right vertical map is surjective by the commutativity of the diagram in Corollary~\ref{morph} and that
$$
H^{1}_{\cF^{\fn}}(K, T/p^{m+1}T) \otimes_{\cR/(p^{m+1})} \cR/(p^{m}) \simeq H^{1}_{\cF^{\fn}}(K, T/p^{m}T) \simeq (\cR/(p^{m}))^{r+\nu(\fn)}.
$$
Hence the map ${\rm SS}_{r}(T/p^{m+1}T, \cF) \to {\rm SS}_{r}(T/p^{m}T, \cF)$ is surjective.

We will show claim (ii). We may assume that $\epsilon^{(m+1)}$ is a basis of ${\rm SS}_{r}(T/p^{m+1}T, \cF)$.
By claim (i), $\epsilon^{(m)}$ is also a basis of ${\rm SS}_{r}(T/p^{m}T, \cF)$.
Then we have
\begin{align*}
I_{i}(\epsilon^{(m+1)})\cR/(p^m) &= {\rm Fitt}_{\cR/(p^{m+1})}^{i}\left(H^{1}_{\cF^{*}}(K, (T/p^{m+1}T)^{\vee}(1))^{\vee}\right)\cR/(p^m)
\\
&= {\rm Fitt}_{\cR/(p^{m})}^{i}\left( \left( H^{1}_{\cF^{*}}(K, (T/p^{m+1}T)^{\vee}(1))[p^{m}]\right)^{\vee} \right)
\\
&= {\rm Fitt}_{\cR/(p^{m})}^{i}\left( H^{1}_{\cF^{*}}(K, (T/p^{m}T)^{\vee}(1))^{\vee}\right)
\\
&= I_{i}(\epsilon^{(m)})
\end{align*}
where the first and forth equality follows from Theorem~\ref{thm stark}(ii) and the third equality follows from Corollary~\ref{dualselisom}.
\end{proof}

\begin{definition}\label{stark ideal gorenstein}
We define the module ${\rm SS}_{r}(T, \cF)$ of Stark systems of rank $r$ for $(T, \cF)$ to be the inverse limit
$$
{\rm SS}_{r}(T, \cF) := \varprojlim_{m \in \bZ_{>0}}{\rm SS}_{r}(T/p^{m}T, \cF).
$$
Let $i$ be non-negative integer and $\epsilon = (\epsilon^{(m)})_{m} \in {\rm SS}_{r}(T, \cF)$. Then, by Lemma~\ref{compatible}(ii), we see that the family $(I_{i}(\epsilon^{(m)}))_{m}$ is an inverse system.
We define an ideal $I_{i}(\epsilon)$ of $\cR$ to be the inverse limit
$$
I_{i}(\epsilon) := \varprojlim_{m} I_{i}(\epsilon^{(m)}).
$$
\end{definition}


\begin{theorem}\label{thm stark'}\
\begin{itemize}
\item[(i)] The $\cR$-module ${\rm SS}_{r}(T, \cF)$ is free of rank one.
\item[(ii)] The following claims are valid for all $\epsilon$ in ${\rm SS}_{r}(T,\cF)$ and all $i \ge 0$.
\begin{itemize}
\item[(a)] $I_i(\epsilon)\subseteq   I_{i+1}(\epsilon)$, with equality for all sufficiently large $i$.
\item[(b)] $I_\infty(\epsilon):= \bigcup_{i \ge 0}I_i(\epsilon)$ is equal to $\cR$ if and only if $\epsilon$ is a basis of ${\rm SS}_{r}(A, \cF)$.
\item[(c)] $I_i(\epsilon) = I_\infty(\epsilon)\cdot {\rm Fitt}_{\cR}^i \left( H^1_{\cF^\ast}(K, T^{\vee}(1))^{\vee}\right).$
\end{itemize}
\end{itemize}
\end{theorem}

\begin{proof} Claim (i) follows directly from Lemma~\ref{compatible}(i) and Theorem~\ref{thm stark}(i).

To prove claim (ii) it is enough, just as with the proof of Theorem~\ref{thm stark}(ii), to show that if $\epsilon = (\epsilon^{(m)})_{m}$ is a basis of ${\rm SS}_{r}(T, \cF)$, then for each non-negative integer $i$ one has $
I_{i}(\epsilon) = \Fitt_{\cR}^{i}(H^{1}_{\cF^{*}}(K, T^{\vee}(1))^{\vee}).$

In this case, each element $\epsilon^{(m)}$ is a basis of ${\rm SS}_{r}(T/(p^m), \cF)$ over $\cR/(p^m)$ and so Theorem~\ref{thm stark}(ii) implies that
\begin{align*}
I_{i}(\epsilon) &= \varprojlim_{m} \Fitt_{\cR/(p^{m})}^{i}\left(H^{1}_{\cF^{*}}(K, (T/p^{m}T)^{\vee}(1))^{\vee}\right)
\\
&= \varprojlim_{m} \Fitt_{\cR}^{i}\left(H^{1}_{\cF^{*}}(K, T^{\vee}(1))^{\vee}\right)\cR/(p^{m})\\
&= \Fitt_{\cR}^{i}\left(H^{1}_{\cF^{*}}(K, T^{\vee}(1))^{\vee}\right),
\end{align*}
where the second equality follows from Corollary~\ref{dualselisom} and the last from the fact that $\cR$ is a complete noetherian ring and so every ideal is closed. 
\end{proof}

\section{Kolyvagin systems}\label{koly sys sec}

In this section, we develop the theory of higher rank Kolyvagin systems. We continue to use the notation introduced in \S \ref{pre}.

\subsection{Definition}\label{defkoly}

In the following, we suppose $r>0$. (We do not define Kolyvagin systems of rank zero).

We recall the definition of Kolyvagin systems. As in \S\ref{defstark}, we fix an isomorphism $H_{/f}^1(K_\fq,A)\simeq R$ for each $\fq \in \cP$ and identify them. We will again use the map $v_\fq$, which is defined by
$$v_\fq : H^1_{\cF(\fn)}(K,A) \stackrel{{\rm loc}_\fq}{\to} H^1(K_\fq, A) \to H^1_{/f}(K_\fq,A)=R.$$
If $\fq \mid \fn$, this map induces
$$
v_\fq: {\bigcap}_R^r H^1_{\cF(\fn)}(K,A)\otimes G_\fn \to {\bigcap}_R^{r-1}H^1_{\cF_\fq(\fn/\fq)}(K,A)\otimes G_\fn.
$$
(For the definition of $G_\fn$, see \S \ref{section hyp}.)

The finite-singular comparison map
$$
\varphi_\fq^{\rm fs}: H^1_{\cF(\fn/\fq)}(K,A) \stackrel{{\rm loc}_\fq}{\to} H_f^1(K_\fq, A) \stackrel{\varphi_\fq^{\rm fs}}{\to} H^1_{\rm tr}(K_\fq,A)\otimes G_\fq =H^1_{/f}(K_\fq,A)\otimes G_\fq=R\otimes G_\fq,
$$
which is defined in \S \ref{section fs}, induces
$$
\varphi_\fq^{\rm fs} :{\bigcap}_R^r H^1_{\cF(\fn/\fq)}(K,A)\otimes G_{\fn/\fq} \to {\bigcap}_R^{r-1}H^1_{\cF_\fq(\fn/\fq)}(K,A)\otimes G_\fn.
$$
A Kolyvagin system of rank $r$ (for $(A,\cF)$) is an element
$$(\kappa_\fn)_\fn \in \prod_{\fn \in \cN}{\bigcap}_R^r H^1_{\cF(\fn)}(K,A)\otimes G_\fn$$
which satisfies the `finite-singular relation'
$$
v_\fq(\kappa_\fn)=\varphi_\fq^{\rm fs}(\kappa_{\fn/\fq}) \ \text{ in } \ {\bigcap}_R^{r-1}H^1_{\cF_\fq(\fn/\fq)}(K,A)\otimes G_\fn.
$$
The set of all Kolyvagin systems of rank $r$ is denoted by ${\rm KS}_r(A,\cF)$. This is an $R$-submodule of $\prod_{\fn \in \cN}{\bigcap}_R^r H^1_{\cF(\fn)}(K,A)\otimes G_\fn$.

\subsection{Regulator maps} \label{sec regulator}

We quickly review the relation between Kolyvagin and Stark systems (see \cite[\S 4.2]{sbA}).

There exists a canonical homomorphism of $R$-modules
$${\rm Reg}_r: {\rm SS}_r(A,\cF) \to {\rm KS}_r(A,\cF),$$
which is referred to in loc. cit. as a `regulator map'. The definition of this map is as follows. Let $\epsilon=(\epsilon_\fn)_\fn \in {\rm SS}_r(A,\cF)$. For each $\fn \in \cN$, we define
$$\kappa(\epsilon_\fn):=\left({\bigwedge}_{\fq \mid \fn}\varphi_\fq^{\rm fs} \right)(\epsilon_\fn)\in {\bigcap}_R^r H^1_{\cF(\fn)}(K,A)\otimes G_\fn,$$
where we regard ${\bigwedge}_{\fq \mid \fn}\varphi_\fq^{\rm fs} $ as a map
$${\bigwedge}_{\fq \mid \fn}\varphi_\fq^{\rm fs} : {\bigcap}_R^{r+\nu(\fn)} H^1_{\cF^\fn}(K,A) \to {\bigcap}_R^r H^1_{\cF(\fn)}(K,A)\otimes G_\fn.$$
Then one sees that $(\kappa(\epsilon_\fn))_\fn$ is a Kolyvagin system. In fact, if $\fq \mid \fn$, then we have
\begin{eqnarray*}
v_\fq(\kappa(\epsilon_\fn))&=&v_\fq\left(\left({\bigwedge}_{\fq' \mid \fn}\varphi_{\fq'}^{\rm fs} \right)(\epsilon_\fn)\right) \\
&=&\varphi_\fq^{\rm fs}\left( \left({\bigwedge}_{\fq' \mid \fn/\fq}\varphi_{\fq'}^{\rm fs} \right)(v_\fq(\epsilon_\fn))\right) \\
&=& \varphi_\fq^{\rm fs}\left( \left({\bigwedge}_{\fq' \mid \fn/\fq}\varphi_{\fq'}^{\rm fs} \right)(\epsilon_{\fn/\fq})\right) \\
&=& \varphi_\fq^{\rm fs}(\kappa(\epsilon_{\fn/\fq})).
\end{eqnarray*}
The regulator map is defined by setting ${\rm Reg}_r(\epsilon):=(\kappa(\epsilon_\fn))_\fn.$

\subsection{Kolyvagin systems and Selmer modules}

For each $\kappa \in {\rm KS}_r(A,\cF)$, we can associate it with invariants $I_i(\kappa)$, similarly to the case of Stark systems (see Definition \ref{stark ideal}).

\begin{definition}\label{koly ideal}
Let $\kappa \in {\rm KS}_{r}(A, \cF)$. We fix a generator of $G_\fq$ for each $\fq \in \cP$ and regard
$$
{\bigcap}_R^r H^1_{\cF(\fn)}(K,A)\otimes G_\fn={\bigcap}_R^r H^1_{\cF(\fn)}(K,A)
$$
for every $\fn \in \cN$.
For a non-negative integer $i$, define an ideal $I_i(\kappa)$ of $R$ by
\begin{align*}
I_i(\kappa) := \sum_{\fn \in \cN, \ \nu(\fn) = i}{\rm im}(\kappa_\fn),
\end{align*}
where we regard $\kappa_\fn \in {\bigcap}_R^r H^1_{\cF(\fn)}(K,A) =\Hom_R \left( {\bigwedge}^r_R H^1_{\cF(\fn)}(K, A)^*, R \right)$.
These ideals, of course, do not depend on the choice of a generator of $G_\fq$ for each $\fq$.
\end{definition}

The aim of this subsection is to prove the following theorem, which is one of the main results in this paper.

\begin{theorem} \label{main}
Assume Hypotheses \ref{hyp1}, \ref{hyp2} and \ref{hyp large}, and also suppose $p > 3$.
\begin{itemize}
\item[(i)] The regulator map
\begin{align*}
{\rm Reg}_r \colon {\rm SS}_r(A, \cF) \to {\rm KS}_r(A, \cF)
\end{align*}
is an isomorphism. In particular, the $R$-module ${\rm KS}_r(A,\cF)$ is free of rank one.
\item[(ii)] Fix $\kappa$ in ${\rm KS}_r(A,\cF)$. Then for each ideal $\fn$ in $\cN$ one has
\begin{align*}
{\rm im}(\kappa_\fn) \subseteq   {\rm Fitt}_R^0 ( H^1_{\cF(\fn)^*}(K, A^*(1))^*),
\end{align*}
with equality if $\kappa$ is a basis of ${\rm KS}_r(A, \cF)$. In particular, if $\kappa$ is a basis, then
\begin{align*}
{\rm im}(\kappa_1) = {\rm Fitt}_R^0( H^1_{\cF^*}(K, A^*(1))^* ).
\end{align*}
\item[(iii)] Fix $\kappa$ in ${\rm KS}_r(A,\cF)$. Then for each non-negative integer $i$ one has
$$I_i(\kappa)\subseteq   {\rm Fitt}_R^i(H^1_{\cF^\ast}(K,A^\ast(1))^\ast),$$
with equality if $R$ is a principal ideal ring and $\kappa$ is a basis of ${\rm KS}_r(A,\cF)$.
\end{itemize}
\end{theorem}

\begin{remark}\label{difficulty remark} For each $\kappa$ in ${\rm KS}_r(A,\cF)$ and each non-negative integer $i$, Theorem~\ref{main}(i) allows us to define a canonical ideal of $R$ by setting
\[ I'_i(\kappa) := I_i({\rm Reg}_r^{-1}(\kappa)),\]
where the right hand side is as defined in Definition~\ref{stark ideal}. If $R$ is a principal ideal ring and $\kappa$ is a basis of ${\rm KS}_r(A,\cF)$, then Theorems~\ref{thm stark}(ii) and \ref{main}(iii) combine to imply that $I'_i(\kappa) = I_i(\kappa)$ for all $i$. It would be interesting to know if such an equality is true more generally but this question seems to be difficult. \end{remark}

The next subsection is devoted to the proof of Theorem \ref{main}.

\subsection{The proof of Theorem \ref{main}}



In this subsection, we always assume Hypotheses \ref{hyp1}, \ref{hyp2}, and \ref{hyp large}, and fix an injection $\Bbbk \hookrightarrow R$.
Note that
$$
H^{1}_{?}(K_{\fq}, A) \otimes_{R} \Bbbk \xrightarrow{\sim} H^{1}_{?}(K_{\fq}, A \otimes_{R} \Bbbk)
$$
and
$$
H^{1}_{?}(K_{\fq}, A \otimes_{R} \Bbbk) \xrightarrow{\sim} H^{1}_{?}(K_{\fq}, A)[\fp]
$$
where $\fq \in \cP$, $? \in \{\emptyset, f, /f, {\rm tr}, /{\rm tr} \}$,
the first map is the natural map, and the second is induced by the fixed map $\Bbbk \hookrightarrow R$.

\begin{lemma}\label{s-iso}
If $\fn \in \cN$ is an ideal such that $H^{1}_{(\cF^{*})_{\fn}}(K, A^{*}(1)) = 0$, then the natural map
$$
H^{1}_{\cF^{\fn}}(K, A) \otimes_{R} \Bbbk \to H^{1}_{\cF^{\fn}}(K, A \otimes_{R} \Bbbk)
$$
and the map
$$
H^{1}_{\cF^{\fn}}(K, A \otimes_{R} \Bbbk) \to H^{1}_{\cF^{\fn}}(K, A)[\fp]
$$
induced by $\Bbbk \hookrightarrow R$ are isomorphisms.
\end{lemma}
\begin{proof}
Note that $H^{1}_{\cF^{\fn}}(K, A) \simeq R^{r+\nu(\fn)}$ by Hypothesis \ref{hyp large}.
Applying Lemma~\ref{injective} with $S = \Bbbk$ and $X = H^{1}_{\cF^{\fn}}(K, A)$, we see that the natural map
$$
H^{1}_{\cF^{\fn}}(K, A) \otimes_{R} \Bbbk \to H^{1}_{\cF^{\fn}}(K, A \otimes_{R} \Bbbk)
$$
is injective.
By Hypothesis~\ref{hyp2} and Lemma~\ref{lemma S}, we have $\left(\coker\left(A \otimes_{R} \Bbbk \to A\right)\right)^{G_{K}} = 0$, and so
the map
$$
H^{1}_{\cF^{\fn}}(K, A \otimes_{R} \Bbbk) \to H^{1}_{\cF^{\fn}}(K, A)[\fp]
$$
induced by $\Bbbk \hookrightarrow R$ is also injective.
Hence we have
\begin{align*}
r + \nu(\fn) &= \dim_{\Bbbk}H^1_{\cF^{\fn}}(K, A) \otimes_{R} \Bbbk
\\
&\leq \dim_{\Bbbk}H^1_{\cF^{\fn}}(K, A \otimes_R \Bbbk)
\\
&\leq \dim_{\Bbbk}H^1_{\cF^{\fn}}(K, A)[\fp]
\\
&= r + \nu(\fn)
\end{align*}
where the last equality follows from the fact that $\dim_{\Bbbk}R[\fp] = 1$ and $H^{1}_{\cF^{\fn}}(K, A) \simeq R^{r+\nu(\fn)}$.
Hence both the injections are isomorphisms.
\end{proof}

\begin{corollary}\label{selisom}
For any ideal $\fn \in \cN$, the map $H^{1}_{\cF(\fn)}(K, A \otimes_{R} \Bbbk) \to H^{1}_{\cF(\fn)}(K, A)[\fp]$ induced by
the map $\Bbbk \hookrightarrow R$ is an isomorphism.
\end{corollary}
\begin{proof}
Let $\fn \in \cN$. By using Lemma \ref{chebotarev}, we can take an ideal $\fm \in \cN$ with $\fn \mid \fm$ and
$H^{1}_{(\cF^{*})_{\fm}}(K, A^{*}(1)) = 0$.
Then by the definition of the Selmer structure $\cF(\fn)$, we have the following diagram, whose rows are exact:
\begin{align*}
{\small
\xymatrix{
H^{1}_{\cF(\fn)}(K, A_\Bbbk) \ar@{^{(}->}[r] \ar[d] & H^{1}_{\cF^{\fm}}(K, A_\Bbbk) \ar[r] \ar[d] & \bigoplus_{\fq \mid \fn}H^{1}_{/ {\rm tr}}(K_{\fq}, A_\Bbbk) \oplus \bigoplus_{\fq \mid \frac{\fm}{\fn}}H^{1}_{/f}(K_{\fq}, A_\Bbbk) \ar[d]^{\simeq}
\\
H^{1}_{\cF(\fn)}(K, A)[\fp] \ar@{^{(}->}[r] & H^{1}_{\cF^{\fm}}(K, A) [\fp] \ar[r] & \bigoplus_{\fq \mid \fn}H^{1}_{/ {\rm tr}}(K_{\fq}, A)[\fp] \oplus \bigoplus_{\fq \mid \frac{\fm}{\fn}}H^{1}_{/f}(K_{\fq}, A )[\fp] .
}}
\end{align*}
Here we abbreviate $A \otimes_{R} \Bbbk$ to $A_{\Bbbk}$, the vertical maps are induced by $\Bbbk \hookrightarrow R$ and the rightmost vertical map is an isomorphism.
By Lemma~\ref{s-iso}, the middle vertical map is also an isomorphism, and so is the left.
\end{proof}

For an ideal $\fn \in \cN$, put
$$
\lambda(\fn) := \dim_\Bbbk H^1_{\cF(\fn)}(K, A \otimes_R \Bbbk)
$$
and
$$
\lambda^*(\fn) := \dim_\Bbbk H^1_{\cF^\ast(\fn)}(K, (A \otimes_R \Bbbk)^*(1)).
$$

\begin{corollary}\label{indep-diff}
For any ideal $\fn$ in $\cN$ one has $r = \lambda(\fn) - \lambda^{*}(\fn)$ and hence also $\lambda^{*}(\fn) < \lambda(\fn)$.
\end{corollary}
\begin{proof}
Let $\fn \in \cN$.
We take an ideal $\fm \in \cN$ with $\fn \mid \fm$ and $H^{1}_{(\cF^{*})_{\fm}}(K, A^{*}(1)) = 0$ by using Lemma \ref{chebotarev}.
Then $H^{1}_{(\cF^{*})_{\fm}}(K, (A \otimes_{R} \Bbbk)^{*}(1)) = 0$ by Corollary \ref{dualselisom}.
Hence we have
\begin{align*}
\lambda(\fn) - \lambda^{*}(\fn) = \dim_\Bbbk H^1_{\cF^{\fm}}(K, A \otimes_R \Bbbk) - \nu(\fm) = r
\end{align*}
where the first equality follows from the global duality and
$$
\dim_{\Bbbk}H^{1}_{/f}(K_{\fq}, A \otimes_{R} \Bbbk) = \dim_{\Bbbk}H^{1}_{/{\rm tr}}(K_{\fq}, A \otimes_{R} \Bbbk) = 1
$$
for any prime $\fq \in \cP$, and the second equality follows from Lemma~\ref{s-iso} and
Hypothesis~\ref{hyp large}.
\end{proof}

\begin{proposition}
\label{rankind}
Let $\fn \in \cN$ and $\fq \in \cP$ with $\fq \nmid \fn$.
Assume that the localization maps
$$
H^1_{\cF(\fn)}(K, A \otimes_R \Bbbk) \to H^1_{f}(K_\fq, A \otimes_R \Bbbk)
$$
and
$$
H^1_{\cF^\ast(\fn)}(K, (A \otimes_R \Bbbk)^*(1)) \to H^1_{f}(K_\fq, (A \otimes_R \Bbbk)^*(1))
$$
are non-zero.
Then we have $\lambda(\fn\fq) = \lambda(\fn) - 1$ and $\lambda^*(\fn\fq) = \lambda^*(\fn) -1$.
\end{proposition}
\begin{proof}
By the assumption and the fact that
$$
\dim_\Bbbk H^1_{f}(K_\fq, A \otimes_R \Bbbk) = \dim_\Bbbk H^1_{f}(K_\fq, (A \otimes_R \Bbbk)^*(1)) = 1,
$$
the sequences
\begin{align*}
0 \to H^{1}_{\cF_{\fq}(\fn)}(K, A \otimes_{R} \Bbbk) \to H^1_{\cF(\fn)}(K, A \otimes_R \Bbbk) \to H^1_{f}(K_\fq, A \otimes_R \Bbbk) \to 0
\end{align*}
and
\begin{align*}
0 \to H^{1}_{/f}(K_{\fq}, A) \to H^{1}_{\cF^{*}(\fn)}(K, (A \otimes_{R} \Bbbk)^{*}(1))^{*} \to H^1_{(\cF^{*})_{\fq}(\fn)}(K, (A \otimes_R \Bbbk)^{*}(1))^{*} \to 0
\end{align*}
are exact. Hence by the global duality and the second exact sequence, we have
$$
H^{1}_{\cF^{\fq}(\fn)}(K, A \otimes_{R} \Bbbk) = H^{1}_{\cF(\fn)}(K, A \otimes_{R} \Bbbk),
$$
and so
$$
H^{1}_{\cF(\fn\fq)}(K, A \otimes_{R} \Bbbk) = H^{1}_{\cF(\fn)}(K, A \otimes_{R} \Bbbk) \cap H^{1}_{\cF(\fn\fq)}(K, A \otimes_{R} \Bbbk) = H^{1}_{\cF_{\fq}(\fn)}(K, A \otimes_{R} \Bbbk).
$$
Thus $\lambda(\fn\fq) = \lambda(\fn) - 1$ by the first exact sequence, and so $\lambda^*(\fn\fq) = \lambda^*(\fn) -1$ by Corollary~\ref{indep-diff}.
\end{proof}

According to Mazur and Rubin \cite{MRkoly}, it is convenient to regard elements in $\cN$ as `vertices'. In Definition \ref{graph} below, we consider a graph, whose vertices consist of elements in $\cN$. We first recall the notion of `core vertices' introduced by Mazur and Rubin.

\begin{definition}
We say that $\fn \in \cN$ is a core vertex if $H^1_{\cF^\ast(\fn)}(K, A^*(1)) = 0$.
\end{definition}

\begin{remark}\label{core rem}\

\begin{itemize}
\item[(i)] By Corollary \ref{dualselisom}, $\fn \in \cN$ is a core vertex if and only if $\lambda^{*}(\fn) = 0$.
In this case, we have $\lambda(\fn) = r$ by Corollary \ref{indep-diff}.
\item[(ii)] Let $\fn \in \cN$ be a core vertex. By the global duality, we have
$$
0 \to H^{1}_{\cF(\fn)}(K, A) \to H^{1}_{\cF^{\fn}}(K, A) \to \bigoplus_{\fq \mid \fn}H^{1}_{/ {\rm tr}}(K_{\fq}, A) \to 0.
$$
Since $H^{1}_{\cF^{\fn}}(K, A) \simeq R^{r + \nu(\fn)}$ by Hypothesis \ref{hyp large} and
$H^{1}_{/{\rm tr}}(K_{\fq}, A) \simeq R$ for any prime $\fq \in \cP$, the $R$-module
$H^{1}_{\cF(\fn)}(K, A)$ is free of rank $r$.
\end{itemize}
\end{remark}

\begin{lemma}\label{diff-lem}
Let $\fn \in \cN$ and $\fq \in \cP$ with $\fq \nmid \fn$.
Then we have $|\lambda(\fn\fq) - \lambda(\fn)| \leq 1$ and $|\lambda^{*}(\fn\fq) - \lambda^{*}(\fn)| \leq 1$.
\end{lemma}
\begin{proof}
Since $\dim_{\Bbbk}H^{1}_{/f}(K_{\fq}, A \otimes_{R} \Bbbk) =
\dim_{\Bbbk}H^{1}_{/{\rm tr}}(K_{\fq}, A \otimes_{R} \Bbbk) = 1$,
we have
$$
0 \leq \dim_{\Bbbk}H^{1}_{\cF^{\fq}(\fn)}(K, A \otimes_{R} \Bbbk) - \lambda(\fn) \leq 1
$$
and
$$
0 \leq \dim_{\Bbbk}H^{1}_{\cF^{\fq}(\fn)}(K, A \otimes_{R} \Bbbk) - \lambda(\fn\fq) \leq 1.
$$
Hence $|\lambda(\fn\fq) - \lambda(\fn)| \leq 1$.
The inequality $|\lambda^{*}(\fn\fq) - \lambda^{*}(\fn)| \leq 1$ follows from Corollary~\ref{indep-diff} and
$|\lambda(\fn\fq) - \lambda(\fn)| \leq 1$.
\end{proof}

\begin{corollary}\label{diff-cor}
There is a core vertex $\fn \in \cN$ with $\nu(\fn) = \lambda^{*}(1)$.
Furthermore, every core vertex $\fn \in \cN$ satisfies an inequality $\nu(\fn) \geq \lambda^*(1)$.
\end{corollary}
\begin{proof}
The existence of a core vertex $\fn \in \cN$ with $\nu(\fn) = \lambda^{*}(1)$ follows from Lemma~\ref{chebotarev}, Corollary~\ref{indep-diff},
and  Proposition~\ref{rankind}.
The second claim follows from Lemma~\ref{diff-lem}.
\end{proof}

\begin{definition} \label{graph}
We define a graph $\cX^0$ as follows.
\begin{itemize}
\item The vertices of $\cX^0$ are the core vertices.
\item Let $\fn$ and $\fn\fq$ be core vertices.
We join $\fn$ and $\fn\fq$ by an edge in $\cX^0$
if and only if the localization map $H^1_{\cF(\fn)}(K, A \otimes_{R} \Bbbk) \to H^1_f(K_\fq, A \otimes_{R} \Bbbk)$ is non-zero.
\end{itemize}
\end{definition}

\begin{lemma}
\label{joincore}
Let $\fn \in \cN$ and $\fq \in \cP$ with $\fq \nmid \fn$.
\begin{itemize}
\item[(i)] If $\fn$ is a core vertex and
the map $H^1_{\cF(\fn)}(K, A \otimes_R \Bbbk) \to H^1_f(K_\fq, A \otimes_R \Bbbk)$ is non-zero,
then $\fn\fq$ is also a core vertex and $\fn$ and $\fn\fq$ are joined by an edge in $\cX^0$.
\item[(ii)] If $\fn\fq$ is a core vertex and
the map $H^1_{\cF(\fn\fq)}(K, A \otimes_R \Bbbk) \to H^1_{\rm tr}(K_\fq, A \otimes_R \Bbbk)$ is non-zero,
then $\fn$ is also a core vertex and $\fn$ and $\fn\fq$ are joined by an edge in $\cX^0$.
\end{itemize}
\end{lemma}
\begin{proof}
We first prove claim (i). If $\fn$ is a core vertex,
we have an exact sequence by the global duality
\begin{multline*}
0 \to H^1_{\cF_\fq(\fn)}(K, A \otimes_R \Bbbk) \to H^1_{\cF(\fn)}(K, A \otimes_R \Bbbk )
\to H^1_f(K_\fq, A \otimes_R \Bbbk)\\ \to H^1_{(\cF^{*})^{\fq}(\fn)}(K, (A \otimes_R \Bbbk)^*(1))^* \to 0.
\end{multline*}
Since $\dim_\Bbbk H^1_f(K_\fq, A \otimes_R \Bbbk) = 1$ and the map $H^1_{\cF(\fn)}(K, A \otimes_R \Bbbk) \to H^1_f(K_\fq, A \otimes_R \Bbbk)$ is non-zero, we have $H^1_{(\cF^{*})^{\fq}(\fn)}(K, (A \otimes_R \Bbbk)^*(1)) = 0$, and so $H^1_{\cF^*(\fn\fq)}(K, (A \otimes_R \Bbbk)^*(1))=0$.

To prove claim (ii), we can show in a similar way that $\fn$ is a core vertex, by showing that $H^1_{(\cF^{*})^{\fq}(\fn)}(K, (A \otimes_R \Bbbk)^*(1)) = 0$.
By the global duality, the cokernel of the localization map $H^1_{\cF(\fn)}(K, A \otimes_{R} \Bbbk) \to H^1_f(K_\fq, A \otimes_{R} \Bbbk)$ injects into
$H^1_{(\cF^{*})^{\fq}(\fn)}(K, (A \otimes_R \Bbbk)^*(1))^{*} = 0$.
Hence $\fn$ and $\fn\fq$ are joined by an edge in $\cX^0$.
\end{proof}

\begin{lemma}[{\cite[Lem. 4.3.9]{MRkoly}}]\label{koly-join1}
If $\fn$ and $\fn\fq$ are core vertices, then there is a path in $\cX^0$ from $\fn$ to $\fn\fq$.
\end{lemma}
\begin{proof}
We may assume that the map $H^1_{\cF(\fn)}(K, A \otimes_R \Bbbk) \to H^1_{f}(K_\fq, A \otimes_R \Bbbk)$ is zero.
Then we have
\begin{align*}
H^1_{\cF(\fn)}(K, A \otimes_R \Bbbk) = H^1_{\cF_\fq(\fn)}(K, A \otimes_R \Bbbk)
= H^1_{\cF(\fn\fq)}(K, A \otimes_R \Bbbk)
\end{align*}
where the second equality follows from $\lambda(\fn) = r = \lambda(\fn\fq)$.
Furthermore, by the global duality, we have $\dim_\Bbbk H^1_{(\cF_\fq(\fn))^*}(K, (A \otimes_R \Bbbk)^*(1)) = 1$.
By Lemma \ref{chebotarev}, there is a prime $\fr \in \cP$ with $\fr \nmid \fn\fq$ such that
the localization maps
$$H^1_{\cF(\fn)}(K, A \otimes_R \Bbbk) \to H^1_f(K_\fr, A \otimes_R \Bbbk)$$
and
$$H^1_{(\cF_\fq(\fn))^*}(K, (A \otimes_R \Bbbk)^*(1)) \to H^1_{f}(K_\fr, (A \otimes_R \Bbbk)^*(1))$$
are non-zero.
By Lemma \ref{joincore}(i) and $H^1_{\cF(\fn)}(K, A \otimes_R \Bbbk) = H^1_{\cF(\fn\fq)}(K, A \otimes_R \Bbbk)$,
it follows from the fact that the map
$H^1_{\cF(\fn)}(K, A \otimes_R \Bbbk) \to H^1_f(K_\fr, A \otimes_R \Bbbk)$ is non-zero
that both $\fn\fr$ and $\fn\fq\fr$ are core vertices and that
there are paths in $\cX^0$ from $\fn$ to $\fn\fr$ and from $\fn\fq$ to $\fn\fq\fr$.
Hence again by Lemma~\ref{joincore}(i), we only need to show that
$$
H^1_{\cF_\fq(\fn\fr)}(K, A \otimes_R \Bbbk) \neq H^1_{\cF(\fn\fr)}(K, A \otimes_R \Bbbk).
$$
Since the map
$H^1_{(\cF_\fq(\fn))^*}(K, (A \otimes_R \Bbbk)^*(1)) \to H^1_{f}(K_\fr, (A \otimes_R \Bbbk)^*(1))$ is non-zero,
we have $H^1_{\cF_\fq(\fn)}(K, A \otimes_R \Bbbk) = H^1_{\cF_\fq^\fr(\fn)}(K, A \otimes_R \Bbbk)$
by the global duality. Hence we get an equality
\begin{align*}
H^1_{\cF_\fq(\fn\fr)}(K, A \otimes_R \Bbbk) = H^1_{\cF_{\fq\fr}(\fn)}(K, A \otimes_R \Bbbk)
= H^1_{\cF_\fr(\fn)}(K, A \otimes_R \Bbbk)
\end{align*}
where the second equality follows from
$H^1_{\cF(\fn)}(K, A \otimes_R \Bbbk) = H^1_{\cF_\fq(\fn)}(K, A \otimes_R \Bbbk)$.
Since the map $H^1_{\cF(\fn)}(K, A \otimes_R \Bbbk) \to H^1_f(K_\fr, A \otimes_R \Bbbk)$ is non-zero,
we have
$$
\dim_\Bbbk H^1_{\cF_\fq(\fn\fr)}(K, A \otimes_R \Bbbk) = \dim_{\Bbbk} H^1_{\cF_\fr(\fn)}(K, A \otimes_R \Bbbk) =
\lambda(\fn) - 1 = \lambda(\fn\fr) -1.
$$
This completes the proof.
\end{proof}

\begin{lemma}\label{change}
Let $s$ be a positive integer.
For $1 \leq i \leq s$, let $\fn_{i} \in \cN$ be a core vertex and $\fq_{i} \in \cP$ with $\fq_{i} \mid \fn_{i}$.
If $2s < p$ and $\fn_{i}/\fq_{i}$ is not a core vertex for any $1 \leq i \leq s$, then
there is a prime $\fr \in \cP$ with $\fr \nmid \fn_{1} \cdots \fn_{s}$ such that $\fn_{1}\fr/\fq_{1}, \ldots, \fn_{s}\fr/\fq_{s}$
are core vertices and that there is a path in $\cX^{0}$ from $\fn_{i}$ to $\fn_{i}\fr/\fq_{i}$ for every $1 \leq i \leq s$.
\end{lemma}
\begin{proof}
Let $1 \leq i \leq s$ and put $\fm_{i} = \fn_{i}/\fq_{i}$.
Since $\lambda^{*}(\fn_{i}) = 0$ and $\fm_{i}$ is not a core vertex, we have
$\lambda(\fm_{i}) = r + 1$ and $\lambda^{*}(\fm_{i}) = 1$ by Corollary~\ref{indep-diff} and
Lemma~\ref{diff-lem}.
Note that
$$
H^{1}_{\cF(\fn_{i})}(K, A \otimes_{R} \Bbbk) \subseteq  H^{1}_{\cF^{\fq_{i}}(\fm_{i})}(K, A \otimes_{R} \Bbbk)
= H^{1}_{\cF(\fm_{i})}(K, A \otimes_{R} \Bbbk).
$$
In fact, we have $\lambda(\fm_{i}) \leq \dim_{\Bbbk}H^{1}_{\cF^{\fq_{i}}(\fm_{i})}(K, A \otimes_{R} \Bbbk) \leq \lambda(\fn_{i}) + 1 = \lambda(\fm_{i})$, where the second inequality follows from the exact sequence
$$
0 \to H^{1}_{\cF(\fn_{i})}(K, A \otimes_{R} \Bbbk) \to H^{1}_{\cF^{\fq_{i}}(\fm_{i})}(K, A \otimes_{R} \Bbbk) \to H^{1}_{/{\rm tr}}(K, A \otimes_{R} \Bbbk)
$$
and $\dim_{\Bbbk}H^{1}_{/{\rm tr}}(K, A \otimes_{R} \Bbbk) =1$.
By Lemma \ref{chebotarev} and $2s < p$, there is a prime $\fr \in \cP$ with $\fr \nmid \fn_{1} \cdots \fn_{s}$ such that
the maps
$$
H^{1}_{\cF(\fn_{i})}(K, A \otimes_{R} \Bbbk) \to H^{1}_{f}(K_{\fr}, A \otimes_{R} \Bbbk)
$$
and
$$
H^{1}_{\cF^{*}(\fm_{i})}(K, (A \otimes_{R} \Bbbk)^{*}(1)) \to H^{1}_{f}(K_{\fr}, (A \otimes_{R} \Bbbk)^{*}(1))
$$
are non-zero for any $1 \leq i \leq s$.
Then $\lambda^{*}(\fm_{i}\fr) = \lambda^{*}(\fm_{i}) - 1 = 0$ by Proposition~\ref{rankind}.
Hence $\fm_{i}\fr$ is a core vertex.
Furthermore, by Lemma~\ref{joincore}(i), $\fn_{i}\fr$ is also a core vertex.
Therefore by using Lemma~\ref{koly-join1}, there is a path in $\cX^{0}$ from $\fn_{i}$ to $\fm_{i}\fr$ for any $1 \leq i \leq s$.
\end{proof}

\begin{corollary}\label{koly-join3}
Suppose that $\fn_{1}, \fn_{2} \in \cN$ are core vertices and $\nu(\fn_{1}) = \nu(\fn_{2}) = \lambda^{*}(1)$.
If $p>3$, then there is a path in $\cX^0$ from $\fn_{1}$ to $\fn_{2}$.
\end{corollary}
\begin{proof}
We will prove this lemma by induction on $\lambda^{*}(1) - \nu(\gcd(\fn_{1}, \fn_{2})) \geq 0$.
When it is equal to zero, we have $\fn_{1} = \fn_{2}$ and there is nothing to prove.
Suppose $\fn_{1} \neq \fn_{2}$, and fix distinct primes $\fq_{1} \mid \fn_{1}$ and $\fq_{2} \mid \fn_{2}$.
Then $\fn_{1}/\fq_{1}$ and $\fn_{2}/\fq_{2}$ are not core vertices by Corollary~\ref{diff-cor}.
By Lemma~\ref{change} and $p>3$, there is a prime $\fr \in \cP$ with $\fr \nmid \fn_{1}\fn_{2}$ such that
both $\fn_{1}\fr/\fq_{1}$ and $\fn_{2}\fr/\fq_{2}$ are core vertices and that
there are paths in $\cX^0$ connecting $\fn_{1}$ to $\fn_{1}\fr/\fq_{1}$ and $\fn_{2}$ to $\fn_{2}\fr/\fq_{2}$.
Since $\nu(\gcd(\fn_{1}\fr/\fq_{1}, \fn_{2}\fr/\fq_{2})) = \nu(\gcd(\fn_{1}, \fn_{2})) + 1$,
there is a path in $\cX^0$ connecting $\fn_{1}\fr/\fq_{1}$ to $\fn_{2}\fr/\fq_{2}$ by the induction hypothesis.
This completes the proof.
\end{proof}

\begin{lemma}\label{koly-join2}
Suppose that $\fn \in \cN$ is a core vertex with $\nu(\fn) > \lambda^*(1)$.
Then there is a core vertex $\fm \in \cN$ with $\nu(\fm) = \nu(\fn) - 1$
such that there is a path in $\cX^0$ from $\fn$ to $\fm$.
\end{lemma}
\begin{proof}
By Lemma~\ref{koly-join1}, we may assume that $\fn/\fq$ is not a core vertex for any $\fq \mid \fn$.
Then by Lemma~\ref{joincore}(ii), we have
$$
H^1_{\cF(\fn)}(K, A \otimes_R \Bbbk) = H^1_{\cF_\fn}(K, A \otimes_R \Bbbk).
$$
Since
$\lambda(1) - \dim_{\Bbbk}H^{1}_{\cF_{\fn}}(K, A \otimes_{R} \Bbbk) = \lambda(1) - \lambda(\fn) = \lambda(1) - r = \lambda^{*}(1) < \nu(\fn)$
and
$$
H^{1}_{\cF_{\fn}}(K, A \otimes_{R} \Bbbk) = \ker\left( H^{1}_{\cF}(K, A \otimes_{R} \Bbbk) \to \bigoplus_{\fq \mid \fn}H^{1}_{f}(K_{\fq}, A \otimes_{R} \Bbbk) \right),
$$
the map $H^{1}_{\cF}(K, A \otimes_{R} \Bbbk) \to \bigoplus_{\fq \mid \fn}H^{1}_{f}(K_{\fq}, A \otimes_{R} \Bbbk)$ is not surjective.
For any prime $\fq \mid \fn$, if $H^1_{\cF_\fn}(K, A \otimes_R \Bbbk) \neq H^1_{\cF_{\fn/\fq}}(K, A \otimes_R \Bbbk)$, there is an element $x \in H^{1}_{\cF}(K, A \otimes_{R} \Bbbk)$ such that
${\rm loc}_{\fq}(x) \neq 0$ and ${\rm loc}_{\fq'}(x) = 0$ for each prime $\fq' \mid \fn/\fq$.
Hence there is a prime $\fq \mid \fn$ such that
$$
H^1_{\cF_\fn}(K, A \otimes_R \Bbbk) = H^1_{\cF_{\fn/\fq}}(K, A \otimes_R \Bbbk)
$$
since the map $H^{1}_{\cF}(K, A \otimes_{R} \Bbbk) \to \bigoplus_{\fq \mid \fn}H^{1}_{f}(K_{\fq}, A \otimes_{R} \Bbbk)$ is not surjective.
Let $\fm = \fn/\fq$.
Since $\fm$ is not a core vertex, we have $\lambda^*(\fm) = 1$ by Lemma~\ref{diff-lem}.
Hence by Lemma~\ref{chebotarev}, there is a prime $\fr \in \cP$ with $\fr \nmid \fn$ such that the maps
\begin{align*}
H^1_{\cF_\fn}(K, A \otimes_R \Bbbk) \to H^1_{f}(K_\fr, A \otimes_R \Bbbk)
\end{align*}
and
\begin{align*}
H^1_{\cF^*(\fm)}(K, (A \otimes_R \Bbbk)^*(1)) \to H^1_{f}(K_{\fr}, (A \otimes_R \Bbbk)^*(1))
\end{align*}
are non-zero.
Then $\fm\fr$ and $\fn\fr$ are core vertices by Proposition~\ref{rankind} and Lemma~\ref{joincore}(i).
Hence by using Lemma~\ref{koly-join1}, we see that there is a path in $\cX^0$ from $\fn$ to $\fm\fr$.

Since the map $H^1_{\cF^*(\fm)}(K, (A \otimes_R \Bbbk)^*(1)) \to H^1_{f}(K_\fr, (A \otimes_R \Bbbk)^*(1))$ is surjective,
we have $H^1_{\cF(\fm)}(K, A \otimes_R \Bbbk) = H^1_{\cF^{\fr}(\fm)}(K, A \otimes_R \Bbbk)$ by the global duality.
Hence we get
\begin{align*}
H^1_{\cF(\fm\fr)}(K, A \otimes_R \Bbbk) &= H^1_{\cF^{\fr}(\fm)}(K, A \otimes_R \Bbbk) \cap H^1_{\cF(\fm\fr)}(K, A \otimes_R \Bbbk)
\\
&= H^1_{\cF_{\fr}(\fm)}(K, A \otimes_R \Bbbk).
\end{align*}
Furthermore, we have
$$
\dim_{\Bbbk}H^{1}_{\cF_{\fm\fr}}(K, A \otimes_{R} \Bbbk) = \dim_{\Bbbk}H^{1}_{\cF_{\fn\fr}}(K, A \otimes_{R} \Bbbk) = r - 1
$$
since the map $H^1_{\cF_\fn}(K, A \otimes_R \Bbbk) \to H^1_{f}(K_\fr, A \otimes_R \Bbbk)$ is non-zero and
$$
H^{1}_{\cF_{\fm}}(K, A \otimes_{R} \Bbbk) = H^1_{\cF_\fn}(K, A \otimes_R \Bbbk) = H^1_{\cF(\fn)}(K, A \otimes_R \Bbbk).
$$
Since $\lambda(\fm\fr) = r$, we conclude that the sum of localization maps
$$
H^1_{\cF_{\fr}(\fm)}(K, A \otimes_R \Bbbk) = H^1_{\cF(\fm\fr)}(K, A \otimes_R \Bbbk) \to
\bigoplus_{\fs \mid \fm}H^{1}_{\rm tr}(K_{\fs}, A \otimes_{R} \Bbbk)
$$
is non-zero.
Thus there is a prime $\fs \mid \fm$ such that
the localization map
$$
H^1_{\cF(\fm\fr)}(K, A \otimes_R \Bbbk) \to H^1_{\rm tr}(K_\fs, A \otimes_R \Bbbk)
$$ is non-zero.
By Lemma~\ref{joincore}(ii), 
$\fm\fr/\fs$ is a core vertex.
Hence by Lemma~\ref{koly-join1} there is a path in $\cX^0$ from $\fm\fr/\fs$ to $\fm\fr$. This completes the proof.
\end{proof}

\begin{theorem}[{\cite[Th. 4.3.12]{MRkoly}}]\label{connected}
Suppose that Hypotheses \ref{hyp1}, \ref{hyp2}, and \ref{hyp large}.
If $p >3$, then the graph $\cX^0$ is connected.
\end{theorem}
\begin{proof}
Let $\fn_{1}, \fn_{2} \in \cN$ be core vertices.
By Corollary~\ref{diff-cor} and Lemma~\ref{koly-join2}, there are core vertices $\fm_{1}, \fm_{2} \in \cN$ with $\lambda(\fm_{1}) = \lambda(\fm_{2}) = \lambda^{*}(1)$
such that there are paths in $\cX^0$ from $\fn_{1}$ to $\fm_{1}$ and $\fn_{2}$ to $\fm_{2}$.
Since $p >3$, there is a path in $\cX^{0}$ from $\fm_{1}$ to $\fm_{2}$ by Corollary~\ref{koly-join3}.
Hence the graph $\cX^{0}$ is connected.
\end{proof}

\begin{lemma}\label{inj-koly}
Let $\fn \in \cN$ and $\fq \in \cP$ with $\fq \nmid \fn$.
If $\fn$ and $\fn\fq$ are core vertices and the localization map
$H^1_{\cF(\fn)}(K, A \otimes_{R} \Bbbk) \to H^1_f(K_\fq, A \otimes_{R} \Bbbk)$ is non-zero, then the maps
\begin{align*}
\varphi_{\fq}^{\rm fs} \colon {\bigcap}^r_R H^1_{\cF(\fn)}(K, A) \otimes G_{\fn} \to {\bigcap}^{r-1}_R H^1_{\cF_\fq(\fn)}(K, A) \otimes G_{\fn\fq}
\end{align*}
and
\begin{align*}
v_{\fq} \colon {\bigcap}^r_R H^1_{\cF(\fn\fq)}(K, A) \otimes G_{\fn\fq} \to {\bigcap}^{r-1}_R H^1_{\cF_\fq(\fn)}(K, A) \otimes G_{\fn\fq}
\end{align*}
are isomorphisms.
\end{lemma}
\begin{proof}
Note that $H^{1}_{\cF(\fn)}(K, A)$ is free of rank $r$ by Remark \ref{core rem}.
The natural map
$$
H^{1}_{\cF(\fn)}(K, A) \otimes_{R} \Bbbk \to H^{1}_{\cF(\fn)}(K, A \otimes_{R} \Bbbk)
$$
is an isomorphism by Lemma~\ref{injective} and the fact that
$$
\dim_{\Bbbk} H^{1}_{\cF(\fn)}(K, A) \otimes_{R} \Bbbk = r = \lambda(\fn) = \dim_{\Bbbk} H^{1}_{\cF(\fn)}(K, A \otimes_{R} \Bbbk).
$$
Thus the localization map
$$
H^1_{\cF(\fn)}(K, A) \to H^1_f(K_\fq, A)
$$ is surjective since
$H^{1}_{f}(K_{\fq}, A) \otimes_{R} \Bbbk \simeq H^{1}_{f}(K_{\fq}, A \otimes_{R} \Bbbk)$ and
the map $H^1_{\cF(\fn)}(K, A \otimes_{R} \Bbbk) \to H^1_f(K_\fq, A \otimes_{R} \Bbbk)$ is non-zero.
Hence, by the global duality, we have a split exact sequence of free $R$-modules
\begin{align*}
0 \to H^1_{\cF_\fq(\fn)}(K, A) \to H^1_{\cF(\fn)}(K, A) \to H^1_f(K_\fq, A) \to 0
\end{align*}
and
$H^{1}_{(\cF^{*})^{\fq}(\fn)}(K, A^{*}(1)) = H^{1}_{\cF^{*}(\fn)}(K, A^{*}(1)) = 0$. Again, by the global duality and the fact that $H^{1}_{(\cF^{*})^{\fq}(\fn)}(K, A^{*}(1)) = 0$, we have a split exact sequence of free $R$-modules
\begin{align*}
0 \to H^1_{\cF_\fq(\fn)}(K, A) \to H^1_{\cF(\fn\fq)}(K, A) \to H^1_{\rm tr}(K_\fq, A) \to 0.
\end{align*}
Since the $R$-modules $H^1_{\cF(\fn)}(K, A)$ and $H^1_{\cF(\fn\fq)}(K, A)$ are free of rank $r$, the maps $\varphi_{\fq}^{\rm fs}$ and $v_{\fq}$
are isomorphisms.
\end{proof}

\begin{theorem}\label{thm koly}
Suppose Hypotheses \ref{hyp1}, \ref{hyp2}, and \ref{hyp large}.
Let $\fn \in \cN$ be a core vertex.
If $p>3$, then the projection map
\begin{align*}
{\rm KS}_r(A, \cF) \to {\bigcap}^r_R H^1_{\cF(\fn)}(K, A) \otimes G_\fn; \ \kappa \mapsto \kappa_{\fn}
\end{align*}
is an isomorphism.
In particular, ${\rm KS}_r(A, \cF)$ is a free $R$-module of rank one.
\end{theorem}
\begin{proof}
Since $\fn \in \cN$ is a core vertex, by the global duality, we have a split exact sequence of free $R$-modules:
$$
0 \to H^{1}_{\cF(\fn)}(K, A) \to H^{1}_{\cF^{\fn}}(K, A) \to \bigoplus_{\fq \mid \fn}H^{1}_{/{\rm tr}}(K_{\fq}, A) \to 0.
$$
Hence, by Hypothesis \ref{hyp large}, the map
\begin{eqnarray}\label{phiisom}
{\bigwedge}_{\fq \mid \fn}\varphi_\fq^{\rm fs} \colon
{\bigcap}_R^{r+\nu(\fn)} H^1_{\cF^\fn}(K,A) \to {\bigcap}_R^r H^1_{\cF(\fn)}(K,A) \otimes G_\fn
\end{eqnarray}
is an isomorphism.
By using the map ${\rm Reg}_r \colon {\rm SS}_r(A, \cF) \to {\rm KS}_r(A, \cF)$ and by Theorem \ref{thm stark}(i), we conclude that the map
${\rm KS}_r(A, \cF) \to  {\bigcap}^{r}_{R} H^1_{\cF(\fn)}(K, A) \otimes G_\fn$ is surjective.

Let $\kappa \in {\rm KS}_r(A, \cF)$ with $\kappa_\fn = 0$.
To prove injectivity, we will show that $\kappa_\fm = 0$ for any ideal $\fm \in \cN$ by induction on $\lambda^*(\fm)$.
If $\lambda^*(\fm) = 0$, then $\fm$ is a core vertex.
Hence by Theorem \ref{connected}, Lemma \ref{inj-koly}, and $\kappa_\fn = 0$,
we have $\kappa_\fm = 0$.
Suppose that $\lambda^*(\fm) > 0$ and $\kappa_\fm \neq 0$.
By Lemma \ref{chebotarev},  we can take an ideal $\fr \in \cN$ coprime to $\fm$ such that
the localization maps
\begin{align*}
H^1_{\cF(\fm)}(K, A \otimes_R \Bbbk) \to H^1_f(K_{\fq}, A \otimes_R \Bbbk)
\end{align*}
and
\begin{align*}
H^1_{\cF^*(\fm)}(K, (A \otimes_R \Bbbk)^*(1)) \to H^1_f(K_{\fq}, (A \otimes_R \Bbbk)^*(1))
\end{align*}
are non-zero for any prime $\fq \mid \fr$ and that
the sum of localization maps
$$
H^1_{\cF(\fm)}(K, A \otimes_R \Bbbk) \to \bigoplus_{\fq \mid \fr}H^{1}_{f}(K_{\fq}, A \otimes_R \Bbbk)
$$
is injective.
Then the map
$$
H^1_{\cF(\fm)}(K, A) \to \bigoplus_{\fq \mid \fr}H^{1}_{f}(K_{\fq}, A)
$$
is also injective by using Corollary \ref{selisom} and the injectivity of the maps $H^{1}_{f}(K_{\fq}, A \otimes_{R} \Bbbk) \to H^{1}_{f}(K_{\fq}, A)$ induced by the map $\Bbbk \hookrightarrow R$ for any prime $\fq \in \cP$.
By taking the dual of this map, we see that
$$
0 = \bigcap_{\fq \mid \fr}\ker\left(\varphi^{\rm fs}_{\fq} \colon {\bigcap}^{r}_{R}H^1_{\cF(\fm)}(K, A) \otimes G_{\fm} \to {\bigcap}^{r-1}_{R}H^1_{\cF_{\fq}(\fm)}(K, A) \otimes {G_{\fm \fq}}\right)
$$
by Corollary \ref{bidual-ker}.
Hence there is a prime $\fq \mid \fr$ such that $\varphi_\fq^{\rm fs}(\kappa_\fm) \neq 0$, since we suppose $\kappa_{\fm} \neq 0$.
Furthermore, we have $\lambda^\ast(\fm\fq) = \lambda^\ast(\fm) - 1$ by Proposition \ref{rankind}, and so we conclude that
$\kappa_{\fm\fq} = 0$ by the induction hypothesis.
By the definition of Kolyvagin system, we have $0 = v_{\fq}(\kappa_{\fm\fq}) = \varphi_\fq^{\rm fs}(\kappa_\fm) \neq 0$.
This is a contradiction. Thus $\kappa_\fm = 0$.
\end{proof}

\begin{remark}
The proof of Theorem~\ref{connected} is parallel to that of \cite[Th. 4.3.12]{MRkoly}.
However, the notion of exterior power bidual plays a critical role in the proof of Theorem~\ref{thm koly}, allowing us to overcome the problem discussed by Mazur and Rubin in \cite[Rem. 11.9]{MRselmer}. More precisely, it is crucial for the proof of Theorem \ref{thm koly} that if $\kappa_\fm$ does not vanish, then there exists a prime ideal $\fq $ such that $\varphi_\fq^{\rm fs}(\kappa_\fm)$ does not vanish and the corresponding fact is not true if one defines Kolyvagin systems by using exterior powers rather than exterior power biduals.  This is the reason why Mazur and Rubin could not prove any result that corresponds to Theorem~\ref{thm koly}.
\end{remark}

The following lemma will be used in the proof of Theorem \ref{main}(iii).

\begin{lemma}\label{fitt-ind-1}
Let $\fn \in \cN$. Then for each natural number $i$ one has
$$ \sum_{\fq \in \cP,\  \fq \nmid \fn}\Fitt_{R}^{i-1}(H^{1}_{\cF^{*}(\fn\fq)}(K, A^{*}(1))^{*})\subseteq
\Fitt_{R}^{i}( H^{1}_{\cF^{*}(\fn)}(K, A^{*}(1))^{*}),
$$
with equality if $\Ann_{R}( H^{1}_{\cF(\fn)}(K, A))$ vanishes.
\end{lemma}
\begin{proof}
Note that, if $\cF$ satisfies Hypothesis \ref{hyp large}, then so does $\cF(\fn)$. In fact,  by Lemma \ref{chebotarev}, we can take an ideal $\fd \in \cN$ coprime to $\fn$ such that $H^{1}_{(\cF^{*})_{\fd}(\fn)}(K, A^{*}(1))$ vanishes. Then, by the global duality, we have an  exact sequence
$$
0 \to H^{1}_{\cF^{\fd}(\fn)}(K, A) \to H^{1}_{\cF^{\fd\fn}}(K, A) \to \bigoplus_{\fq \mid \fn} H^{1}_{/{\rm tr}}(K_{\fq}, A) \to 0.
$$
Since $H^{1}_{\cF^{\fd\fn}}(K, A) \simeq R^{r+\nu(\fd\fn)}$ by Remark \ref{free rem} and $H^{1}_{/{\rm tr}}(K_{\fq}, A) \simeq R$ for any prime $\fq \mid \fn$, we conclude that $H^{1}_{\cF^{\fd}(\fn)}(K, A) \simeq R^{r+\nu(\fd)}$. 

Hence we can apply Lemma \ref{fitt-lemma} for the Selmer structure $\cF(\fn)$ 
and we have
\begin{align*}
\Fitt_{R}^{i}( H^{1}_{\cF^{*}(\fn)}(K, A^{*}(1))^{*}) =
\sum_{\fq \in \cP, \ \fq \nmid \fn}\Fitt_{R}^{i-1}(H^{1}_{(\cF^{*})_{\fq}(\fn)}(K, A^{*}(1))^{*}).  \end{align*}
(Note that $S(\cF(\fn)) = S \cup \{\fq \mid \fn \}$ and so the primes running in the sum on the right hand side are restricted to $\fq \in \cP$ with $\fq \nmid \fn$.)
Since the natural maps $H^{1}_{\cF^{*}(\fn\fq)}(K, A^{*}(1))^{*} \to
H^{1}_{(\cF^{*})_{\fq}(\fn)}(K, A^{*}(1))^{*}$ are surjective,
we have
$$
\sum_{\fq \in \cP, \ \fq \nmid \fn}\Fitt_{R}^{i-1}(H^{1}_{\cF^{*}(\fn\fq)}(K, A^{*}(1))^{*})
\subseteq   \sum_{\fq \in \cP, \ \fq \nmid \fn}\Fitt_{R}^{i-1}(H^{1}_{(\cF^{*})_{\fq}(\fn)}(K, A^{*}(1))^{*}).
$$

It is therefore enough to show that if $\Ann_{R}( H^{1}_{\cF(\fn)}(K, A))$ vanishes, then the reverse inclusion is also valid.

Under this assumption, there exists an element $e$ of $H^{1}_{\cF(\fn)}(K, A)$ with $\Ann_{R}(e) = 0$. Let $x$ be a generator of $R[\fp]$.
Then by Lemma \ref{chebotarev}, there is an ideal $\fm \in \cN$ coprime to $\fn$ such that ${\rm loc}_{\fq}(xe) \neq 0$ for any prime $\fq \mid \fm$ and $H^{1}_{(\cF^\ast)_{\fm}(\fn)}(K, A^{*}(1)) = 0$.
Since $H^{1}_{f}(K_{\fq}, A) \simeq R$, it follows from the fact that ${\rm loc}_{\fq}(xe) \neq 0$ that the map
$H^{1}_{\cF(\fn)}(K, A) \to H^{1}_{f}(K_{\fq}, A)$ is surjective for any prime $\fq \mid \fm$.
Hence, by the global duality, we have $H^{1}_{\cF^{*}(\fn)}(K, A^{*}(1)) = H^{1}_{(\cF^{*})^{\fq}(\fn)}(K, A^{*}(1))$. Therefore we have
$$
H^{1}_{\cF^{*}(\fn\fq)}(K, A^{*}(1)) = H^{1}_{\cF^{*}(\fn)}(K, A^{*}(1)) \cap H^{1}_{\cF^{*}(\fn\fq)}(K, A^{*}(1)) = H^{1}_{(\cF^{*})_{\fq}(\fn)}(K, A^{*}(1)).
$$
Again by Lemma \ref{fitt-lemma}, we have
\begin{align*}
\Fitt_{R}^{i}( H^{1}_{\cF^{*}(\fn)}(K, A^{*}(1))^{*} )
&= \sum_{\fq \in \cP, \ \fq \mid \fm}\Fitt_{R}^{i-1}(H^{1}_{(\cF^{*})_{\fq}(\fn)}(K, A^{*}(1))^{*})
\\
&=  \sum_{\fq \in \cP, \ \fq \mid \fm}\Fitt_{R}^{i-1}(H^{1}_{\cF^{*}(\fn\fq)}(K, A^{*}(1))^{*})
\\
&\subseteq  \sum_{\fq \in \cP, \ \fq \nmid \fn}\Fitt_{R}^{i-1}(H^{1}_{\cF^{*}(\fn\fq)}(K, A^{*}(1))^{*}).
\end{align*}
This completes the proof.
\end{proof}

\begin{corollary}\label{fitt-ind} For each natural number $i$ one has
$$
\sum_{\fm \in \cN,\  \nu(\fm) = i}\Fitt_{R}^{0}(H^{1}_{\cF^{*}(\fm)}(K, A^{*}(1))^{*})\subseteq   \Fitt_{R}^{i}( H^{1}_{\cF^{*}}(K, A^{*}(1))^{*}),
$$
with equality provided that $\Ann_{R}(H^{1}_{\cF(\fn)}(K, A))$ vanishes for all ideals $\fn$ in $\cN$.
\end{corollary}

\begin{proof} This result follows directly from Lemma \ref{fitt-ind-1}.
\end{proof}

\begin{proof}[Proof of Theorem \ref{main}]

Claim (i) follows from Theorem~\ref{thm stark}(i), Theorem~\ref{thm koly}, and the isomorphism (\ref{phiisom}).

To prove claim (ii) it is enough to consider the case that $\kappa$ is a basis of ${\rm KS}_{r}(A, \cF)$. To do this we fix an ideal $\fn$ in $ \cN$ and a generator of each $G_{\fq}$ and we regard $\kappa_{\fn}$ as an element of ${\bigcap}^{r}_{R}H^{1}_{\cF(\fn)}(K, A)$.
By claim (i), there exists a basis $\epsilon$ of ${\rm SS}_{r}(A, \cF)$ such that ${\rm Reg}_{r}(\epsilon) = \kappa$.

By using Lemma \ref{chebotarev}, we can take an ideal $\fm \in \cN$ with $\fn \mid \fm$ and $H^{1}_{(\cF^{*})_{\fm}}(K, A^{*}(1)) = 0$.
Then by global duality, we have an exact sequence
\begin{align}\label{fitt-exact}
0 \to H^{1}_{\cF(\fn)}(K, A) \to H^{1}_{\cF^{\fm}}(K, A) \to X \to H^{1}_{\cF^{*}(\fn)}(K, A^{*}(1))^{*} \to 0
\end{align}
where $X = \bigoplus_{\fq \mid \fn}H^{1}_{/ {\rm tr}}(K_{\fq}, A) \oplus \bigoplus_{\fq \mid \frac{\fm}{\fn}}H^{1}_{/f}(K_{\fq}, A )$.
Note that, under Hypothesis~\ref{hyp large}, the $R$-modules $X$ and $H^{1}_{\cF^{\fm}}(K, A)$ are respectively free of ranks $\nu(\fm)$ and $r + \nu(\fm)$. Since $\epsilon_{\fm}$ is a generator of ${\bigcap}^{r+\nu(\fm)}_{R}H^{1}_{\cF^{\fm}}(K, A)$ by Theorem~\ref{thm stark}(i),
$\kappa_{\fn}$ is a generator of the $R$-module
$$
\im\left({\bigwedge}_{\fq \mid \fn}\varphi_{\fq}^{\rm fs} \colon {\bigcap}^{r+\nu(\fm)}_{R}H^{1}_{\cF^{\fm}}(K,A) \to {\bigcap}^{r}_{R}H^{1}_{\cF(\fn)}(K, A)\right)
$$
by the definitions of Stark system and the map ${\rm Reg}_{r}$.
Hence by Lemma~\ref{prop injective}(ii) and the exact sequence (\ref{fitt-exact}), we have
\begin{align*}
\Fitt_{R}^{0}( H^{1}_{\cF^{*}(\fn)}(K, A^{*}(1))^{*})
= \im (\kappa_{\fn}).
\end{align*}

Next, we will prove claim(iii). In view of claim(ii) and Corollary \ref{fitt-ind}, we only need to show that
$\Ann_{R}( H^{1}_{\cF(\fn)}(K, A))$ vanishes for any ideal $\fn \in \cN$ when $R$ is a principal ideal ring.

Let $\fn \in \cN$. Then, by using Lemma \ref{chebotarev}, take an ideal $\fm \in \cN$ with $\fn \mid \fm$ and $H^{1}_{(\cF^{*})_{\fm}}(K, A^{*}(1))$ vanishes. Then we have
$$
{\rm rank}_{R}(X) = \nu (\fm) < r + \nu(\fm) = {\rm rank}_{R}(H^{1}_{\cF^{\fm}}(K, A)).
$$
Since $R$ is principal, there is an injection
$$
R^{r} \hookrightarrow \ker(H^1_{\cF^\fm}(K,A) \to X)=H^{1}_{\cF(\fn)}(K, A),
$$
by the elementary divisor theorem. Hence the ideal $\Ann_{R}( H^{1}_{\cF(\fn)}(K, A))$ vanishes, as required.
\end{proof}

\subsection{Kolyvagin Systems over Gorenstein orders}

In this subsection, we use the same notation as in \S\ref{one-dim case}.
Furthermore, {\it we assume that Hypothesis~\ref{hyp1'} and $(T/p^{m}T, \cF, \cP_{m})$ satisfies Hypothesis~\ref{hyp large} for any positive integer $m$}.

Let $m$ be a positive integer and $\fn \in \cN_{m+1}$ a core vertex for $(T/p^{m+1}T, \cF)$. Then $\fn$ is also a core vertex for $(T/p^{m}T, \cF)$ by Corollary \ref{dualselisom}. Hence in the same way as in \S\ref{one-dim case}, we can construct a map
$$
{\rm KS}_{r}(T/p^{m+1}T, \cF) \to {\rm KS}_{r}(T/p^{m}T, \cF).
$$
such that the diagram
\begin{align*}
\xymatrix{
{\rm SS}_{r}(T/p^{m+1}T, \cF) \ar[d] \ar[r]^{{\rm Reg}_{r}} & {\rm KS}_{r}(T/p^{m+1}T, \cF) \ar[d]
\\
{\rm SS}_{r}(T/p^{m}T, \cF) \ar[r]^{{\rm Reg}_{r}} & {\rm KS}_{r}(T/p^{m}T, \cF)
}
\end{align*}
commutes.

\begin{definition}
We define the module ${\rm KS}_{r}(T, \cF)$ of Kolyvagin systems for $(T, \cF)$ to be the inverse limit
$$
{\rm KS}_{r}(T, \cF) := \varprojlim_{m}{\rm KS}_{r}(T/p^{m}T, \cF).
$$
\end{definition}
The maps ${\rm Reg}_{r} \colon {\rm SS}_{r}(T/p^{m}T, \cF) \to {\rm KS}_{r}(T/p^{m}T, \cF)$ 
induce a homomorphism (also denoted by ${\rm Reg}_{r}$)
$$
{\rm Reg}_{r} \colon {\rm SS}_{r}(T, \cF) \to {\rm KS}_{r}(T, \cF).
$$
By Corollary \ref{dualselisom} and Theorem \ref{main}(ii), we have $I_{i}(\kappa^{(m+1)})\cR/(p^{m}) \subseteq   I_{i}(\kappa^{(m)})$ for any $\kappa = \{\kappa^{(n)}\} \in {\rm KS}_{r}(T, \cF)$ and non-negative integer $i$. Hence we can define an ideal $I_{i}(\kappa)$ of $\cR$ to be the inverse limit
$$
I_{i}(\kappa) := \varprojlim_{m}I_{i}(\kappa^{(m)}).
$$

\begin{theorem}\label{thm koly'}
Suppose that $p>3$.
\begin{itemize}
\item[(i)] The map ${\rm Reg}_{r} \colon {\rm SS}_{r}(T, \cF) \to {\rm KS}_{r}(T, \cF)$ constructed above is an isomorphism. In particular, the $\cR$-module ${\rm KS}_{r}(T, \cF)$ is free of rank one.

\item[(ii)] For each $\kappa$ in ${\rm KS}_r(T,\cF)$ one has
\[ I_{0}(\kappa) \subseteq   {\rm Fitt}_\cR^0  \left(H^1_{\cF^*}(K, T^\vee(1))^\vee \right),\]
with equality if $\kappa$ is a basis of ${\rm KS}_r(T,\cF)$.

\item[(iii)] Fix $\kappa$ in ${\rm KS}_r(A,\cF)$. Then for each non-negative integer $i$ one has
$$I_i(\kappa)\subseteq   {\rm Fitt}_{\cR}^i(H^1_{\cF^\ast}(K,T^\vee(1))^\vee),$$
with equality if $\cR$ is a principal ideal ring and $\kappa$ is a basis of ${\rm KS}_r(T,\cF)$.

\end{itemize}
\end{theorem}
\begin{proof} Claims (i), (ii) and (iii) follow as direct consequences of the respective claims in Theorem \ref{main}.
\end{proof}

\begin{remark}\label{difficulty remark2} For each $\kappa$ in ${\rm KS}_{r}(T, \cF)$ and each non-negative integer $i$, Theorem~\ref{thm koly'}(i)  allows us to define an ideal of $\cR$ by setting
\[ I'_i(\kappa) := I_i({\rm Reg}_r^{-1}(\kappa)),\]
where the right hand side is as defined in Definition \ref{stark ideal gorenstein}. If $\cR$ is a principal ring and $\kappa$ is a $\cR$-basis of ${\rm KS}_r(T,\cF)$, then Theorems \ref{thm koly'}(iii) and \ref{thm stark'}(ii)(c) combine to imply that $I'_i(\kappa) = I_i(\kappa)$ for all $i$ but we do not know if this is true more generally. \end{remark}


\section{Euler systems and Kolyvagin systems}\label{euler sys sec}

In this section, we shall give a natural construction of higher rank Kolyvagin systems from higher rank Euler systems (see Theorem \ref{derivable1} and Corollary \ref{higher der}).

By using results in previous sections, we shall then show that higher rank Euler systems control Selmer modules (see Corollaries \ref{derivable cor}, \ref{remark surjective} and \ref{main cor}).

\subsection{Definition} \label{euler sys sec 1}

Let $K$ be a number field. Let $p$ be a prime number. Let $Q$ be a finite extension of $\QQ_p$, and $\cO$ the ring of integers of $Q$. Let $\mathcal{Q}$ be a finite dimensional semisimple commutative $Q$-algebra. Let $\cR$ be a semilocal Gorenstein $\cO$-order in $\mathcal{Q}$. Let $T$ be a free $\cR$-module of finite rank with a continuous $\cR$-linear action of $G_K$. We assume that $S_{\rm ram}(T)$ is finite, and choose a finite set $S$ of places of $K$ such that
$$S_{\infty}(K)\cup S_p(K) \cup S_{\rm ram}(T) \subseteq  S.$$
For a prime $\fq \notin S$, we set
$$P_\fq(x):=\det(1-{\rm Fr}_\fq^{-1}x \mid T^\ast(1)) \in \cR[x],$$
where $T^\ast(1):=\Hom_\cR(T, \cR(1))$.
Let $\cK/K$ be an abelian pro-$p$ extension such that all infinite places $v \in S_\infty(K)$ split completely in $\cK$. We define a set of subfields of $\cK/K$ by
$$\Omega(\cK/K):=\{F \mid K\subseteq  F\subseteq  \cK, \text{ $F/K$ is finite}\}.$$
For a field $F$ in $\Omega(\cK/K)$, set
\begin{align*}
S(F) &:= S\cup S_{\rm ram}(F/K),
\\
\cG_F &:= \Gal(F/K).
\end{align*}



In the following, we assume

\begin{hypothesis}\label{hyp free}\
\begin{itemize}
\item[(i)] $H^1(\cO_{F, S(F)}, T)$ is a reflexive $\cR[\cG_F]$-module for every $F \in \Omega(\cK/K)$.
\item[(ii)] $H^0(F, T)=0$ for every $F \in \Omega(\cK/K)$.
\end{itemize}
\end{hypothesis}

\begin{remark}\label{hyp free rem}
Since $\cR$ is a Gorenstein $\cO$-order, Hypothesis~\ref{hyp free}(i) is satisfied if and only if each group $H^1(\cO_{F, S(F)}, T)$ is free as an $\cO$-module. (See \cite[Th. 6.2]{bassgorenstein}.)
\end{remark}

\begin{remark} \label{rem torsion free}
When $\cR=\cO=\ZZ_p$ and $T=\ZZ_p(1)$, Hypothesis \ref{hyp free}(i) is equivalent to the condition that the $p$-completion of the unit group $\cO_{F,S(F)}^\times$ is torsion-free for every $F \in \Omega(\cK/K)$. This condition often appears in the context of Stark conjectures, and we usually choose another set $\Sigma$ of places to avoid assuming the condition, by considering the `$\Sigma$-modified unit group' $\cO_{F,S(F), \Sigma}^\times$ (see \cite{rubinstark}, where our $\Sigma$ is denoted by $T$). For a general $p$-adic representation $T$, we can consider the `$\Sigma$-modified cohomology' in a similar way to avoid assuming Hypothesis \ref{hyp free}(i). For details, see \cite[\S 2.3]{sbA}. In this article, we do not consider such modified cohomology theory for simplicity.
\end{remark}

The definition of higher rank Euler systems is as follows.

\begin{definition}[{\cite[Definition 2.3]{sbA}}]
Let $r$ be a non-negative integer.
An element
$$(c_F)_F \in \prod_{F \in \Omega(\cK/K)} {{\bigcap}}_{\cR[\cG_F]}^rH^1(\cO_{F,S(F)},T)$$
is said to be an Euler system of rank $r$ for ($T,\cK$) if
$${\rm Cor}_{F'/F}(c_{F'})=\left(\prod_{\fq \in S(F')\setminus S(F)} P_\fq ({\rm Fr}_\fq^{-1})\right)c_F \ \text{ in } \ {{\bigcap}}_{\cR[\cG_F]}^r H^1(\cO_{F,S(F')},T)$$
for any $F,F' \in \Omega(\cK/K)$ with $F \subseteq  F'$, where
$${\rm Cor}_{F'/F}: {\bigcap}_{\cR[\cG_{F'}]}^r H^1(\cO_{F',S(F')}, T) \to {\bigcap}_{\cR[\cG_F]}^r H^1(\cO_{F,S(F')},T)$$
is the map induced by the corestriction map.

The set of Euler systems of rank $r$ (for $(T, \cK)$) is denoted by ${\rm ES}_r(T,\cK)$. This has a natural structure of $\cR[[\Gal(\cK/K)]]$-module.
\end{definition}

\begin{remark}\label{remark free}
If Hypothesis \ref{hyp free}(i) is satisfied, then we have
$${\bigcap}_{\cR[\cG_F]}^1 H^1(\cO_{F, S(F)}, T)=H^1(\cO_{F, S(F)}, T)^{\ast \ast}=H^1(\cO_{F, S(F)}, T)$$
for every $F\in \Omega(\cK/K)$, so we can regard an Euler system of rank one as an element in $ \prod_{F \in \Omega(\cK/K)} H^1(\cO_{F,S(F)},T)$. Thus our definition generalizes the classical definition of Euler systems given in \cite[Def. 2.1.1]{R}.
\end{remark}

%
%

\subsection{The canonical Selmer structure}\label{section canonical}

The canonical Selmer structure $\cF_{\rm can}$ on $T$ (see \cite[Def. 3.2.1]{MRkoly}) is the following data:
\begin{itemize}
\item $S(\cF_{\rm can}):=S_\infty(K) \cup S_p(K) \cup S_{\rm ram}(T)$;
\item for $v\in S(\cF_{\rm can})$,
$$H_{\cF_{\rm can}}^1(K_v,T):=\begin{cases}
\ker(H^1(K_v,T) \to  H^1(K_v^{\rm ur}, T \otimes_{\ZZ_p}\QQ_p)) &\text{if $v \notin S_\infty(K)\cup S_p(K)$,}\\
H^1(K_v,T) &\text{if $v \in S_\infty(K)\cup S_p(K)$.}
\end{cases}
$$
\end{itemize}
Here $K_v^{\rm ur}$ denotes the maximal unramified extension of $K_v$.

The significance of this Selmer structure is explained by the following well-known result (taken from \cite[Cor. B.3.5]{R}). 

\begin{lemma} \label{lemma unram} Let $c $ be an Euler system of rank one for $(T,\cK)$. Assume that $\cK$ contains a $\ZZ_p^d$-extension of $K$ for some $d\geq 1$, in which no finite place of $K$ splits completely. Then $c_F$ belongs to $H^1_{\cF_{\rm can}}(F,T)$ for every $F$ in $\Omega(\cK/K)$.
\end{lemma}


In practice, one usually takes $\cK$ to be a sufficiently large abelian pro-$p$ extension. For later purposes, we now state the usual assumptions on $\cK$ as an explicit hypothesis.

\begin{hypothesis}\label{hyp K}
The field $\cK$ contains $K(\fq)$ for every $\fq \notin S$ and a $\ZZ_p^d$-extension of $K$ for some $d\geq 1$, in which no finite place of $K$ splits completely.
\end{hypothesis}

\begin{remark}
This hypothesis is included in the definition of Euler systems given in \cite[Def. 2.1.1]{R}.
\end{remark}


\subsection{Kolyvagin derivatives}\label{koly sect}
We review the construction of `Kolyvagin derivatives' in the higher rank case (see \cite[\S4.3.1]{sbA}).
We fix $M$, a power of $p$. We also fix $E \in \Omega(\cK/K)$ such that $E/K$ is unramified outside $S$ and that $K(1) \subseteq  E$. (Recall that $K(1)$ denotes the maximal $p$-extension inside the Hilbert class field of $K$.) 
We denote $\overline \cR:=\cR/(M)$, $R:=\overline \cR [\Gal(E/K)]$, $A:=T/MT$,
$\cT:={\rm Ind}_{G_K}^{G_E}(T)$, and $\cA:={\rm Ind}_{G_K}^{G_E}(A)=\cT/M\cT$.


We shall recall some notation from \S \ref{section hyp} and set some new notation. We consider the set $\cP$ of primes $\fq \notin S$ such that
\begin{itemize}
\item $\fq$ splits completely in $K(\mu_M,(\cO_K^\times)^{1/M})K(1)$,
\item $\cA/({\rm Fr}_\fq-1)\cA \simeq R$ as $R$-modules.
\end{itemize}
Note that $\cP$ contains that defined in \S \ref{section hyp} if we assume Hypothesis~\ref{hyp1}(ii) for $\cA$.
Let $\cN=\cN(\cP)$ be the set of square-free products of primes in $\cP$.
We set
$$G_\fq :=\Gal(K(\fq)/K(1)) \simeq \Gal(E(\fq)/E),$$
where $E(\fq):=E\cdot K(\fq)$.
For $\fn\in \cN$, we set
$K(\fn):=\prod_{\fq \mid \fn} K(\fq)$ (compositum) and $E(\fn ):=E\cdot K(\fn)$.
We also set $\cG_\fn:=\Gal(E(\fn)/K)$ and $\cH_\fn:=\Gal(E(\fn)/E)$.
Note that we have natural identifications
$$
\cH_\fn=\Gal(K(\fn)/K(1)) =\prod_{\fq \mid \fn} G_\fq.
$$


For an Euler system $c$ of rank $r$ and $\fn \in \cN$, we set
$$
c_\fn:=c_{E(\fn)} \in {{\bigcap}}_{\cR[\cG_\fn]}^r H^1 (\cO_{E(\fn),S_\fn},T),
$$
where $S_\fn :=S\cup \{\fq \mid \fn\}(=S(E(\fn)))$.
Fix a generator $\sigma_\fq$ of $G_\fq$ for each $\fq$, and consider the `derivative operator'
$$
D_\fq:=\sum_{i=1}^{| G_\fq|-1}i\sigma_\fq^i \in \ZZ[G_\fq].
$$
For $\fn \in \cN$, we set
$$
D_\fn:=\prod_{\fq \mid \fn }D_\fq \in \ZZ[\cH_\fn].
$$
(We set $D_1:=1$ for convention.)

Note that the natural `mod $M$' map $T \to T/MT=A$ induces a map
\begin{eqnarray} \label{mod M map}
{{\bigcap}}_{\cR[\cG_\fn]}^r H^1(\cO_{E(\fn),S_\fn},T) \to {{\bigcap}}_{\overline \cR[\cG_\fn]}^r H^1(\cO_{E(\fn),S_\fn},A).
\end{eqnarray}
We explain the construction of this map, since we need Hypothesis~\ref{hyp free}(i) here.
We set $H^1(T):= H^1(\cO_{E(\fn),S_\fn},T)$ and $H^1(A):=H^1(\cO_{E(\fn),S_\fn},A)$ for simplicity. Also, for the moment we denote $\cR[\cG_\fn]$ and $\overline \cR[\cG_\fn]$ simply by $\cR$ and $\overline \cR$ respectively (by abuse of notation). First, note that Hypothesis~\ref{hyp free}(i) implies that
$$\Ext_\cR^1(H^1(T),\cR)=0.$$
(See \cite[\S A.3]{sbA}.) From this, we see that
$$H^1(T)^\ast/M=\Hom_\cR(H^1(T), \cR)/M \simeq \Hom_{\overline \cR}(H^1(T)/M, \overline \cR)= (H^1(T)/M)^\ast,$$
where we abbreviate $X/MX$ to $X/M$. Since there is a natural map $H^1(T)/M \to H^1(A)$, we obtain a map
$$H^1(A)^\ast \to (H^1(T)/M)^\ast \simeq H^1(T)^\ast/M.$$
This map induces a map
\begin{eqnarray} \label{wedge induce}
{\bigwedge}_{\overline \cR}^r H^1(A)^\ast \to {\bigwedge}_{\overline \cR}^r (H^1(T)^\ast/M)=\left( {\bigwedge}_{\cR}^r H^1(T)^\ast \right)/M.
\end{eqnarray}
Then we obtain (\ref{mod M map}) as the following map:
\begin{eqnarray*}
{\bigcap}_\cR^r H^1(T) &\to& \Hom_\cR \left( {\bigwedge}_\cR^r H^1(T)^\ast, \overline \cR \right) \\
&=& \Hom_{\overline \cR} \left( \left({\bigwedge}_{ \cR}^r H^1(T)^\ast\right)/M, \overline \cR  \right) \\
&\stackrel{(\ref{wedge induce})}{\to}& \Hom_{\overline \cR} \left( {\bigwedge}_{\overline \cR}^r H^1(A)^\ast, \overline \cR\right) \\
&=& {\bigcap}_{\overline \cR}^r H^1(A).
\end{eqnarray*}
We denote the image of $c_\fn$ under (\ref{mod M map}) by $\bar c_\fn$. The following is well-known.

\begin{lemma}[{\cite[Lem. 4.4.2(i)]{R}}] \label{lem invariant}
The element $D_\fn \cdot \bar c_\fn$ lies in $\left({{\bigcap}}_{\overline \cR[\cG_\fn]}^r H^1(\cO_{E(\fn),S_\fn},A)\right)^{\cH_\fn}$.
\end{lemma}

\begin{proof}
We use the identity
\begin{eqnarray}\label{telescoping}
(\sigma_\fq-1)D_\fq = |G_\fq| - \N_{G_\fq},
\end{eqnarray}
where $\N_{G_\fq}:=\sum_{\sigma \in G_\fq} \sigma$. (This is checked by direct computation.)

We shall show that
$$(\sigma-1)D_\fn \cdot c_\fn \in M \cdot {\bigcap}_{\cR[\cG_\fn]}^r H^1(\cO_{E(\fn),S_\fn},T)$$
for every $\sigma \in \cH_\fn$. We prove this by induction on $\nu(\fn)$. When $\nu(\fn)=0$, i.e. $\fn=1$, there is nothing to prove. When $\nu(\fn)>0$, it is sufficient to show that
$$(\sigma_\fq-1)D_\fn \cdot c_\fn \in M \cdot {\bigcap}_{\cR[\cG_\fn]}^r H^1(\cO_{E(\fn),S_\fn},T)$$
for every $\fq \mid \fn$
(since the augmentation ideal of $\ZZ[\cH_\fn]$ is generated by the elements
$\sigma_\fq-1$ for any $\fq \mid \fn$).
Using (\ref{telescoping}), we compute
\begin{eqnarray*}
(\sigma_\fq-1)D_\fn c_\fn &=&(| G_\fq| -\N_{G_\fq})D_{\fn/\fq} c_\fn \\
&=&| G_\fq| D_{\fn/\fq}c_\fn - P_\fq({\rm Fr}_\fq^{-1}) D_{\fn/\fq} c_{\fn/\fq},
\end{eqnarray*}
where the second equality follows from the definition of Euler systems. Here we regard $c_{\fn/\fq}$ as an element in ${\bigcap}_{\cR[\cG_\fn]}^r H^1(\cO_{E(\fn),S_\fn},T)$ via the restriction map. (To relate the corestriction map on $\bigcap^r$ with $\N_{G_\fq}$, a slightly delicate consideration is needed, but we omit the detail. See \cite[Rem. 2.12]{sano} concerning this issue.)
Note that
$$P_\fq({\rm Fr}_\fq^{-1}) D_{\fn/\fq} c_{\fn/\fq} =(P_\fq({\rm Fr}_\fq^{-1})-P_\fq(1)) D_{\fn/\fq} c_{\fn/\fq} + P_\fq(1) D_{\fn/\fq} c_{\fn/\fq},$$
and that we know by induction hypothesis that the first term in the right hand side vanishes modulo $M$. Hence, since both $| G_\fq|$ and $P_\fq(1)$ are divisible by $M$ (by the definition of $\cP$), we conclude that $(\sigma_\fq-1)D_\fn c_\fn$ vanishes modulo $M$. This proves the lemma.
\end{proof}


One can also show that
\begin{align*}
\left({{\bigcap}}_{\overline \cR[\cG_\fn]}^r H^1(\cO_{E(\fn),S_\fn},A)\right)^{\cH_\fn} =
{{\bigcap}}_{R}^r H^1(\cO_{E(\fn),S_\fn},A)^{\cH_\fn}.
\end{align*}
(See \cite[Prop. A.4]{sbA}).
Furthermore, under Hypothesis \ref{hyp free}(ii), we have
$$H^1(\cO_{E(\fn),S_\fn},A)^{\cH_\fn}=H^1(\cO_{E,S_\fn},A)=H^1(\cO_{K,S_\fn},\cA)$$
(see \cite[\S4.3.1]{sbA}).
Combining these observations with Lemma \ref{lem invariant}, we have proved the following

\begin{proposition}
Assume Hypothesis \ref{hyp free}. Then for each $\fn \in \cN$ one has
$$\kappa'(c_\fn):=D_\fn \cdot \bar c_\fn \in {{\bigcap}}_R^rH^1 (\cO_{K,S_\fn},\cA).$$
\end{proposition}

The element $\kappa'(c_\fn)$ is called a Kolyvagin derivative.


\subsection{Construction of Kolyvagin systems} \label{construction koly}

In this subsection, we modify the Kolyvagin derivatives $(\kappa'(c_\fn))_\fn$ to construct a Kolyvagin system. Theorem \ref{derivable1} below is the main result of this section.

We write $\cI_\fn$ for the augmentation ideal of $\ZZ[\cH_\fn]$. We recall that the cyclic subgroup of $\cI_\fn^{\nu(\fn)}/\cI_\fn^{\nu(\fn)+1}$ generated by $\prod_{\fq \mid \fn} (\sigma_\fq -1)$ is a direct summand, and is isomorphic to $G_\fn:=\bigotimes_{\fq \mid \fn} G_\fq$:
\begin{align*}
G_\fn=\bigotimes_{\fq \mid \fn} G_\fq &\stackrel{\sim}{\to} \left\langle \prod_{\fq \mid \fn}(\sigma_\fq-1)\right\rangle \subseteq  \cI_\fn^{\nu(\fn)}/\cI_\fn^{\nu(\fn)+1}
\\
\bigotimes_{\fq \mid \fn } \sigma_{\fq} &\mapsto \prod_{\fq \mid \fn}(\sigma_\fq -1).
\end{align*}
(See \cite[Prop. 4.2]{MR}.) We often identify $G_\fn$ with $ \left\langle \prod_{\fq \mid \fn}(\sigma_\fq-1)\right\rangle $. In particular, we regard a Kolyvagin system for $(\cA, \cF)$ (for the definition, see \S \ref{defkoly}) as an element in $\prod_{\fn \in \cN} {\bigcap}_R^r H^1_{\cF(\fn)}(K,\cA) \otimes \left \langle \prod_{\fq \mid \fn}(\sigma_\fq -1)\right \rangle$.

For $\fq \in \cP$, we shall denote $P_\fq({\rm Fr}_\fq^{-1})$ simply by $P_\fq$. If $\fq $ does not divide $\fn\in \cN$, then $\fq$ is unramified in $E(\fn)$, so we can regard $P_\fq \in \cR[\cH_\fn]$. Furthermore, since $P_\fq(1) \equiv 0 \text{ (mod $M$)}$, we can regard $P_\fq \in  \overline \cR \otimes \cI_\fn/\cI_\fn^2$, which we denote by $P_\fq^\fn$. (Recall that $\overline \cR:=\cR/(M)$.) 

For $\fn \in \cN$, we define an element $\cD_\fn \in \overline \cR \otimes \left\langle \prod_{\fq \mid \fn}(\sigma_\fq-1)\right\rangle$ as follows. We write $\fn=\fq_1\cdots \fq_\nu$ ($\nu=\nu(\fn)$). We define
$$\cD_\fn:=\det\left(
\begin{array}{ccccc}
	0 &P_{\fq_1}^{\fq_2} &\cdots & &P_{\fq_1}^{\fq_\nu} \\
	P_{\fq_2}^{\fq_1} & 0 & P_{\fq_2}^{\fq_3} & \cdots & P_{\fq_2}^{\fq_\nu} \\
	\vdots & P_{\fq_3}^{\fq_2} & \ddots & &\vdots\\
	\vdots & \vdots & &\ddots &\vdots\\
	P_{\fq_\nu}^{\fq_1} &P_{\fq_\nu}^{\fq_2} &\cdots & & 0
\end{array}
\right) \in  \overline \cR \otimes \left\langle \prod_{\fq \mid \fn}(\sigma_\fq-1)\right\rangle.$$
(Compare \cite[Def. 6.1]{MR} and \cite[Def. 4.4]{sanojnt}.) One checks that this does not depend on the choice of the labeling $\fq_1,\ldots,\fq_\nu$ of the prime divisors of $\fn$.

Now we consider the following modification of $(\kappa'(c_\fn))_\fn$:
$$\kappa(c)_\fn:=\sum_{\fd \mid \fn} \left(\kappa'(c_\fd) \otimes \prod_{\fq \mid \fd}(\sigma_\fq-1)\right)\cD_{\fn/\fd} \in {\bigcap}_R^r H^1(\cO_{K,S_\fn},\cA) \otimes \left \langle \prod_{\fq \mid \fn}(\sigma_\fq -1)\right \rangle.$$
(Each $\kappa'(c_\fd) \in {\bigcap}_R^r H^1(\cO_{K,S_\fd},\cA)$ is regarded as an element of ${\bigcap}_R^r H^1(\cO_{K,S_\fn},\cA)$.) One easily checks that
$$\sum_{\fd \mid \fn} \left(\kappa'(c_\fd) \otimes \prod_{\fq \mid \fd}(\sigma_\fq-1)\right)\cD_{\fn/\fd} =\sum_{\tau \in \mathfrak{S}(\fn)} {\rm sgn}(\tau)\kappa'(c_{\fd_\tau}) \otimes \prod_{\fq \mid \fd_\tau}(\sigma_\fq-1) \prod_{\fq \mid \fn/\fd_\tau}P_{\tau(\fq)}^\fq,$$
where $\mathfrak{S}(\fn)$ is the set of permutations of prime divisors of $\fn$, and $\fd_\tau:=\prod_{\tau(\fq)=\fq}\fq$. So one can also write
$$ \kappa(c)_\fn= \sum_{\tau \in \mathfrak{S}(\fn)} {\rm sgn}(\tau)\kappa'(c_{\fd_\tau}) \otimes \prod_{\fq \mid \fd_\tau}(\sigma_\fq-1) \prod_{\fq \mid \fn/\fd_\tau}P_{\tau(\fq)}^\fq.$$
Note that this construction is parallel to that given by Mazur and Rubin in \cite[(33) in Appendix A]{MRkoly}.

We need the following hypothesis which corresponds to an assumption in \cite[Th. 3.2.4]{MRkoly}.

\begin{hypothesis}\label{hyp local}
${\rm Fr}_\fq^{p^k}-1$ is injective on $T$ for every $\fq \in \cP$ and $k \geq 0$.
\end{hypothesis}
%

Now we state the main theorem of this section.

\begin{theorem} \label{derivable1}
Let $r$ be a positive integer and $c \in {\rm ES}_r(T,\cK)$. Let $\cF:=\cF_{\rm can}$ be the canonical Selmer structure (see \S \ref{section canonical}). Assume Hypotheses \ref{hyp free}, \ref{hyp K}
and \ref{hyp local}.
Then, for every $\fn \in \cN$, we have
$$
\kappa(c)_\fn \in {{\bigcap}}_R^r H^1_{\cF(\fn)}(K,\cA) \otimes \left \langle \prod_{\fq \mid \fn}(\sigma_\fq -1)\right \rangle
$$
and
$$
v_\fq(\kappa(c)_\fn)=\varphi_\fq^{\rm fs}(\kappa(c)_{\fn/\fq}). 
$$
for every $\fq \mid \fn$. In particular, $\kappa(c):=(\kappa(c)_\fn)_\fn \in {\rm KS}_r(\cA,\cF)$.
\end{theorem}

The proof of Theorem \ref{derivable1} will be given in the next subsection. For the moment, however, we use the result to derive several important consequences.

\begin{corollary} \label{higher der}
Let $r$ and $\cF$ be as in Theorem \ref{derivable1}. Assume Hypotheses \ref{hyp free}, \ref{hyp K} and \ref{hyp local}.
Choose a subfield $F$ of $E/K$ and set $A_F:={\rm Ind}_{G_K}^{G_F}(T/MT)$. Then there is a canonical `higher Kolyvagin derivative' homomorphism
$$\cD_r=\cD_r^F: {\rm ES}_r(T,\cK) \to {\rm KS}_r(A_F,\cF).$$
\end{corollary}

\begin{proof}
The construction of Kolyvagin systems given in Theorem \ref{derivable1} gives a homomorphism
$${\rm ES}_r(T,\cK) \to {\rm KS}_r(\cA,\cF); \ c \mapsto \kappa(c).$$
The map $\cD_r^F$ is obtained by composing this map with the natural map
$${\rm KS}_r(\cA,\cF)\to {\rm KS}_r(A_F,\cF)$$
induced by $\cA(={\rm Ind}_{G_K}^{G_E}(T/MT) ) \to A_F$.
\end{proof}

\begin{remark}\label{remark E}
Although we take $F$ as a subfield of $E/K$ in Corollary~\ref{higher der}, this condition is not essential, since for an arbitrary finite abelian ($p$-)extension $F/K$ one can take $E$ so that $F \cdot K(1) \subseteq E$ by enlarging $\cK$ and $S$ if necessary. The role of the field $E$ is auxiliary (in fact, $\cD_r^F$ is independent of the choice of $E$), and so one can think of $F$ in Corollary~\ref{higher der} as arbitrary.
\end{remark}





\begin{corollary}\label{derivable cor} Suppose $p>3$. Let $r $ be a positive integer, $c \in {\rm ES}_r(T,\cK)$ and $\cF:=\cF_{\rm can}$. Choose a subfield $F$ of $E/K$ and set $A_F:={\rm Ind}_{G_K}^{G_F}(T/MT)$. Assume Hypotheses \ref{hyp free}, \ref{hyp K} and \ref{hyp local}, and Hypotheses \ref{hyp1}, \ref{hyp2} and \ref{hyp large} for $A_F$ and $\cF$. Let $\kappa(c):=\cD_r^F(c) \in {\rm KS}_r(A_F,\cF)$ be the Kolyvagin system constructed in Corollary \ref{higher der}.
\begin{itemize}
\item[(i)] For $\fn$ in $\cN$ one has $\im (\kappa(c)_\fn) \subseteq  {\rm Fitt}_{\overline \cR[\cG_F]}^0(H^1_{\cF(\fn)^\ast}(K,A_F^\ast(1))^\ast)$.
 In particular, one has
$$
\im(c_F)\subseteq  {\rm Fitt}_{\overline \cR[\cG_F]}^0(H^1_{\cF^\ast}(K,A_F^\ast(1))^\ast).
$$
Here we regard $c_F \in {\bigcap}_{\cR[\cG_F]}^r H^1(\cO_{F,S},T)$ as an element in ${\bigcap}_{\overline \cR[\cG_F]}^r H^1(\cO_{F,S},A)\simeq \bigcap_{\overline \cR[\cG_F]}^r H^1(\cO_{K,S},A_F)$ by using the natural map (\ref{mod M map}).
\item[(ii)] 
For every non-negative integer $i$, we have
$$I_i(\kappa(c))\subseteq  {\rm Fitt}_{\overline \cR[\cG_F]}^i(H^1_{\cF^\ast}(K,A_F^\ast(1))^\ast).$$
\end{itemize}
%
%
\end{corollary}

\begin{proof}
This follows directly from Theorems~\ref{main}(ii) and (iii), noting that $\kappa(c)_1=c_F$.
\end{proof}

\begin{remark}
The set $\cP$ defined in \S \ref{koly sect} for $\cA$ is in general smaller than the corresponding set that is defined in \S \ref{section hyp} for $A_F$. However, this difference does not matter since, as long as we can choose a subset of $\cP$ of positive density as in Lemma \ref{chebotarev} (by the Chebotarev density theorem), the theory of Stark and Kolyvagin systems work. We shall implicitly consider such a smaller set $\cP$ also in the statements of the results below.
\end{remark}

 In the following result we recall the ideals $I'_i(\kappa)$ from Remark \ref{difficulty remark}.

 \begin{corollary} \label{remark surjective}
 Let $p$, $r$, $\cF$, $F$ and $A_F$ be as in Corollary \ref{derivable cor}. For $c \in {\rm ES}_r(T,\cK)$, we set $\kappa(c):=\cD_r^F(c) \in {\rm KS}_r(A_F,\cF)$. Assume Hypotheses~\ref{hyp free}, \ref{hyp K} and \ref{hyp local}, and Hypotheses~\ref{hyp1}, \ref{hyp2} and \ref{hyp large} for $\cA$ and $\cF$. 

We also consider the following additional hypotheses.
 \begin{itemize}
\item[(a)] $Y_K(T):=\bigoplus_{v \in S_\infty(K)} H^0(K_v, T^\ast(1))$ is a free $\cR$-module of rank $r$;
\item[(b)] $H^0(K_v,\mathcal{A}^*(1))$ vanishes for each prime $v \in S \setminus S_{\infty}(K)$.
\end{itemize}
 Then the following claims are valid.
\begin{itemize}
\item[(i)] One has
$$\langle  \im (c_F) \mid c \in {\rm ES}_r(T,\cK)\rangle_{\overline \cR[\cG_F]} \subseteq
{\rm Fitt}_{\overline \cR[\cG_F]}^0(H^1_{\cF^\ast}(K,A_F^\ast(1))^\ast),$$
with equality if hypotheses (a) and (b) are satisfied.

\item[(ii)] For each non-negative integer $i$ one has
$$\langle  I'_i(\kappa(c)) \mid c \in {\rm ES}_r(T,\cK)\rangle_{\overline \cR[\cG_F]} \subseteq
{\rm Fitt}_{\overline \cR[\cG_F]}^i(H^1_{\cF^\ast}(K,A_F^\ast(1))^\ast),$$
with equality if hypotheses (a) and (b) are satisfied.
%
\item[(iii)] If $\overline \cR[\cG_F]$ is a principal ideal ring, then for each non-negative integer $i$ one has
$$\langle  I_i(\kappa(c)) \mid c \in {\rm ES}_r(T,\cK)\rangle_{\overline \cR[\cG_F]} \subseteq  {\rm Fitt}_{\overline \cR[\cG_F]}^i(H^1_{\cF^\ast}(K,A_F^\ast(1))^\ast),$$
with equality if hypotheses (a) and (b) are satisfied.
\end{itemize}
\end{corollary}

\begin{proof} By Theorems~\ref{main} and \ref{thm stark} (and Remark \ref{difficulty remark}), it is sufficient to show that the validity of the given hypotheses (a) and (b) imply the existence of an Euler system $c$ such that $\kappa(c)$ is a basis of ${\rm KS}_r(A_F,\cF)$, or equivalently, that the homomorphism
 $$\cD_r^F: {\rm ES}_r(T,\cK) \to {\rm KS}_r(A_F,\cF)$$
 is surjective.

 Since the natural map
 $${\rm KS}_r(\cA, \cF) \to {\rm KS}_r(A_F,\cF)$$
 is surjective (by Theorem \ref{main}(i) and the argument in the proof of Lemma \ref{compatible}(i)), it is thus enough to show surjectivity of the map
 $$\cD_r=\cD_r^E: {\rm ES}_r(T,\cK) \to {\rm KS}_r(\cA,\cF),$$
 or equivalently, surjectivity of the composite
 $${\rm Reg}_r^{-1} \circ \cD_r: {\rm ES}_r(T,\cK) \to {\rm SS}_r(\cA,\cF).$$
 We prove surjectivity of this map by using results of the first and the third author in \cite{sbA}.

 To do this we recall that the Selmer structure $\cF_S$ on $\cA$ that is considered in \cite{sbA} is defined by setting
 \begin{itemize}
 \item $S(\cF_S):=S$;
\item for $v \in S$, $H^1_{\cF_S}(K_v,\cA):=H^1(K_v,\cA).$
 \end{itemize}

We also note that the stated hypotheses (a) and (b) above correspond to \cite[Hyp. 2.11 and 3.9]{sbA}.
 We shall recall some constructions given in \cite{sbA}.

 By \cite[Th. 2.17]{sbA}, under hypotheses (a) and \ref{hyp free}, there is a homomorphism
 $$\theta_{T,\cK}:{\rm VS}(T,\cK) \to {\rm ES}_r(T,\cK),$$
 where ${\rm VS}(T,\cK)$ is the module of `vertical determinantal systems' (see \cite[Def. 2.8]{sbA}). We define $\mathcal{E}^{\rm b}(T,\cK):= \im (\theta_{T,\cK})$. This is called the module of `basic Euler systems' (see \cite[Def. 2.18]{sbA}).

By \cite[Th. 3.11(ii)]{sbA}, under hypotheses (b) and \ref{hyp1}, there is an isomorphism
$$
{\rm HS}(\cA) \stackrel{\sim}{\to} {\rm SS}_r(\cA,\cF_S),
$$
where ${\rm HS}(\cA)$ is the module of `horizontal determinantal systems' (see \cite[Def. 3.2]{sbA}).
There is a natural surjection ${\rm VS}(T,\cK) \to {\rm HS}(\cA)$ (as in \cite[\S 4.3.2]{sbA}).

The argument in the proof of \cite[Th. 4.16]{sbA} shows the following diagram is commutative:
$$
\xymatrix{
\mathcal{E}^{\rm b} (T,\cK) \ar[rr]^{{\rm Reg}_r^{-1}\circ \cD_r}&  & {\rm SS}_r(\cA,\cF) \ar@{^{(}->}[r] &{\rm SS}_r(\cA,\cF_S)
\\
{\rm VS}(T,\cK) \ar@{->>}[u]^{\theta_{T,\cK}} \ar@{->>}[rr] & & {\rm HS}(\cA). \ar[ur]_{\simeq} &
}
$$
This diagram shows that ${\rm Reg}_r^{-1}\circ \cD_r$ is surjective, as required.
\end{proof}

  \begin{corollary} \label{main cor}
Suppose $p>3$. Let $r $ be a positive integer, $c \in {\rm ES}_r(T,\cK)$ and $\cF:=\cF_{\rm can}$. Choose a subfield $F$ of $E/K$ and set $T_F:={\rm Ind}_{G_K}^{G_F}(T)$. Assume Hypotheses~\ref{hyp free}, \ref{hyp K}, \ref{hyp local} and \ref{hyp1'} for $T$ and Hypothesis~\ref{hyp large} for $(T_F/p^{m}T_F, \cF, \cP_{m})$ for all positive integers $m$.
\begin{itemize}
\item[(i)] One has $\im(c_F) \subseteq   {\rm Fitt}_{\cR[\cG_F]}^0(H^1_{\cF^*}(K, T_F^\vee(1))^\vee )$.
\item[(ii)] Let $\kappa(c)_m:=\cD_r^F(c) \in {\rm KS}_r(T_F/p^m T_F,\cF)$ be the Kolyvagin system constructed in Corollary \ref{higher der} (with $M=p^m$). We set
$$\kappa(c):=(\kappa(c)_m)_m \in \varprojlim_m {\rm KS}_r(T_F/p^mT_F,\cF)={\rm KS}_r(T_F,\cF).$$
Then for each non-negative integer $i$ one has $I_i(\kappa(c))\subseteq \! {\rm Fitt}_{\cR[\cG_F]}^i(H^1_{\cF^\ast}(K,T_F^\vee(1))^\vee). $
%
\item[(iii)] Assume, in addition, that $\bigoplus_{v \in S_\infty(K)} H^0(K_v, T^\ast(1))$ is a free $\cR$-module of rank $r$ and that $H^{0}(E_{w}, (T/pT)^\vee(1))$ vanishes for all non-archimedean places $w$ of $E$ above $S$. Then for each non-negative integer $i$ one has
$$ \langle  I'_i(\kappa(c)) \mid c \in {\rm ES}_r(T,\cK)\rangle_{ \cR[\cG_F]} = {\rm Fitt}_{ \cR[\cG_F]}^i(H^1_{\cF^\ast}(K,T_F^\vee(1))^\vee). $$
If $\cR[\cG_F]$ is a principal ideal ring, then one can replace the ideals $I'_i(\kappa(c))$ by $I_i(\kappa(c))$ in this equality.
\end{itemize}
\end{corollary}

\begin{proof}
Claims (i) and (ii) follow directly from Theorem~\ref{thm koly'}(ii) and (iii) (after noting that $I_0(\kappa(c))=\im (c_F)$.)

In a similar way, since the additional hypotheses in claim (iii) are equivalent to the validity of the hypotheses (a) and (b) in Corollary~\ref{remark surjective} for the representations $T$ and ${\rm Ind}_{G_K}^{G_E}(T/pT)$ respectively, this claim is a straightforward consequence of Corollary~\ref{remark surjective}(iii) and Remark~\ref{difficulty remark2}.
\end{proof}

\begin{remark}\label{remark p=3}
In results above, we excluded the case $p=3$. This is due to the technical condition on $p$ in Lemma \ref{chebotarev}. However, it is possible to treat the case $p=3$, if we assume that $T$ is not `self-dual'. In fact, Mazur and Rubin proved the result corresponding to Lemma \ref{chebotarev} in \cite[Prop. 3.6.1]{MRkoly} under their running hypotheses, one of which is `either $T$ is not self-dual or $p>4$' (see \cite[(H.4) in \S 3.5]{MRkoly}).
\end{remark}

\subsection{The proof of Theorem \ref{derivable1}}


In this subsection, we prove Theorem~\ref{derivable1}.

When $r=1$, this result can be proved by using the argument of Mazur and Rubin in \cite[Th. 3.2.4]{MRkoly}. More precisely, whilst our setting is more general than that of \cite{MRkoly}, since we work over a general number field $K$ and the coefficient ring of $T$ is a general Gorenstein order, it can be checked that the method of the proof of \cite[Th. 3.2.4]{MRkoly} also applies in this more general setting.

The essential idea is, therefore, to reduce the general case of Theorem~\ref{derivable1} to the case that $r=1$.


Throughout this subsection, we assume Hypotheses~\ref{hyp free}, \ref{hyp K} and \ref{hyp local}.
We remark that
$$
\Psi = (\Psi_F)_F \in \varprojlim_{F \in \Omega(\cK/K)} {{\bigwedge}}_{\cR[\cG_F]}^{r-1}H^1(\cO_{F,S(F)},T)^\ast
$$
induces a homomorphism
$$
\Psi \colon {\rm ES}_r(T,\cK) \to {\rm ES}_1(T,\cK); \ (c_F)_F\to (\Psi_F(c_F))_F.
$$
This construction was introduced by Rubin in \cite[\S6]{rubinstark} and also used by Perrin-Riou in \cite[\S1.2.3]{PR}.

The following lemma is a key.

\begin{lemma}\label{key lemma}
Let $\fn \in \cN$. Then, for every $\Phi \in {\bigwedge}_R^{r-1}H^1(\cO_{K,S_\fn},\cA)^\ast$, there exists $\Psi \in \varprojlim_F {\bigwedge}_{\cR[\cG_F]}^{r-1}H^1(\cO_{F,S(F)},T)^\ast$ such that for any $\fd \mid \fn$ we have
$$\Phi(\kappa'(c_\fd))=\kappa'(\Psi(c)_\fd) \text{ in } H^1_{\cF^\fn}(K,\cA).$$
(Note that $\Psi(c) \in {\rm ES}_1(T,\cK)$, and $\kappa'(\Psi(c)_\fd)$ denotes the Kolyvagin derivative of the rank one Euler system $\Psi(c)$.)
\end{lemma}

Before proving this lemma, we use it to give a proof of Theorem \ref{derivable1}.

\begin{proof}[Proof of Theorem \ref{derivable1}]
We first show that, for each $\fn \in \cN$, the element
$$\kappa(c)_\fn=\sum_{\fd \mid \fn} \left(\kappa'(c_\fd) \otimes \prod_{\fq \mid \fd}(\sigma_\fq-1)\right)\cD_{\fn/\fd} $$
lies in ${{\bigcap}}_R^r H^1_{\cF(\fn)}(K,\cA) \otimes \left \langle \prod_{\fq \mid \fn}(\sigma_\fq -1)\right \rangle.$
By Proposition \ref{reduction}, it is sufficient to show that
\begin{equation} \label{trans}
\Phi(\kappa(c)_\fn)=\sum_{\fd \mid \fn} \left(\Phi(\kappa'(c_\fd)) \otimes \prod_{\fq \mid \fd}(\sigma_\fq-1)\right)\cD_{\fn/\fd} \in H^1_{\cF(\fn)}(K,\cA) \otimes \left \langle \prod_{\fq \mid \fn}(\sigma_\fq -1)\right \rangle
\end{equation}
for every $\Phi \in {\bigwedge}_R^{r-1}H^1(\cO_{K,S_\fn},\cA)^\ast$.

By Lemma \ref{key lemma}, there exists $\Psi \in \varprojlim_F {\bigwedge}_{\cR[\cG_F]}^{r-1}H^1(\cO_{F,S(F)},T)^\ast$ such that
$$\Phi(\kappa'(c_\fd))=\kappa'(\Psi(c)_\fd)$$
for every $\fd \mid \fn$, so (\ref{trans}) follows from the fact that Theorem \ref{derivable1} holds for the rank one Euler system $\Psi(c)$.


Next, we show that
$$v_\fq(\kappa(c)_\fn)=\varphi_\fq^{\rm fs}(\kappa(c)_{\fn/\fq}) 
$$
for every $\fn \in \cN$ and $\fq \mid \fn$. As in Definition \ref{koly ideal}, we fix an identification
$${\bigcap}_R^r H^1_{\cF(\fn)}(K,\cA)\otimes G_\fn={\bigcap}_R^r H^1_{\cF(\fn)}(K,\cA)$$
for each $\fn \in \cN$.
We note that, by definition,
$${{\bigcap}}_R^r H^1_{\cF(\fn)}(K,\cA)=\left( {{\bigwedge}}_R^r H^1_{\cF(\fn)}(K,\cA)^\ast\right)^\ast,$$
so $\kappa(c)_\fn \in {{\bigcap}}_R^r H^1_{\cF(\fn)}(K,\cA)$ is a map
$$\kappa(c)_\fn: {{\bigwedge}}_R^r H^1_{\cF(\fn)}(K,\cA)^\ast \to R.$$
We also note that $v_\fq(\kappa(c)_\fn)$ is the map
$${{\bigwedge}}_R^{r-1}H^1_{\cF(\fn)}(K,\cA)^\ast \to R; \ \Phi \mapsto \kappa(c)_\fn(v_\fq \wedge \Phi).$$
Since we identify ${\bigcap}_R^1 H^1_{\cF(\fn)}(K,\cA)=H^1_{\cF(\fn)}(K,\cA)$ and regard $\Phi(\kappa(c)_\fn) \in H^1_{\cF(\fn)}(K,\cA)$, we have
$$\kappa(c)_\fn(v_\fq \wedge \Phi)=(-1)^{r-1} v_\fq(\Phi(\kappa(c)_\fn)).$$
Similarly, $\varphi_\fq^{\rm fs}(\kappa(c)_{\fn/\fq}) \in {{\bigcap}}_R^{r-1}H^1_{\cF(\fn)}(K,\cA)$ is the map
$${{\bigwedge}}_R^{r-1}H^1_{\cF(\fn)}(K,\cA)^\ast \to R; \ \Phi \mapsto \kappa(c)_{\fn/\fq}(\varphi_\fq^{\rm fs} \wedge \Phi)=(-1)^{r-1}\varphi_\fq^{\rm fs}(\Phi(\kappa(c)_{\fn/\fq})).$$
So, to prove the equality of maps $v_\fq(\kappa(c)_\fn)=\varphi_\fq^{\rm fs}(\kappa(c)_{\fn/\fq})$, it is sufficient to prove that they send each $\Phi \in {{\bigwedge}}_R^{r-1}H^1_{\cF(\fn)}(K,\cA)^\ast$ to the same element, namely,
\begin{eqnarray} \label{fs}
v_\fq(\Phi(\kappa(c)_\fn))=\varphi_\fq^{\rm fs}(\Phi(\kappa(c)_{\fn/\fq}))
\end{eqnarray}
for every $\Phi \in {{\bigwedge}}_R^{r-1}H^1_{\cF(\fn)}(K,\cA)^\ast$.

By Lemma \ref{key lemma}, there exists $\Psi \in \varprojlim_F {\bigwedge}_{\cR[\cG_F]}^{r-1}H^1(\cO_{F,S(F)},T)^\ast$ such that
$$\Phi(\kappa(c)_\fn)=\kappa(\Psi(c))_\fn \text{ and }\Phi(\kappa(c)_{\fn/\fq})=\kappa(\Psi(c))_{\fn/\fq},$$
So (\ref{fs}) follows again from the fact that Theorem \ref{derivable1} holds for the rank one Euler system $\Psi(c)$.
\end{proof}

In the rest of this section we prove Lemma \ref{key lemma}.

At the outset we fix $\Phi$ in ${\bigwedge}_{R}^{r-1}H^1(\cO_{K,S_\fn},\cA)^\ast$. Then, since the restriction map
$$
{{\bigwedge}}_{\overline \cR [\cG_\fn]}^{r-1} H^1(\cO_{E(\fn),S_\fn},A)^\ast \to {{\bigwedge}}_R^{r-1}H^1(\cO_{E,S_\fn},A)^\ast= {{\bigwedge}}_{R}^{r-1}H^1(\cO_{K,S_\fn},\cA)^\ast
$$
is surjective (since rings we consider here are self-injective),
we can choose a lift $\widetilde \Phi \in{{\bigwedge}}_{\overline \cR[\cG_\fn]}^{r-1} H^1(\cO_{E(\fn),S_\fn},A)^\ast $ of $\Phi \in {\bigwedge}_{R}^{r-1}H^1(\cO_{K,S_\fn},\cA)^\ast$.

We regard $\widetilde \Phi$ as an element of ${{\bigwedge}}_{\overline \cR[\cG_\fn]}^{r-1}H^1(\cO_{E(\fn),S_\fn},T)^\ast/M$ via the map
$$
{{\bigwedge}}_{\overline \cR[\cG_\fn]}^{r-1} H^1(\cO_{E(\fn),S_\fn},A)^\ast   \to{{\bigwedge}}_{\overline \cR[\cG_\fn]}^{r-1}H^1(\cO_{E(\fn),S_\fn},T)^\ast/M
$$
induced by $T \to T/M=A$ (see (\ref{wedge induce})).
We also have a surjective homomorphism
$$
{{\bigwedge}}_{\cR[\cG_\fn]}^{r-1} H^1(\cO_{E(\fn),S_\fn},T)^\ast \to {{\bigwedge}}_{\overline \cR[\cG_\fn]}^{r-1}H^1(\cO_{E(\fn),S_\fn},T)^\ast/M,
$$
and we fix a lift $\Psi_\fn \in {{\bigwedge}}_{\cR[\cG_\fn]}^{r-1} H^1(\cO_{E(\fn),S_\fn},T)^\ast$ of $\widetilde \Phi \in  {{\bigwedge}}_{\overline \cR[\cG_\fn]}^{r-1}H^1(\cO_{E(\fn),S_\fn},T)^\ast/M$.
Then, since the transition maps of the inverse limit
$$
\varprojlim_{F\in \Omega(\cK/K)} {{\bigwedge}}_{\cR[\cG_F]}^{r-1} H^1(\cO_{F,S(F)},T)^\ast
$$
are surjective (by Hypothesis \ref{hyp free} and \cite[Lem. 2.10]{sano}), one can take
$$
\Psi=(\Psi_F)_F \in \varprojlim_{F\in \Omega(\cK/K)} {{\bigwedge}}_{\cR[\cG_F]}^{r-1} H^1(\cO_{F,S(F)},T)^\ast
$$
such that $\Psi_{E(\fn)}=\Psi_\fn$.

We shall show that this $\Psi$ satisfies the condition in Lemma \ref{key lemma}, namely that
$$\Phi(\kappa'(c_\fd))=\kappa'(\Psi(c)_\fd)$$
for all $\fd \mid \fn$.

We first note that
$$\kappa'(\Psi(c)_\fd)=D_\fd \cdot \widetilde \Phi_\fd(\bar c_\fd) \text{ in }H^1(\cO_{E(\fd),S_\fd},A),$$
where $\widetilde \Phi_\fd \in {\bigwedge}_{\overline \cR[\cG_\fd]}^{r-1}H^1(\cO_{E(\fd),S_\fd},A)^\ast$ is the restriction of $\widetilde \Phi \in{{\bigwedge}}_{\overline \cR[\cG_\fn]}^{r-1} H^1(\cO_{E(\fn),S_\fn},A)^\ast$.
Hence one has
$$\kappa'(\Psi(c)_\fd)=D_\fd \widetilde \Phi_\fd(\bar c_\fd)=\widetilde \Phi_\fd(D_\fd \bar c_\fd)=\Phi(\kappa'(c_\fd)),$$
where the last equality follows from the fact that $\widetilde \Phi_\fd$ is a lift of $\Phi$ by construction.
This completes the proof of Lemma~\ref{key lemma}.

\section{Rubin-Stark elements and ideal class groups}\label{app sec}

In this final section, we give a straightforward application of our theory in the original setting considered by Rubin in \cite{rubinstark}.

In this setting we shall (unconditionally) prove a strong refinement of previous results of Rubin and of B\"uy\"ukboduk that were obtained under the assumed validity of Leopoldt's Conjecture. In addition, by a slightly more careful application of the same methods one can also prove much stronger result in this direction (see Remark \ref{promise} below).

At the outset we fix an odd prime number $p$ (see Remark~\ref{remark p=3}).
We also fix a number field $K$ and a homomorphism
$$\chi: G_K \to \overline \QQ^\times,$$
that has finite prime-to-$p$ order. We fix an embedding $\overline \QQ \hookrightarrow \overline \QQ_p$ and set $\cO:=\ZZ_p[\im(\chi)]$.
We write $L$ for the field extension of $K$ that corresponds to $\ker (\chi)$ and set $\Delta:=\Gal(L/K)$. {\it We suppose that all archimedean  places of $K$ split completely in $L$}. (In particular, if $K$ is totally real, then we assume that $\chi$ is totally even.)

For a $\ZZ_p[\Delta]$-module $X$, we define its `$\chi$-part' by
$$X^\chi:=\{a \in \cO \otimes_{\ZZ_p} X \mid \sigma (a) = \chi(\sigma) a \text{ for every $\sigma \in \Delta$}\}.$$
We note that, since $|\Delta|$ is prime to $p$, this module is naturally isomorphic to $\cO\otimes_{\ZZ_p[\Delta]} X$, where $\cO$ is regarded as a $\ZZ_p[\Delta]$-algebra via $\chi$.

We write $T_\chi$ for a free $\cO$-module of rank one upon which $G_K$ acts by the rule
$$\sigma \cdot a:=\chi_{\rm cyc}(\sigma)\chi^{-1}(\sigma) a\quad (\sigma \in G_K, \ a \in T),$$
where $\chi_{\rm cyc}:G_K \to \ZZ_p^\times$ denotes the cyclotomic character. This $T_\chi$ is usually denoted by $\cO(1)\otimes \chi^{-1}$.

Let $\cF_{\rm can}$ be the canonical Selmer structure on $T_\chi$ and recall that there is a natural identification
\begin{equation}\label{selmer ident} H^1_{(\cF_{\rm can})^{*}}(K,T_\chi^\vee(1))^\vee \simeq (\ZZ_p \otimes_\ZZ {\rm Cl}(\cO_L[1/p]))^\chi.\end{equation}
Here ${\rm Cl}(\cO_{L}[1/p])$ denotes the quotient of the ideal class group ${\rm Cl}(\mathcal{O}_L)$ of $L$ by the subgroup generated by the classes of all prime ideals that divide $p$. 

We set $r:=|S_\infty(K)|$ and fix a finite set $S$ of places of $K$ such that
$$S_\infty(K)\cup  S_{\rm ram}(L/K) \subseteq S.$$
We always assume that $|S| > r$. (In particular, if $L/K$ is ramified then we can take $S = S_\infty(K)\cup  S_{\rm ram}(L/K)$.) 

We quickly recall some notations from \S\ref{euler sys sec 1}. We fix a pro-$p$ abelian extension $\cK/K$ that is sufficiently large to ensure  Hypothesis \ref{hyp K} is satisfied (for $S\cup S_p(K)$), and we write $\Omega(\cK/K)$ for the set of subfields of $\cK/K$ that are finite over $K$. For a finite abelian extension $F/K$ we set $\cG_F:=\Gal(F/K)$ and $S(F):=S \cup S_{\rm ram}(F/K)$. For a set $\Sigma$ of places of $K$ with $S(F) \subseteq \Sigma$, we denote by $\Sigma_F$ the set of places of $F$ which lie above a place in $\Sigma$. The ring of $\Sigma_F$-integers of $F$ is denoted by $\cO_{F,\Sigma}$ and we write $\theta_{F/K,\Sigma}(s)$ for the $\Sigma$-truncated equivariant $L$-function for $F/K$ (as defined, for example, in \cite[\S 3.1]{bks1}).

Since $p$ is odd and $\cK/K$ is a pro-$p$ extension, all places in $S_\infty(K)$ split completely in $\cK$. We label, and hence order, the places in $S_\infty(K)$ as $\{v_1,\ldots,v_r\}$ and for each $v_i$ we fix a place $w_i$ of $\overline \QQ$ that lies above $v_i$. Since $|S| > r$ we can also fix a non-archimedean place $v_0$ in $S$ and a place $w_0$ of $\overline \QQ$ lying above $v_0$.

Then for each $F \in \Omega(\cK/K)$ and each finite set $\Sigma$ of places of $K$ with $S(F) \subseteq \Sigma$, the Dirichlet regulator map induces an isomorphism of $\RR[\cG_{LF}]$-modules
$$\RR \otimes_\ZZ {\bigwedge}_{\ZZ[\cG_{LF}]}^r \cO_{LF,\Sigma}^\times \stackrel{\sim}{\to} \RR  \otimes_\ZZ {\bigwedge}_{\ZZ[\cG_{LF}]}^r X_{LF,\Sigma}$$
where $X_{LF,\Sigma}$ denotes the kernel of the map $\bigoplus_{w \in \Sigma_{LF}}\ZZ \cdot w \to \ZZ$ that sends $\sum_w a_w w$ to $\sum_w a_w$.

Under the stated assumptions on $S$, the functions $s^{-r}\theta_{LF/K,\Sigma}(s)$ are holomorphic at $s=0$ (see \cite[Chap. I, Prop. 3.4]{tatebook}) and, following Rubin \cite{rubinstark}, one defines the `Rubin-Stark element' $\eta_{LF/K,\Sigma}$ to be the unique element of $\RR \otimes_\ZZ {\bigwedge}_{\ZZ[\cG_{LF}]}^r  \cO_{LF,\Sigma}^\times$ that the above isomorphism sends to
$$
\underset{s\to 0}{\lim} s^{-r}\theta_{LF/K,\Sigma}(s) \cdot (w_1-w_0)\wedge \cdots \wedge (w_r-w_0).
$$
If $\Sigma$ contains $S_p(K)$, then Kummer theory induces canonical isomorphisms
$$(\ZZ_p\otimes_\ZZ \cO_{LF,\Sigma}^\times)^\chi\simeq H^1(\cO_{LF,\Sigma},\ZZ_p(1))^\chi  \simeq H^1(\cO_{F,\Sigma}, T_\chi).$$
(Note that, since $[L:K]$ is prime to $p$, $L$ is disjoint from $\cK$ and so one can define the $\chi$-component of a $\ZZ_p[\cG_{LF}]$-module.)

In particular, after fixing an embedding $\RR \hookrightarrow \CC_p$ we define $\eta_{LF/K,\Sigma}^\chi$ to be the image of $\eta_{LF/K,\Sigma}$ under the composite
\begin{eqnarray*}
\RR \otimes_\ZZ {\bigwedge}_{\ZZ[\cG_{LF}]}^r  \cO_{LF,\Sigma}^\times &\subseteq& \CC_p \otimes_{\ZZ_p} {\bigwedge}_{\ZZ_p[\cG_{LF}]}^r H^1(\cO_{LF,\Sigma\cup S_p(K)},\ZZ_p(1)) \\
& \to& \CC_p \otimes_{\ZZ_p} {\bigwedge}_{\cO[\cG_F]}^r H^1(\cO_{F,\Sigma\cup S_p(K)},T_\chi)
\end{eqnarray*}
with the second map induced by the projection
$$H^1(\cO_{LF,\Sigma\cup S_p(K)},\ZZ_p(1))\to H^1(\cO_{LF,\Sigma\cup S_p(K)},\ZZ_p(1))^\chi=H^1(\cO_{F,\Sigma\cup S_p(K)},T_\chi).$$

For each $F$ in $\Omega(\cK/K)$ we set $S(F)_p:=S\cup S_{\rm ram}(F/K) \cup S_p(K)(=S(F)\cup S_p(K))$ and {\it we assume that the group $(\ZZ_p\otimes_\ZZ \cO_{LF,S(F)_p}^\times)^\chi$ is a free $\cO$-module.}

Under this hypothesis the Rubin-Stark conjecture \cite[Conj. B$'$]{rubinstark} predicts that the element $c_{F,\chi}^{\rm RS}:=\eta_{LF/K, S(F)}^\chi$
 belongs to ${\bigcap}_{\cO[\cG_{F}]}^r H^1(\cO_{F,S(F)_p},T_\chi)$ (regarded as a submodule of $\CC_p \otimes_{\ZZ_p} {\bigwedge}_{\cO[\cG_F]}^r H^1(\cO_{F,S(F)_p},T_\chi)$ via \cite[Prop. A.7]{sbA}).

In addition, if this conjecture is valid for every $F$ in $\Omega(\cK/K)$, then the collection
\[ c^{\rm RS}_\chi := (c_{F,\chi}^{\rm RS})_F\]
forms an Euler system of rank $r$ for the pair $(T_\chi,\cK)$ (for a proof of this see \cite[Prop. 6.1]{rubinstark} or \cite[Prop. 3.5]{sano}).

By applying Corollary \ref{main cor} in this setting, we can now prove the following result (the context of which is explained in Remark \ref{promise} below).

\begin{theorem} \label{RS theorem} Assume that
\begin{itemize}
\item[(a)] $\chi$ is neither trivial nor equal to the Teichm\"uller character, that
\item[(b)] no place in $S_p(K)$ splits completely in $L/K$, and that
\item[(c)] either $p>3$ or $\chi^2$ is not equal to the Teichm\"uller character.
\end{itemize}

For each non-negative integer $i$ define an ideal of $\mathcal{O}$ by setting
\[ I_i(T_\chi) := \langle I_i(\kappa(c)) \mid c \in {\rm ES}_r(T_\chi,\cK)\rangle_\mathcal{O},\]
where $\kappa(c) \in {\rm KS}_r(T_\chi,\cF)$ is the Kolyvagin system constructed from $c$. Then the following claims are valid.

\begin{itemize}
\item[(i)] For each $i$ one has $I_i(T_\chi)\subseteq {\rm Fitt}_{\mathcal{O}}^i((\ZZ_p\otimes_\ZZ {\rm Cl}(\mathcal{O}_L))^\chi)$.

\item[(ii)] If no place in $S\setminus S_\infty(K)$ splits completely in $L/K$, then the inclusions in claim (i) are equalities and there is an isomorphism of $\mathcal{O}$-modules
\[ (\ZZ_p\otimes_\ZZ {\rm Cl}(\mathcal{O}_L))^\chi \simeq \bigoplus_{i \ge 0} I_{i+1}(T_\chi)/I_i(T_\chi).\]
\item[(iii)] If the Rubin-Stark Conjecture is valid for $LF/K$ for each $F$ in $\Omega(\cK/K)$, then one has
 $\im(\eta_{L/K,S}^\chi) \subseteq {\rm Fitt}_{\cO}^0((\ZZ_p\otimes_\ZZ {\rm Cl}(\mathcal{O}_L))^\chi).$
\end{itemize}
\end{theorem}

\begin{proof} At the outset, note that assumption (c) is used simply so that we can include the case $p=3$ (see  Remark~\ref{remark p=3}).

We next show that the conditions of Corollary \ref{main cor} are satisfied under the hypotheses given above.

Firstly, under hypothesis (a) above, the $\mathcal{O}$-module $H^1(\cO_{F,S(F)_p},T_\chi) = (\ZZ_p \otimes_\ZZ \cO_{LF,S(F)_p}^\times)^\chi$ is easily seen to be free for every $F$ in $\Omega(\cK/K)$ and so Hypotheses~\ref{hyp free}(i) is satisfied.
We also obviously have $H^0(F,T_\chi)=0$ and so Hypothesis~\ref{hyp free}(ii) is satisfied.

The field $\cK$ is chosen so that Hypothesis~\ref{hyp K} is satisfied and Hypothesis~\ref{hyp local} is also trivial in this case.

Hypotheses~\ref{hyp1'}(i) and (ii) are trivial since $T_\chi$ is of rank one over $\cO$ (and so one can take $\tau=1$).
Hypothesis~\ref{hyp1'}(iii) also follows from the stated assumption (a): in fact this is checked in \cite[Lem. 3.1.1]{R} (see also \cite[Lem. 6.1.5]{MRkoly}).

By  \cite[Cor. 4.1.9(iii)]{MRkoly} and \cite[Cor. 3.5(ii)]{MRselmer},  for all positive integers $m$, Hypothesis~\ref{hyp large} is satisfied for $(T_\chi/p^{m}T_\chi, \cF_{\rm can}, \cP_{m})$ since in this case the coefficient ring is $\cO$.
(Note that in this case, the hypothesis~(b) given above implies that the `core rank' is equal to $r=|S_\infty(K)|$.)

Finally we consider the conditions (a) and (b) in Corollary~\ref{remark surjective}. In this setting the validity of condition~(a) follows directly from the assumption that all places in $S_\infty(K)$ split completely in $L/K$.

In addition, if no place in $S \setminus S_\infty(K)$ splits completely in $L/K$, then the $\cO$-module $$H^{0}(E_{w}, (T_\chi/pT_\chi)^\vee(1)) = H^0(E_w, \cO/(p) \otimes \chi)$$ is easily seen to vanish for each non-archimedean prime $w$ of $E$ above $S \cup S_{p}(K)$. (Recall, from \S\ref{koly sect}, that $E$ is a fixed auxiliary field in $\Omega(\cK/K)$ that contains $K(1)$ and is such that $E/K$ is unramified outside $S\cup S_p(K)$.  For example, one could take $E:=K(1)$.)


We have now verified that the hypotheses of Corollary~\ref{main cor}(i) and (ii) are satisfied under the given conditions~(a), (b) and (c), and that the larger set of hypotheses of Corollary~\ref{main cor}(iii) is satisfied if, in addition, no place in $S\setminus S_\infty(K)$ splits completely in $L/K$.

Next we note that, since $\chi$ is both non-trivial and primitive, the natural projection map $(\ZZ_p\otimes_\ZZ{\rm Cl}(\mathcal{O}_L))^\chi \to (\ZZ_p\otimes_\ZZ{\rm Cl}(\mathcal{O}_L[1/p]))^\chi$ is bijective 
under the given condition ~(b).
Therefore the identification (\ref{selmer ident}) implies that $
H^1_{(\cF_{\rm can})^\ast}(K,T_\chi^\vee(1))^\vee = \left( \ZZ_p\otimes_\ZZ{\rm Cl}(\mathcal{O}_L) \right)^\chi.
$


Given these observations, all of the stated claims follow directly from the result of
Corollary~\ref{main cor} and the fact that $c_{K,\chi}^{\rm RS} =\eta_{L/K, S}^\chi$.
%
\end{proof}


\begin{remark} If $K$ is totally real, then the assumption that places in $S_\infty(K)$ split completely in $L$ implies $\chi$ is not the Teichm\"uller character and so the condition in
Theorem~\ref{RS theorem}(a) reduces to requiring $\chi$ is not trivial. In regard to Theorem \ref{RS theorem}(ii), note that if $S = S_\infty(K)\cup S_{\rm ram}(L/K)$ (which is permissable if  $L/K$ is ramified), then no place in $S\setminus S_\infty(K)$ splits completely in $L/K$. In all cases, it is straightforward to choose a set $S$ that contains $S_\infty(K)\cup  S_{\rm ram}(L/K)$ and is such that no place in $S\setminus S_\infty(K)$ splits completely in $L/K$.
\end{remark}

\begin{remark}\label{promise} Theorem \ref{RS theorem} both refines and extends the results of Rubin in \cite{rubincrelle} and \cite{rubinstark} and, more recently, of B\"uy\"ukboduk in \cite{Buyuk}. (For example, the main result of the latter article deals only with the case $i=0$ and assumes, amongst other things, that $K$ is totally real, $L/K$ is unramified at $p$ and, crucially, that Leopoldt's conjecture is valid.) In addition, under certain mild additional hypotheses, a more careful application of the methods used to prove Theorem \ref{RS theorem} allows one to prove that for a wide range of abelian extensions $L$ of $K$ the higher Fitting ideals of $\ZZ_p\otimes_\ZZ {\rm Cl}(\mathcal{O}_L)$ as a $\ZZ_p[\Gal(L/K)]$-module are determined by Euler systems of rank $r$ for induced forms of the representation $\ZZ_p(1)$. For brevity, however, we defer further discussion of this result to a subsequent article.
\end{remark}

\begin{remark} The approach used to prove Theorem \ref{RS theorem} also leads to analogous results for the twisted representations $T_\chi(a) := T_\chi\otimes_{\ZZ_p}\ZZ_p(a)$ for arbitrary integers $a$. Taken in conjunction with the known validity of the Quillen-Lichtenbaum conjecture, this in turn leads to concrete new information about the Galois structure of even dimensional higher algebraic $K$-groups.

To be a little more precise, in this setting the Rubin-Stark Euler system $c^{\rm RS}_\chi$ defined above can be replaced, modulo the generalized Rubin-Stark conjecture formulated by Kurihara and the first and third authors in \cite[Conj. 3.5(i)]{bks2-2}, by an Euler system that is constructed in just the same way after replacing Rubin-Stark elements by the `generalized Stark elements' $\eta_{L/K,S}(-a)$ that are introduced in loc. cit.

The same argument as in the proof of Theorem \ref{RS theorem} then shows that for any integer $a$, and all suitable characters $\chi$ of $\Gal(L/K)$, the validity of \cite[Conj. 3.5(i)]{bks2-2} implies that ideals of the form
$\im(\eta_{L/K,S}(-a)^\chi)$ are contained in ${\rm Fitt}^0_\mathcal{O}(H^2(\mathcal{O}_{L,S},\ZZ_p(a+1))^\chi)$, respectively in ${\rm Fitt}_{\cO}^0((\ZZ_p\otimes_\ZZ K_{2a}(\mathcal{O}_{L}))^\chi)$ if $a > 0$.

However, since no essentially new ideas are involved in this argument, we prefer not to give any more details here.
\end{remark}


\begin{thebibliography}{99999999}



\bibitem{bassgorenstein} H. Bass,
\newblock On the ubiquity of Gorenstein rings,
\newblock Math. Z. {\bf 82} (1963) 8-28.































\bibitem{bks1} D. Burns, M. Kurihara, T. Sano,
\newblock On zeta elements for $\mathbb{G}_m$,
\newblock Doc. Math. \textbf{21} (2016) 555-626.


\bibitem{bks2-2} D. Burns, M. Kurihara, T. Sano,
\newblock On Stark elements of arbitrary weight and their $p$-adic families I,
\newblock submitted for publication.


\bibitem{sbA} D. Burns, T. Sano,
\newblock On the theory of higher rank Euler, Kolyvagin and Stark systems,
\newblock submitted for publication, arXiv:1612.06187v1.

\bibitem{Buyuk} K. B\"uy\"ukboduk,
\newblock Kolyvagin systems of Stark units,
\newblock J. reine u. Angew. Math. {\bf 631} (2009) 85-107.

\bibitem{Buyuk2} K. B\"uy\"ukboduk,
\newblock On Euler systems of rank $r$ and their Kolyvagin systems,
\newblock Indiana Univ. Math. J. {\bf 59} (2010) 1277-1332.



























\bibitem{GK} C. Greither, R. Ku\v cera,
\newblock Eigenspaces of ideal class groups,
\newblock Ann. Inst. Fourier \textbf{65} (2014) 2165-2203.




\bibitem{GreitherPopescu} C. Greither, C. Popescu,
\newblock The Galois module structure of $\ell$-adic realizations of Picard $1$-motives and applications,
\newblock Int. Math. Res. Notices, Volume 2012, 986-1036.





























\bibitem{kolyvagin} V. A. Kolyvagin,
\newblock Euler systems,
\newblock The Grothendieck Festschrift Vol II (1990) 435--483.



\bibitem{kuri} M. Kurihara,
\newblock Refined Iwasawa theory for $p$-adic representations and the structure of Selmer groups,
\newblock Muenster J. Math. {\bf 7} (2014) 149--223.




\bibitem{MRkoly} B. Mazur, K. Rubin,
\newblock Kolyvagin systems,
\newblock Mem. Amer. Math. Soc. \textbf{799} (2004).

\bibitem{MR} B. Mazur, K. Rubin,
\newblock Refined class number formulas and Kolyvagin systems,
\newblock Compos. Math. $\mathbf{147}$ (2011) 56--74.

\bibitem{MRselmer} B. Mazur, K. Rubin,
\newblock Controlling Selmer groups in the higher core rank case,
\newblock J. Th. Nombres Bordeaux {\bf 28} (2016) 145--183.








\bibitem{NSW} J. Neukirch, A. Schmidt, K. Wingberg,
\newblock Cohomology of number fields,
\newblock Springer Verlag, 2000.

\bibitem{north} D. G. Northcott,
\newblock Finite free resolutions,
\newblock Cambridge Univ. Press, Cambridge New York 1976.







\bibitem{PR} B. Perrin-Riou,
\newblock Syst\`emes d'Euler $p$-adiques et th\'eorie d'Iwasawa,
\newblock Ann. Inst. Fourier (Grenoble) {\bf 48} (1998) 1231-1307.






\bibitem{rubincrelle} K. Rubin,
\newblock Stark units and Kolyvagin's `Euler systems',
\newblock J. reine Angew. Math. {\bf 425} (1992) 141-154.

\bibitem{rubinstark} K. Rubin,
\newblock A Stark Conjecture `over $\bz$' for abelian $L$-functions with multiple zeros,
\newblock Ann. Inst. Fourier \textbf{46} (1996) 33-62.

\bibitem{R} K. Rubin,
\newblock Euler systems,
\newblock Annals of Math. Studies \textbf{147}, Princeton Univ. Press, 2000.

\bibitem{sakamoto} R. Sakamoto,
\newblock Stark systems over complete regular local rings,
\newblock submitted for publication.

\bibitem{sano} T. Sano,
\newblock Refined abelian Stark conjectures and the equivariant leading term conjecture of Burns,
\newblock Compositio Math. {\bf 150} (2014) 1809-1835.

\bibitem{sanojnt} T. Sano,
\newblock A generalization of Darmon's conjecture for Euler systems for general $p$-adic representations,
\newblock J. Number Theory {\bf 144} (2014) 281-324.






















\bibitem{tatebook} J. Tate,
\newblock Les Conjectures de Stark sur les Fonctions $L$ d'Artin en $s=0$ (notes par D. Bernardi et N. Schappacher),
\newblock Progress in Math., \textbf{47}, Birkh\"auser, Boston, 1984.











\end{thebibliography}
\end{document}